\documentclass[a4paper, 12pt]{article}

\usepackage[top = 2.5cm, bottom = 2.5cm, left = 2.4cm, right = 2.4cm]{geometry}
\usepackage{amsfonts, amsmath, amsthm, amssymb, mathtools, cases, bbm, array}
\usepackage{lmodern, indentfirst, setspace, changepage, fancyhdr}
\usepackage{tikz, pgfplots, float, subcaption}
\usepackage[title]{appendix}
\usepackage[english]{babel}
\usepackage{cite}

\usepackage{hyperref}
\hypersetup{
	colorlinks = true,
	linkcolor = blue,
	citecolor = blue,
}

\newtheorem{theorem}{Theorem}
\newtheorem{proposition}{Proposition}
\newtheorem{corollary}{Corollary}
\newtheorem{lemma}{Lemma}
\theoremstyle{definition}
\newtheorem{definition}{Definition}

\newtheorem{remark}{Remark}

\newcommand{\ind}[1]{\mathbbm{1}_{\left\{#1\right\}}}
\newcommand{\norm}[1]{\left|\left|#1\right|\right|}
\newcommand{\floor}[1]{\left\lfloor#1\right\rfloor}
\newcommand{\ceil}[1]{\left\lceil#1\right\rceil}
\newcommand{\map}[3]{#1 : #2 \longrightarrow #3}
\newcommand{\set}[2]{\left\{#1 : #2\right\}}

\newcommand{\mayorigual}{\trianglerighteq}
\newcommand{\menorigual}{\trianglelefteq}
\newcommand{\avg}[1]{\langle#1\rangle}
\newcommand{\mayor}{\vartriangleright}
\newcommand{\menor}{\vartriangleleft}
\newcommand{\boldf}{\boldsymbol{f}}
\newcommand{\defeq}{\vcentcolon=}
\newcommand{\eqdef}{=\vcentcolon}
\newcommand{\ba}{\boldsymbol{a}}
\newcommand{\bd}{\boldsymbol{d}}
\newcommand{\bg}{\boldsymbol{g}}
\newcommand{\bh}{\boldsymbol{h}}
\newcommand{\bq}{\boldsymbol{q}}
\newcommand{\br}{\boldsymbol{r}}
\newcommand{\bs}{\boldsymbol{s}}
\newcommand{\bv}{\boldsymbol{v}}
\newcommand{\bw}{\boldsymbol{w}}
\newcommand{\bx}{\boldsymbol{x}}
\newcommand{\by}{\boldsymbol{y}}
\newcommand{\bz}{\boldsymbol{z}}
\newcommand{\bE}{E}
\newcommand{\bP}{P}
\newcommand{\bX}{\boldsymbol{X}}
\newcommand{\bY}{\boldsymbol{Y}}
\newcommand{\bound}{\mathrm{bd}}
\newcommand{\prob}{\mathbbm{P}}
\newcommand{\calA}{\mathcal{A}}

\newcommand{\calD}{\mathcal{D}}

\newcommand{\calF}{\mathcal{F}}
\newcommand{\calG}{\mathcal{G}}
\newcommand{\calI}{\mathcal{I}}

\newcommand{\calK}{\mathcal{K}}

\newcommand{\calN}{\mathcal{N}}

\newcommand{\calR}{\mathcal{R}}

\newcommand{\indc}{\mathbbm{1}}
\newcommand{\scdot}{{}\cdot{}}
\newcommand{\eq}{\textrm{eq}}

\newcommand{\N}{\mathbbm{N}}

\newcommand{\R}{\mathbbm{R}}

\newcommand{\e}{\mathrm{e}}

\usetikzlibrary{automata, arrows, positioning, calc, external, babel}
\usepgfplotslibrary{fillbetween}
\pgfplotsset{
	compat = 1.16,
	every axis/.append style = {
		grid style = {dashed, gray, opacity = 0.2},
		label style = {font = \footnotesize},
		tick label style = {font = \footnotesize},  
		width = 1 * \columnwidth,
		height = 0.618 * 1 * \columnwidth
	}
}

\definecolor{britishracinggreen}{rgb}{0.0, 0.26, 0.15}
\definecolor{bostonuniversityred}{rgb}{0.8, 0.0, 0.0}
\definecolor{ceruleanblue}{rgb}{0.16, 0.32, 0.75}
\definecolor{airforceblue}{rgb}{0.36, 0.54, 0.66}
\definecolor{cadmiumgreen}{rgb}{0.0, 0.42, 0.24}
\definecolor{ao(english)}{rgb}{0.0, 0.5, 0.0}
\definecolor{coolblack}{rgb}{0.0, 0.18, 0.39}
\definecolor{alizarin}{rgb}{0.82, 0.1, 0.26}
\definecolor{arsenic}{rgb}{0.23, 0.27, 0.29}
\definecolor{cobalt}{rgb}{0.0, 0.28, 0.67}
\definecolor{amber}{rgb}{1.0, 0.75, 0.0}

\title{Utility maximizing load balancing policies \vspace{\baselineskip}}

\author{
\normalsize{Diego Goldsztajn, Sem C. Borst}\\ \footnotesize{Eindhoven University of Technology, d.e.goldsztajn@tue.nl, s.c.borst@tue.nl} \\
\normalsize{Johan S.H. van Leeuwaarden}\\ \footnotesize{Tilburg University, j.s.h.vanleeuwaarden@uvt.nl} \\
}

\date{\vspace{\baselineskip} February 10, 2024}

\begin{document}

%% Title	
\maketitle

\noindent\rule{\textwidth}{1pt}

\vspace{2\baselineskip}

\onehalfspacing

\begin{adjustwidth}{0.8cm}{0.8cm}
	\begin{center}
		\textbf{Abstract}
	\end{center}
	
	\vspace{0.3\baselineskip}
	
	\noindent Consider a service system where incoming tasks are instantaneously dispatched to one out of many heterogeneous server pools. Associated with each server pool is a concave utility function which depends on the class of the server pool and its current occupancy. We derive an upper bound for the mean normalized aggregate utility in stationarity and introduce two load balancing policies that achieve this upper bound in a large-scale regime. Furthermore, the transient and stationary behavior of these asymptotically optimal load balancing policies is characterized on the scale of the number of server pools, in the same large-scale regime.
	
	\vspace{\baselineskip}
	
	
	\vspace{0.3\baselineskip}
	
	\small{\noindent \textit{Acknowledgment:} the work in this paper is supported by the Netherlands Organisation for Scientific Research (NWO) through Gravitation-grant NETWORKS-024.002.003 and Vici grant 202.068.} 
\end{adjustwidth}

\newpage

\section{Introduction}
\label{sec: introduction}

We consider a service system where incoming tasks are instantaneously assigned to one out of many heterogeneous server pools. All the tasks sharing a server pool are executed in parallel and the execution times do not depend on the class of the server pool or the number of tasks currently contending for service. Nevertheless, associated with each server pool is a not necessarily increasing concave utility function which does depend on the class of the server pool and the number of tasks currently sharing it. These features are characteristic of streaming and online gaming services, where the duration of tasks is mainly determined by the application, but still congestion can have a strong impact on the experienced performance (e.g., video resolution and fluency).

The goal is to assign tasks so as to maximize the overall utility of the system, defined as the aggregate utility of all the server pools normalized by the number of server pools. We derive an upper bound for its stationary mean through an optimization problem where the optimization variable is a sequence that describes the distribution of a fractional number of tasks across the server pools; the objective of the problem is the overall utility function, and the main constraint is that the total number of tasks must be equal to the offered load of the system. We construct an optimal (fractional) task assignment that solves this problem and has a particularly insightful structure, and we formulate the upper bound for the mean stationary overall utility in terms of this solution.

Armed with the above insight, we propose and analyze two assignment policies that maintain the occupancy state of the system aligned with an optimal task assignment. Specifically, we examine a policy that assigns every new task to a server pool with the largest marginal utility; this policy is dubbed Join the Largest Marginal Utility (JLMU). We also introduce a multi-threshold policy that follows the same greedy principle but only approximately, and uses significantly less state information. The optimal threshold values depend on the typically unknown offered load of the system and are adjusted over time through an inbuilt learning scheme; thus we name this policy Self-Learning Threshold Assignment (SLTA). Assuming exponential service times, we characterize the asymptotic transient and stationary behavior of both policies on the scale of the number of server pools, and we prove that both policies achieve the upper bound for the mean stationary overall utility as the number of server pools grows large.

A fundamental difference between JLMU and SLTA is that the former is naturally agnostic to the offered load, whereas for the latter, the optimal thresholds depend on the offered load. However, we show that the online learning scheme of SLTA is capable of finding the optimal threshold values without any prior knowledge of the offered load, which makes it possible to deploy SLTA if the offered load is not known in advance.

\subsection{Main contributions}
\label{sub: main contributions}

The main contribution of this paper is an upper bound for the mean stationary overall utility that is asymptotically tight for exponentially distributed service times, and thereby serves as a crucial performance benchmark. The asymptotic tightness of the upper bound is proved by studying the stationary behavior of JLMU and SLTA in the regime where the number of server pools grows large, and by establishing that both assignment policies achieve the upper bound in the latter regime.

The analysis of JLMU is based on a fluid limit given by an infinite system of differential equations with a discontinuous right-hand side. We prove that the associated initial value problem always has a unique solution, by making a connection with a system of integral equations, expressed in terms of Skorokhod one-dimensional reflection mappings, and using a uniqueness result for certain Kolmogorov backward equations. Moreover, we show that the fluid limit holds with respect to an $\ell_1$ norm, and that the system of differential equations is globally asymptotically stable with respect to this norm. These results are used to prove that the stationary distribution of the process that describes the occupancy state of the system converges in $\ell_1$ to an optimal task assignment for the offered load of the system. The asymptotic optimality of JLMU is then established by proving that the stationary overall utilities form a convergent and uniformly integrable sequence of random variables; the proof of the latter properties exploits a representation of the overall utility as a linear functional on $\ell_1$ and our convergence results with respect to the $\ell_1$ norm.

While SLTA is simple to implement, its analysis is inherently challenging due to the complex interdependence between two components of the policy. Namely, the dispatching rule, which depends on the multiple thresholds, and the online learning scheme, which adjusts the thresholds over time. Furthermore, an additional technical difficulty is that the learning scheme is triggered by excursions of the occupancy state of the system that asymptotically vanish on the scale of the number of server pools. 

In order to analyze the large-scale transient behavior of SLTA, we use a methodology of \cite{goldsztajn2021learning} which allows to overcome the aforementioned challenges by means of a non-traditional fluid limit analysis. In the present paper, we extend the latter methodology to also prove weak convergence of the stationary distribution of the occupancy process and thresholds. Here our contributions are proofs of ergodicity and tightness of stationary distributions through a careful drift analysis, as well as a suitably adapted interchange of limits argument designed to leverage the large-scale transient result obtained with the methodology of \cite{goldsztajn2021learning}. Equipped with the weak convergence results for the stationary distributions, we prove the asymptotic optimality of SLTA in a similar way as for JLMU, by showing that all of our limit theorems hold with respect to the $\ell_1$ norm and exploiting the linear representation of the overall utility function.

\subsection{Related work}
\label{sub: related work}

Load balancing and task assignment in parallel-server systems has received immense attention in the past decades; some relevant papers are \cite{mitzenmacher2001power,vvedenskaya1996queueing,lu2011join,stolyar2015pull,winston1977optimality,eschenfeldt2018join}. While traditionally the focus used to be on performance, more recently the implementation overhead has emerged as an equally important issue. In large-scale deployments, this overhead has two main sources: the communication burden of messaging between the dispatcher and the servers, and the operational cost of storing and managing state information at the dispatcher \cite{gamarnik2018delay,gamarnik2020lower}. We refer to \cite{van2018scalable} for an extensive survey on scalable load balancing.

While the load balancing literature has been predominantly concerned with systems of parallel single-server queues, the present paper considers an infinite-server setting where the service times of tasks do not depend on the number of competing tasks. This feature is characteristic of streaming applications, where the level of congestion does not significantly affect the duration of tasks. The level of congestion has, however, a strong impact on the amount of resources received by individual streaming sessions, and thus on the experienced quality-of-service, which can be modeled through utility functions. Infinite-server dynamics have been commonly adopted as a natural paradigm for modeling streaming sessions on flow-level \cite{benameur2002quality,key2004fair} and the problem of managing large data centers serving streaming sessions has been recently addressed in \cite{mukherjee2020asymptotic}. Systems with infinite-server dynamics have also been analyzed in \cite{karthik2017choosing,mukhopadhyay2015mean,mukhopadhyay2015power,xie2015power}, which concern loss models that are different in nature from the setting considered in the present paper.

When the server pools are homogeneous, the overall utility is a Schur-concave function of the vector describing the number of tasks at each server pool. In this case, maximizing the aggregate utility of the system boils down to equalizing the number of tasks across the various server pools. Join the Shortest Queue (JSQ) maximizes the mean stationary overall utility of the system for exponential service times, and in fact has stronger stochastic optimality properties \cite{menich1991optimality,sparaggis1993extremal}. In the homogeneous setting, JLMU reduces to JSQ and is thus optimal for exponential service times. Also, SLTA reduces to the policy considered in \cite{goldsztajn2021self,goldsztajn2021learning}, which asymptotically matches the performance of JSQ on the fluid and diffusion scales for exponentially distributed service times. While the policy considered in \cite{zhou2017designing,zhou2018heavy,horvath2019mean} is similar to SLTA in name, this policy does not equalize the queue lengths.

\begin{figure}
	\centering
	\begin{tikzpicture}[x = 2cm, y = 0.75cm]
	%% Headers
	\node[align = center] at (0, -2) {\footnotesize{\textbf{Single-server}}\\ \footnotesize{\textbf{dynamics}}};
	\node[align = center] at (0, -4) {\footnotesize{\textbf{Infinite-server}}\\ \footnotesize{\textbf{dynamics}}};
	\node[align = center] at (2, 0) {\footnotesize{\textbf{Homogenenous}}\\ \footnotesize{\textbf{setting}}};
	\node[align = center] at (4, 0) {\footnotesize{\textbf{Heterogeneous}}\\ \footnotesize{\textbf{setting}}};
	
	%% Contents
	\node at (2, -2) {\footnotesize{\cite{van2018scalable,mitzenmacher2001power,vvedenskaya1996queueing,lu2011join,stolyar2015pull,winston1977optimality,eschenfeldt2018join,gamarnik2018delay}}};
	\node at (2, - 4) {\footnotesize{\cite{goldsztajn2021self,goldsztajn2021learning,mukherjee2020asymptotic}}};
	\node at (4, -2) {\footnotesize{\cite{gardner2019smart,gardner2021scalable,bhambay2022asymptotic,van2021load,stolyar2015pull,mukhopadhyay2016randomized,abdul2022general}}};
	\node at (4, -4) {\footnotesize{Present paper}};
	
	%% Lines
	\draw (1, -1) -- (5, -1);
	\draw (1, -3) -- (5, -3);
	\draw (1, -5) -- (5, -5);
	\draw (1, -1) -- (1, -5);
	\draw (3, -1) -- (3, -5);
	\draw (5, -1) -- (5, -5);
	\end{tikzpicture}
	\caption{Schematic view of some of the related work. Most of the load balancing literature concerns systems of parallel and homogeneous single-server queues; this vast literature is surveyed in \cite{van2018scalable}. Some recent papers study single-server dynamics in heterogeneous settings or infinite-server dynamics in homogeneous settings, whereas the present paper considers a heterogeneous system with infinite-server dynamics.}
	\label{fig: related work}
\end{figure}

The problem of maximizing the overall utility of the system is more challenging if the server pools are heterogeneous as in this paper. Heterogeneity is the norm in data centers, where servers from different generations coexist because old machines are only gradually replaced by more powerful versions; as shown in Figure \ref{fig: related work}, this feature has been recently addressed in the load balancing literature for single-server models \cite{gardner2019smart,gardner2021scalable,abdul2022general} but not in the infinite-server context. When the server pools are heterogeneous, it is no longer optimal to maintain an evenly balanced distribution of the load, in fact it is not even obvious at all how tasks should be distributed in order to maximize the overall utility function, and the optimal distribution of tasks across the server pools depends on this function. Another striking difference with the homogeneous setting is that JLMU is generally not optimal in the pre-limit for exponentially distributed service times; we establish that, in general, the optimality is only achieved asymptotically in the heterogeneous case.

From a theoretical perspective, one of the most interesting features of SLTA is its capacity to learn the offered load of the system. The problem of adaptation to unknown demands was previously addressed in \cite{goldsztajn2018controlling,goldsztajn2021automatic,mukherjee2017optimal} in the context of single-server models, by assuming that the number of servers can be right-sized on the fly to match the load of the system. However, in the latter papers the dispatching rule remains the same at all times since the right-sizing mechanism alone is sufficient to maintain small queues, by adjusting the number of servers. Different from these right-sizing mechanisms, the learning scheme of SLTA modifies the parameters of the dispatching rule over time to maximize the overall utility of the system. 

\subsection{Outline of the paper}
\label{sub: outline of the paper}

In Section \ref{sec: problem formulation} we introduce some of the notation used throughout the paper and we formulate the upper bound for the mean stationary overall utility. In Section \ref{sec: load balancing policies} we specify the JLMU and SLTA policies, and we state their asymptotic optimality with respect to the mean stationary overall utility. In Section \ref{sec: approximation theorems} we present several results that pertain to the asymptotic transient behavior of these two policies, and that are used to establish their asymptotic optimality. In Section \ref{sec: performance upper bound} we prove the upper bound for the mean stationary overall utility. In order to characterize the asymptotic behavior of JLMU and SLTA, we construct systems of different sizes on a common probability space in Section \ref{sec: strong approximations}, where we also prove relative compactness results. Limit theorems for JLMU and SLTA are proved in Sections \ref{sec: limiting behavior of jlmu} and \ref{sec: limiting behavior of threshold}, respectively, and the asymptotic optimality of these two policies is established in Section \ref{sec: asymptotic optimality}. Some proofs are deferred to Appendices \ref{app: auxiliary results}, \ref{app: relative compactness} and \ref{app: limiting behavior of threshold}.

\section{Problem formulation}
\label{sec: problem formulation}

In this section we define some of the notation used throughout the paper and we state the upper bound for the mean stationary overall utility. In Section \ref{sub: basic notation} we introduce two descriptors for specifying the state of the system and we define the overall utility function. In Section \ref{sub: optimization problem} we present the optimization problem used to derive the upper bound for the mean stationary overall utility. In Section \ref{sub: structure of an optimal solution} we construct a solution of this problem explicitly, and in Section \ref{sub: performance upper bound} we use the constructed solution to formulate the upper bound for the mean stationary overall utility.

\subsection{Basic notation}
\label{sub: basic notation}

Consider a system with $m$ classes of server pools. All the tasks sharing a server pool are executed in parallel and the execution times do not depend on the class of the server pool or the number of tasks currently contending for service. Nevertheless, associated with each server pool is a concave utility function which does depend on the class of the server pool and the number of tasks sharing it. For example, these functions can be used to model the overall quality-of-service provided to streaming tasks sharing an underlying resource with a fixed capacity. The objective is to assign the incoming tasks to the various server pools so as to maximize the aggregate utility of all the server pools in stationarity.

The number of server pools is denoted by $n$ and the number and fraction of server pools of class $i$ are denoted by $A_n(i)$ and $\alpha_n(i) = A_n(i) / n$, respectively. We assume that tasks arrive as a Poisson process of intensity $n \lambda$ with independent and identically distributed service times of mean $1 / \mu$, and we define $\bX_n(i, k)$ as the number of tasks in server pool $k$ of class $i$; boldface symbols are used in the paper to indicate time-dependence. Server pools of the same class that have the same number of tasks are exchangeable, thus we usually consider a different state descriptor. Specifically, we let
\begin{equation*}
\bq_n(i, j) \defeq \frac{1}{n} \sum_{k = 1}^{A_n(i)} \ind{\bX_n(i, k) \geq j}
\end{equation*}
denote the fraction of server pools which are of class $i$ and have at least $j$ tasks. The values of $\bX_n$ and $\bq_n$ at a given time are referred as the occupancy state or task assignment.

The concave utility function associated with server pools of class $i$ is denoted by $u_i$ and the overall utility of the system is defined as the aggregate utility of all the server pools normalized by the number of server pools. More precisely, we let
\begin{equation*}
u_n\left(\bX_n\right) \defeq \frac{1}{n} \sum_{i = 1}^m \sum_{k = 1}^{A_n(i)} u_i\left(\bX_n(i, k)\right).
\end{equation*}
Note that $\bq_n(i, j) - \bq_n(i, j + 1)$ is the fraction of server pools of class $i$ with $j$ tasks. Thus, the overall utility may equivalently be expressed as
\begin{equation*}
u\left(\bq_n\right) \defeq \sum_{i = 1}^m \sum_{j = 0}^\infty u_i(j) \left[\bq_n(i, j) - \bq_n(i, j + 1)\right].
\end{equation*}
While the overall utility function is generally not linear as a function of $\bX_n$, it is always linear as a function of $\bq_n$, as shown by the above expression.

The total number of tasks in the system, normalized by the number of server pools, can be expressed in terms of the occupancy state $\bq_n$ as follows:
\begin{equation}
\label{eq: definition of s_n}
\bs_n \defeq \sum_{i = 1}^m \sum_{j = 1}^\infty \bq_n(i, j) = \sum_{i = 1}^m \sum_{j = 1}^\infty j \left[\bq_n(i, j) - \bq_n(i, j + 1)\right].
\end{equation}
The quantity $j \left[\bq_n(i, j) - \bq_n(i, j + 1)\right]$ represents the number of tasks in server pools of class $i$ with exactly $j$ tasks, normalized by the total number of server pools. Hence, $\bs_n$ indeed corresponds to the normalized total number of tasks.

Throughout the paper, we write $\bP$ and $\bE$ to denote the probability and expectation with respect to a given probability measure. If $\rho \defeq \lambda / \mu$ denotes the normalized offered load, then the stationary distribution of the total number of tasks is Poisson with mean $n \rho$ due to the infinite-server dynamics of the system. Thus, $\bE\left[\bs_n\right] = \rho$ in stationarity, for any task assignment policy.

\subsection{Optimization problem}
\label{sub: optimization problem}

Based on the above, we now formulate an optimization problem which yields an upper bound for the mean stationary overall utility:
\begin{equation}
	\label{pr: optimal task assignments}
	\begin{split}
		\underset{q}{\text{maximize}} &\quad u(q) \\
		\text{subject to} &\quad \sum_{i = 1}^m \sum_{j = 1}^\infty q(i, j) = \rho, \\
		&\quad 0 \leq q(i, j + 1) \leq q(i, j) \leq q(i, 0) = \alpha_n(i) \quad \text{for all} \quad i, j.
	\end{split}
\end{equation}

In order to see that the optimum of \eqref{pr: optimal task assignments} yields an upper bound for the mean stationary overall utility, consider any policy such that $\bq_n$ has a stationary distribution. We assume that the policy is such that the evolution of the system over time can be described by a Markov process with a possibly uncountable state space that has a stationary distribution. Let $q_n$ be a random variable with the stationary distribution of $\bq_n$ and define $\bE\left[q_n\right]$ as the sequence whose $(i, j)$ element is $\bE\left[q_n(i, j)\right]$. Observe that
\begin{equation}
\label{eq: interchange of u and e}
\bE\left[u\left(q_n\right)\right] = u\left(\bE\left[q_n\right]\right).
\end{equation}
Indeed, the utility functions are concave, so for each $i$ there exists $j_i \in \N$ such that $u_i(j)$ and $u_i(k)$ have the same sign if $j, k > j_i$. Therefore, \eqref{eq: interchange of u and e} follows from Tonelli's theorem. In addition, $\bE\left[q_n\right]$ satisfies the constraints of \eqref{pr: optimal task assignments} because the total number of tasks in stationarity has mean $n \rho$. Thus, $\bE\left[u(q_n)\right]$ is upper bounded by the optimum of \eqref{pr: optimal task assignments}. 

\subsection{Structure of an optimal solution}
\label{sub: structure of an optimal solution}

For brevity, we refer to an optimizer of \eqref{pr: optimal task assignments} as an optimal task assignment; the term optimal fractional task assignment would be more appropriate since the offered load $n \rho$ may not be integral. In this section we define a ranking of the server pools that can be used to construct an optimal task assignment. For this purpose, consider the sets
\begin{equation*}
\calI \defeq \{1, \dots, m\} \times \N \quad \text{and} \quad \calI_+ \defeq \set{(i, j) \in \calI}{j \geq 1}.
\end{equation*}
A server pool has coordinates $(i, j) \in \calI_+$ if its class is $i$ and it has precisely $j - 1$ tasks; e.g., in Figure \ref{fig: marginal utilities}, server pool A of class $1$ has coordinates $(1, 4)$ and both server pools of class~$3$ have coordinates $(3, 1)$. Since server pools with the same coordinates are statistically identical, we may focus on ranking coordinates rather than server pools.

\begin{figure}
	\centering
	\begin{tikzpicture}[x = 1.6cm, y = 0.9cm]
	%% Class boundaries
	\draw[black, dashed] (-0.25, -1) -- (-0.25, 5);
	\draw[black, dashed] (2, -1) -- (2, 5);
	\draw[black, dashed] (5.25, -1) -- (5.25, 5);
	\draw[black, dashed] (7.5, -1) -- (7.5, 5);
	\draw[black, dashed] (-0.25, -1) -- (7.5, -1);
	
	%% Class 1
	
	\foreach \x in {0, 1}
	\foreach \y in {0, 1, 2, 3}
	\draw[black] (\x, \y) rectangle (\x + 0.75, \y + 0.75);
	
	\foreach \x in {0, 1}
	\node at (\x + 0.375, 0.375) {\tiny{$\Delta(1, 0)$}};
	
	\foreach \x in {0, 1}
	\node at (\x + 0.375, 1.375) {\tiny{$\Delta(1, 1)$}};
	
	\foreach \x in {0, 1}
	\node at (\x + 0.375, 2.375) {\tiny{$\Delta(1, 2)$}};
	
	\foreach \x in {0, 1}
	\node at (\x + 0.375, 3.375) {\tiny{$\Delta(1, 3)$}};
	
	\foreach \x in {0, 1}
	\node[rotate = 90] at (\x + 0.375, 4.375) {$\cdot\cdot\cdot$};
	
	\node at (0.375, -0.5) {\scriptsize{A}};
	\node at (1.375, -0.5) {\scriptsize{B}};
	
	\node at (0.875, -1.5) {\scriptsize{$1$}};
	
	%% Class 2
	\foreach \x in {2.25, 3.25, 4.25}
	\foreach \y in {0, 1, 2, 3}
	\draw[black] (\x, \y) rectangle (\x + 0.75, \y + 0.75);
	
	\foreach \x in {2.25, 3.25, 4.25}
	\node at (\x + 0.375, 0.375) {\tiny{$\Delta(2, 0)$}};
	
	\foreach \x in {2.25, 3.25, 4.25}
	\node at (\x + 0.375, 1.375) {\tiny{$\Delta(2, 1)$}};
	
	\foreach \x in {2.25, 3.25, 4.25}
	\node at (\x + 0.375, 2.375) {\tiny{$\Delta(2, 2)$}};
	
	\foreach \x in {2.25, 3.25, 4.25}
	\node at (\x + 0.375, 3.375) {\tiny{$\Delta(2, 3)$}};
	
	\foreach \x in {2.25, 3.25, 4.25}
	\node[rotate = 90] at (\x + 0.375, 4.375) {$\cdot\cdot\cdot$};
	
	\node at (2.625, -0.5) {\scriptsize{A}};
	\node at (3.625, -0.5) {\scriptsize{B}};
	\node at (4.625, -0.5) {\scriptsize{C}};
	
	\node at (3.625, -1.5) {\scriptsize{$2$}};
	
	%% Class 3
	\foreach \x in {5.5, 6.5}
	\foreach \y in {0, 1, 2, 3}
	\draw[black] (\x, \y) rectangle (\x + 0.75, \y + 0.75);
	
	\foreach \x in {5.5, 6.5}
	\node at (\x + 0.375, 0.375) {\tiny{$\Delta(3, 0)$}};
	
	\foreach \x in {5.5, 6.5}
	\node at (\x + 0.375, 1.375) {\tiny{$\Delta(3, 1)$}};
	
	\foreach \x in {5.5, 6.5}
	\node at (\x + 0.375, 2.375) {\tiny{$\Delta(3, 2)$}};
	
	\foreach \x in {5.5, 6.5}
	\node at (\x + 0.375, 3.375) {\tiny{$\Delta(3, 3)$}};
	
	\foreach \x in {5.5, 6.5}
	\node[rotate = 90] at (\x + 0.375, 4.375) {$\cdot\cdot\cdot$};
	
	\node at (5.875, -0.5) {\scriptsize{A}};
	\node at (6.875, -0.5) {\scriptsize{B}};
	
	\node at (6.375, -1.5) {\scriptsize{$3$}};
	
	%% Axis labels
	\node[rotate = 90] at (-0.65, 2) {\scriptsize{Number of servers $j$}};
	\node at (3.625, -2.25) {\scriptsize{Server pool class $i$}};
	
	%% Busy servers
	\foreach \y in {0, 1, 2}
	\draw[fill = black!20!white] (0, \y) rectangle (0.75, \y + 0.75);
	
	\foreach \y in {0, 1}
	\draw[fill = black!20!white] (1, \y) rectangle (1.75, \y + 0.75);
	
	\draw[fill = black!20!white] (3.25, 0) rectangle (4, 0.75);
	
	\foreach \y in {0, 1}
	\draw[fill = black!20!white] (4.25, \y) rectangle (5, \y + 0.75);
	\end{tikzpicture}
	\caption{Schematic representation of the marginal utilities. White rectangular slots and gray rectangles represent idle and busy servers, respectively. Each of the columns labeled with letters represents a server pool and the dashed lines enclose server pools of the same class. If the tasks sent to a given server pool are always placed in the first idle server from bottom to top, then the marginal utilities written on top of the idle servers indicate the increase in the aggregate utility of the system when the server receives a task.}
	\label{fig: marginal utilities}
\end{figure}
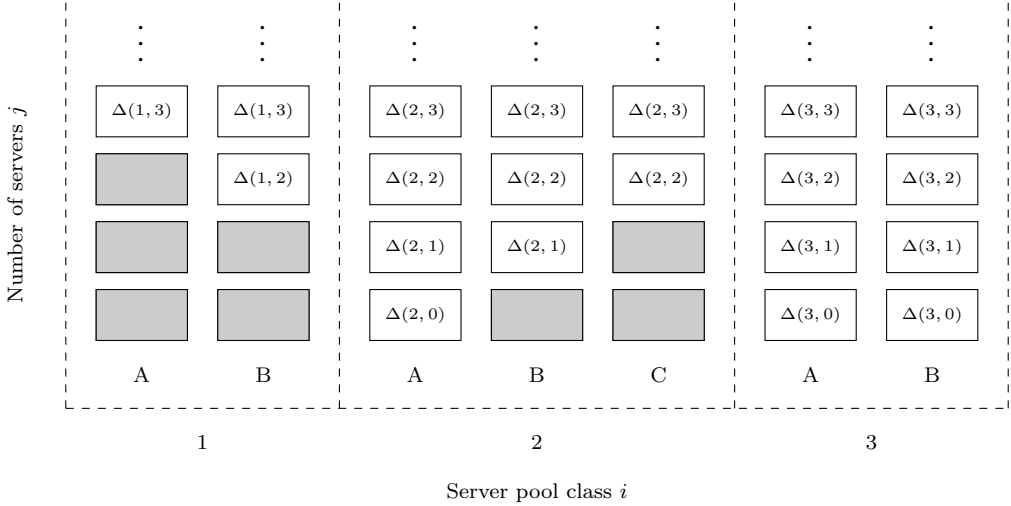

Formally, we define a total order on $\calI_+$ that gives precedence to coordinates associated with larger marginal utilities. The marginal utility of a server pool of class $i$ with $j$ tasks is denoted by $\Delta(i, j) \defeq u_i(j + 1) - u_i(j)$ and represents the change in the utility function of such a server pool if it receives an additional task. The marginal utility of the coordinates $(i, j)$ is just $\Delta(i, j - 1)$, the marginal utility of server pools of class $i$ with $j - 1$ tasks. 

Consider the dictionary order $\prec$ on $\calI$, defined by
\begin{equation*}
(i_1, j_1) \prec (i_2, j_2) \quad \text{if and only if} \quad i_1 < i_2 \quad \text{or} \quad i_1 = i_2 \quad \text{and} \quad j_1 < j_2.
\end{equation*}
We obtain a total order $\menor$ on $\calI_+$ by writing $(i_1, j_1) \menor (i_2, j_2)$ if and only if one of the following conditions holds:
\begin{itemize}
	\item $\Delta(i_1, j_1 - 1) < \Delta(i_2, j_2 - 1)$,
	
	\item $\Delta(i_1, j_1 - 1) = \Delta(i_2, j_2 - 1)$ and $(i_1, j_1 - 1) \succ (i_2, j_2 - 1)$.
\end{itemize}
In particular, the marginal utility of server pools with coordinates $(i_1, j_1)$ is smaller than or equal to that of server pools with coordinates $(i_2, j_2) \mayor (i_1, j_1)$. The dictionary order is used to break the tie when both coordinates are associated with the same marginal utility, but a different tie breaking rule could be used instead.

Consider the task assignment $q_n^*$ defined by
\begin{equation}
\label{eq: definition of q_n^N}
q_n^*(i, j) \defeq \begin{cases}
0												&	\text{if} \quad (i, j) \menor \sigma_n^*, \\
\alpha_n(i)										&	\text{if} \quad (i, j) \mayor \sigma_n^*, \\
\rho - \sum_{(i, j) \mayor \sigma_n^*} \alpha_n(i)	&	\text{if} \quad (i, j) = \sigma_n^*,
\end{cases}
\quad \text{for all} \quad (i, j) \in \calI_+.
\end{equation}
In Section \ref{sec: performance upper bound} we prove that $q_n^*$ constitutes an optimal task assignment if $\sigma_n^*$ is defined as the unique element of $\calI_+$ that satisfies
\begin{equation}
\label{eq: definition of sigma_n^N}
\sum_{(i, j) \mayor \sigma_n^*} \alpha_n(i) \leq \rho < \sum_{(i, j) \mayorigual \sigma_n^*} \alpha_n(i).
\end{equation}
For the uniqueness of $\sigma_n^*$, note that the number of terms in the summations on both sides of \eqref{eq: definition of sigma_n^N} increases as the ranking of $\sigma_n^*$ becomes worse, and that the summation on the right has exactly one more term than the summation on the left.

\subsubsection{Numerical examples}

Figure \ref{fig: optimal task asssignments} illustrates the optimal task assignments obtained through \eqref{eq: definition of q_n^N} for two sets of utility functions and different values of $\rho$. If $n$ is such that $n\alpha_n(i)$ and $n \rho$ are integers for all $i$, then the plots can be interpreted as sets of $n$ adjacent columns, where each column represents a server pool and the colored portion of a column indicates the number of tasks sharing the server pool, as in the diagram of Figure \ref{fig: marginal utilities}. The thick vertical lines separate the server pool classes and the quantities $q_n^*(i, j)$ can be read off by rotating the plots.

\begin{figure}
	\centering
	\begin{subfigure}{0.49\columnwidth}
		\centering
		\includegraphics{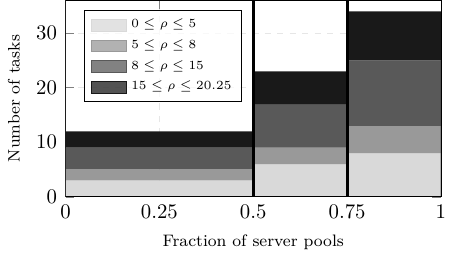}
	\end{subfigure}
	\hfill
	\begin{subfigure}{0.49\columnwidth}
		\centering
		\includegraphics{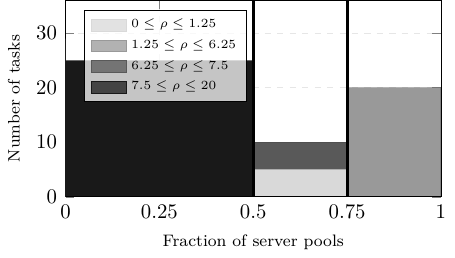}
	\end{subfigure}
	\caption{Distribution of the offered load across the various server pools under the optimal task assignment of \eqref{eq: definition of q_n^N} for $m = 3$ and $\alpha_n = (1/2, 1/4, 1/4)$. On the left, $u_i(x) = x \log(r(i) / x)$ with $r = (5, 10, 15)$. On the right, $u_1(x) = x$, $u_2(x) = 2x - x^2/20$, $u_3(x) = 3x/2$ if $x < 20$ and $u_3(x) = 30$ if $x \geq 20$.}
	\label{fig: optimal task asssignments}
\end{figure}

The left plot corresponds to utility functions of the form $u_i(x) = x g(r(i) / x)$, with $g$ a concave and increasing function. These utility functions can be used to model the aggregate quality-of-service provided to streaming tasks sharing a single server pool. The quantity $r(i)$ represents the total amount of resources in a server pool of class $i$ and $g(r(i) / x)$ models the quality-of-service provided to a single task when the server pool is shared by $x$ tasks and each task gets a fraction $r(i) / x$ of the total resources. For a given $\rho$, the left plot of Figure \ref{fig: optimal task asssignments} depicts server pools with roughly $r(i) / c(r, \rho)$ tasks, with $c(r, \rho)$ a normalizing constant that does not depend on $g$; for some values of $\rho$, all server pools have exactly $r(i) / c(r, \rho)$ tasks, but in other cases some of these numbers are rounded. This behavior is explained by noting that the derivative of $u_i(x)$ can be expressed as a function of $r(i) / x$, thus the occupancy levels $r(i) / c(r, \rho)$ equalize the marginal utilities.

In the left plot of Figure \ref{fig: optimal task asssignments}, the occupancy levels of the various server pools maintain approximately fixed ratios as $\rho$ increases. The right plot shows a completely different behavior: as $\rho$ increases from $0$ to $1.25$, only the occupancy of server pools of class $2$ grows, but from $1.25$ to $6.25$, only the occupancy of server pools of class $3$ grows, and eventually exceeds the occupancy of server pools of class $2$. Furthermore, as $\rho$ increases beyond $7.5$, only the occupancy of server pools of class $1$ increases.

\subsection{Performance upper bound}
\label{sub: performance upper bound}

We now state the upper bound for the mean stationary overall utility, which follows from the observations of Sections \ref{sub: optimization problem} and \ref{sub: structure of an optimal solution}; a rigorous proof is provided in Section \ref{sec: performance upper bound}.

\begin{theorem}
	\label{the: upper bound for mean stationary overall utility}
	Consider any task assignment policy such that the occupancy process $\bq_n$ has a stationary distribution, and let $q_n$ be a random variable distributed as this stationary distribution. Then
	\begin{equation*}
	\bE\left[u(q_n)\right] \leq u\left(q_n^*\right).
	\end{equation*}
\end{theorem}

In the following section we establish that the upper bound is asymptotically achievable when service times are exponentially distributed. In particular, we will see that JLMU achieves the upper bound of Theorem \ref{the: upper bound for mean stationary overall utility} as the number of server pools grows large; recall that JLMU is generally not optimal in the pre-limit, not even for exponential service times. Moreover, we will establish that SLTA also achieves the upper bound asymptotically, while relying on considerably less state information.

Before proceeding, it is illustrative to draw an analogy between the setting considered in this paper and the load balancing literature for systems of parallel single-server queues, where the primary objective is to minimize queueing delay. The natural policy for the setting considered in this paper is JLMU, while the natural policy for minimizing queueing delay in systems of parallel single-server queues is JSQ. The deployment of these policies involves a considerable communication overhead, or storing and managing a significant amount of state information. In the setting considered in this paper, SLTA provides a asymptotically optimal performance for exponential service times and uses substantially less state information than JLMU. From this perspective, SLTA is the counterpart of JIQ in the load balancing literature for systems of parallel single-server queues \cite{lu2011join,stolyar2015pull}.

\section{Load balancing policies}
\label{sec: load balancing policies}

In this section we describe the load balancing policies considered in the paper and we state their asymptotic optimality with respect to the mean stationary overall utility. In Sections \ref{sub: join the largest marginal utility (jlmu)} and \ref{sub: threshold} we specify JLMU and SLTA, respectively. In Section \ref{sub: stochastic models} we define stochastic models, based on continuous-time Markov chains, for the analysis of both of these policies. In Section \ref{sub: asymptotic optimality} we state the asymptotic optimality result.

\subsection{Join the Largest Marginal Utility (JLMU)}
\label{sub: join the largest marginal utility (jlmu)}

JLMU assigns every new task to a server pool that currently has the best ranked coordinates, thus also the largest marginal utility. Formally, define
\begin{equation}
\label{eq: definition of sigma}
\sigma(q) = \left(\sigma_i(q), \sigma_j(q)\right) \defeq \max \set{(i, j) \in \calI_+}{q(i, j - 1) > q(i, j)}
\end{equation}
for each occupancy state $q$. The maximum is taken with respect to $\menor$ and the condition $q(i, j - 1) > q(i, j)$ implies that some server pool of class $i$ has precisely $j - 1$ tasks. If $q_n$ is the occupancy state right before a task arrives, then JLMU assigns the task to a server pool of class $\sigma_i\left(q_n\right)$ with exactly $\sigma_j\left(q_n\right) - 1$ tasks.

The coordinates obtained through \eqref{eq: definition of sigma} correspond to server pools with the largest marginal utility by definition of $\menor$. In addition, observe that the dictionary order is used to break ties between coordinates associated with the same marginal utility. If two server pools have the same coordinates, then it does not matter which of them is assigned the new task since they are statistically identical. For definiteness, we postulate that the tie is broken uniformly at random.

If all the server pools have the same utility function, then JLMU reduces to JSQ and the overall utility is a Schur-concave function of $\bX_n$. If in addition the service times are exponential, then the stochastic optimality properties proven in \cite{sparaggis1993extremal,menich1991optimality} for JSQ imply that JLMU maximizes the mean stationary overall utility in this homogeneous setting. It might be natural to expect that the optimality with respect to the mean stationary overall utility extends to the heterogeneous setting. We refute this, however, in Section \ref{sub: suboptimality result}, where we construct a heterogeneous system for which JLMU is strictly suboptimal. The constructed example also hints at the underlying reasons for the suboptimality in the heterogeneous case. Essentially, instead of always assigning incoming tasks greedily, such that the increase in the overall utility is maximal, it is sometimes advantageous to dispatch the new tasks conservatively, to hedge against pronounced drops of the overall utility that may be caused by a quick succession of departures. The right balance between greedy and conservative actions depends intricately on the utility functions, but we prove that JLMU is always asymptotically optimal for exponential service times, regardless of the specific set of utility functions; this result is stated formally in Section \ref{sub: asymptotic optimality}.

\subsection{Self-Learning Threshold Assignment (SLTA)}
\label{sub: threshold}

JLMU relies on complete information about the number of tasks per server pool, which could be impractical in large-scale deployments. In contrast, SLTA only requires to store at most two bits per server pool, which is considerably less state information. In order to specify this policy, we need to describe its two components. Namely, the dispatching rule, for assigning the incoming tasks to the server pools, and the learning scheme, for dynamically adjusting a set of thresholds that the dispatching rule uses.

Define $\set{(i_k, j_k)}{k \geq 1}$ by
\begin{equation*}
(i_k, j_k) \defeq \max \set{(i, j) \in \calI_+}{(i, j) \neq (i_1, j_1), \dots, (i_{k - 1}, j_{k - 1})},
\end{equation*}
where the maximum is taken with respect to $\menor$, i.e., $(i_k, j_k)$ is the $k$th best ranked element of $\calI_+$. Given $r \geq 1$, we define a set of thresholds by
\begin{equation*}
\ell_i(r) \defeq \max \set{j \geq 0}{(i, j) \mayor (i_r, j_r)\ \text{or}\ j = 0} \quad \text{for all} \quad i.
\end{equation*}
Recall the optimal task assignment defined at the end of Section \ref{sub: structure of an optimal solution}. The learning scheme keeps an estimate $(i_{\br_n}, j_{\br_n})$ of the coordinates $\sigma_n^*$, which depend on the typically unknown offered load of the system. The index $\br_n$ determines this estimate and is used to compute thresholds from the above expression, which are in turn used to assign tokens to the server pools. Specifically, a server pool of class $i$ with exactly $j - 1$ tasks has:
\begin{itemize}
	\item a green token if $j - 1 < \ell_i(\br_n)$,
	
	\item a yellow token if $i = i_{\br_n}$ and $j - 1 \leq \ell_i(\br_n)$.
\end{itemize}
The first condition is equivalent to $\br_n > 1$ and $(i, j) \mayor (i_{\br_n}, j_{\br_n})$, and the second condition is equivalent to $i = i_{\br_n}$ and $(i, j) \mayorigual (i_{\br_n}, j_{\br_n})$. Also, a server pool can have both a green and a yellow token at the same time. As indicated in Figure \ref{fig: threshold}, a larger increase in the overall utility is obtained by dispatching tasks to server pools with green tokens first, then to server pools with yellow tokens and only afterwards to server pools without tokens.

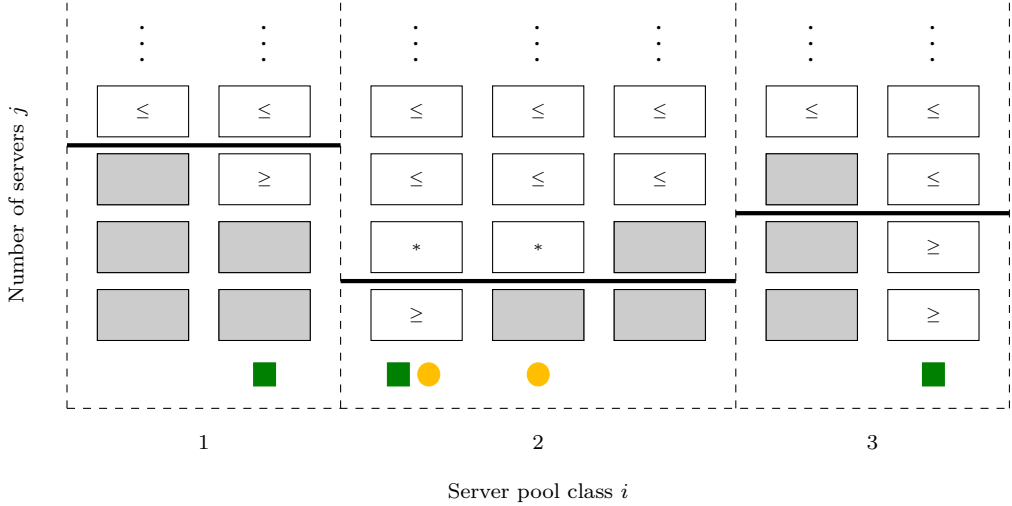
\begin{figure}
	\centering
	\begin{tikzpicture}[x = 1.6cm, y = 0.9cm]
	%% Class boundaries
	\draw[black, dashed] (-0.25, -1) -- (-0.25, 5);
	\draw[black, dashed] (2, -1) -- (2, 5);
	\draw[black, dashed] (5.25, -1) -- (5.25, 5);
	\draw[black, dashed] (7.5, -1) -- (7.5, 5);
	\draw[black, dashed] (-0.25, -1) -- (7.5, -1);
	
	%% Class 1
	\foreach \x in {0, 1}
	\foreach \y in {0, 1, 2, 3}
	\draw[black] (\x, \y) rectangle (\x + 0.75, \y + 0.75);
	
	\foreach \x in {0, 1}
	\node at (\x + 0.375, 0.375) {\tiny{$\geq$}};
	
	\foreach \x in {0, 1}
	\node at (\x + 0.375, 1.375) {\tiny{$\geq$}};
	
	\foreach \x in {0, 1}
	\node at (\x + 0.375, 2.375) {\tiny{$\geq$}};
	
	\foreach \x in {0, 1}
	\node at (\x + 0.375, 3.375) {\tiny{$\leq$}};
	
	\foreach \x in {0, 1}
	\node[rotate = 90] at (\x + 0.375, 4.375) {$\cdot\cdot\cdot$};
	
	\node at (0.875, -1.5) {\scriptsize{$1$}};
	
	\draw[black, ultra thick] (-0.25, 2.875) -- (2, 2.875);
	
	%% Class 2
	\foreach \x in {2.25, 3.25, 4.25}
	\foreach \y in {0, 1, 2, 3}
	\draw[black] (\x, \y) rectangle (\x + 0.75, \y + 0.75);
	
	\foreach \x in {2.25, 3.25, 4.25}
	\node at (\x + 0.375, 0.375) {\tiny{$\geq$}};
	
	\foreach \x in {2.25, 3.25, 4.25}
	\node at (\x + 0.375, 1.375) {\tiny{$*$}};
	
	\foreach \x in {2.25, 3.25, 4.25}
	\node at (\x + 0.375, 2.375) {\tiny{$\leq$}};
	
	\foreach \x in {2.25, 3.25, 4.25}
	\node at (\x + 0.375, 3.375) {\tiny{$\leq$}};
	
	\foreach \x in {2.25, 3.25, 4.25}
	\node[rotate = 90] at (\x + 0.375, 4.375) {$\cdot\cdot\cdot$};
	
	\node at (3.625, -1.5) {\scriptsize{$2$}};
	
	\draw[black, ultra thick] (2, 0.875) -- (5.25, 0.875);
	
	%% Class 3
	\foreach \x in {5.5, 6.5}
	\foreach \y in {0, 1, 2, 3}
	\draw[black] (\x, \y) rectangle (\x + 0.75, \y + 0.75);
	
	\foreach \x in {5.5, 6.5}
	\node at (\x + 0.375, 0.375) {\tiny{$\geq$}};
	
	\foreach \x in {5.5, 6.5}
	\node at (\x + 0.375, 1.375) {\tiny{$\geq$}};
	
	\foreach \x in {5.5, 6.5}
	\node at (\x + 0.375, 2.375) {\tiny{$\leq$}};
	
	\foreach \x in {5.5, 6.5}
	\node at (\x + 0.375, 3.375) {\tiny{$\leq$}};
	
	\foreach \x in {5.5, 6.5}
	\node[rotate = 90] at (\x + 0.375, 4.375) {$\cdot\cdot\cdot$};
	
	\node at (6.375, -1.5) {\scriptsize{$3$}};
	
	\draw[black, ultra thick] (5.25, 1.875) -- (7.5, 1.875);
	
	%% Axis labels
	\node[rotate = 90] at (-0.65, 2) {\scriptsize{Number of servers $j$}};
	\node at (3.625, -2.25) {\scriptsize{Server pool class $i$}};
	
	%% Busy servers
	\foreach \y in {0, 1, 2}
	\draw[fill = black!20!white] (0, \y) rectangle (0.75, \y + 0.75);
	
	\foreach \y in {0, 1}
	\draw[fill = black!20!white] (1, \y) rectangle (1.75, \y + 0.75);
	
	\draw[fill = black!20!white] (3.25, 0) rectangle (4, 0.75);
	
	\foreach \y in {0, 1}
	\draw[fill = black!20!white] (4.25, \y) rectangle (5, \y + 0.75);
	
	\foreach \y in {0, 1, 2}
	\draw[fill = black!20!white] (5.5, \y) rectangle (6.25, \y + 0.75);
	
	%% Tokens
	\draw[color = ao(english), fill = ao(english)] (1.375 - 0.09, -0.5 - 0.17) rectangle (1.375 + 0.09, -0.5 + 0.17);
	
	\draw[color = ao(english), fill = ao(english)] (2.625 - 0.15 - 0.09, -0.5 - 0.17) rectangle (2.625 - 0.15 + 0.09, -0.5 + 0.17);
	
	\draw[color = amber, fill = amber] (2.625 + 0.1, -0.5) ellipse (0.09 and 0.17);
	
	\draw[color = amber, fill = amber] (3.625, -0.5) ellipse (0.09 and 0.17);
	
	\draw[color = ao(english), fill = ao(english)] (6.875 - 0.09, -0.5 - 0.17) rectangle (6.875 + 0.09, -0.5 + 0.17);
	
	\end{tikzpicture}
	\caption{Schematic representation of the thresholds and tokens used by SLTA for $(i_{\br_n}, j_{\br_n}) = (2, 2)$. The thresholds are indicated by thick horizontal lines that cross the server pools, and the tokens are represented using squares and circles underneath the server pools; a square corresponds to a green token and a circle corresponds to a yellow token. Assuming that the rectangular slots within a server pool are always filled from bottom to top, the slots marked with an $*$ provide a marginal utility of $\Delta(i_{\br_n}, j_{\br_n} - 1)$. The symbols $\leq$ and $\geq$ indicate how the marginal utility of the other slots compares to the latter value.}
	\label{fig: threshold}
\end{figure}

The tokens are used by the dispatching rule. Specifically, when a task arrives, it is assigned to a server pool according to the following criteria.
\begin{itemize}
	\item In the presence of green tokens of class $i \neq i_{\br_n - 1}$, the dispatcher picks one of these green tokens uniformly at random, and if only green tokens of class $i = i_{\br_n - 1}$ remain, then one of these is picked. Then the task is sent to the corresponding server pool.
	
	\item In the presence of only yellow tokens, the dispatcher picks a yellow token uniformly at random and sends the task to the corresponding server pool.
	
	\item Otherwise, the task is sent to a server pool chosen uniformly at random.
\end{itemize}
If $(i_{\br_n}, j_{\br_n})$ are the coordinates $\sigma_n^*$ defined in \eqref{eq: definition of sigma_n^N}, then this dispatching rule drives the occupancy state of the system towards the optimal task assignment $q_n^*$ specified in \eqref{eq: definition of q_n^N}.

The learning scheme aims at finding the coordinates $\sigma_n^*$, which depend on the typically unknown offered load. The learning scheme is parameterized by $\beta_n > 0$ and adjusts the value of $\br_n$ at certain arrival epochs, in steps of one unit. Specifically, when a task arrives, the learning scheme acts only under the following circumstances.
\begin{itemize}
	\item If the system has at least $n \beta_n$ green tokens and at least one belongs to a server pool of class $i_{\br_n - 1}$, then $\br_n$ is decremented by one after the task is dispatched.
	
	\item If the number of yellow tokens is smaller than or equal to one and there are no other tokens, then $\br_n$ is incremented by one after the task is dispatched.
\end{itemize}
Observe that exactly one of the thresholds changes when the value of $\br_n$ is modified, and that this threshold changes by one unit. Also,
\begin{equation*}
n - \sum_{i = 1}^m n\bq_n\left(i, \ell_i(\br_n)\right) \quad \text{and} \quad n\alpha_n\left(i_{\br_n}\right) - n\bq_n\left(i_{\br_n}, j_{\br_n}\right)
\end{equation*}
are the number of green and yellow tokens, respectively.

\subsubsection{Comparison with the homogeneous case}

When all the server pools are of the same class, SLTA reduces to the load balancing policy analyzed in \cite{goldsztajn2021learning}. In this case there is a single threshold whose optimal value is simply $\floor{\rho}$. When the threshold has this value, the number of green tokens and the total number of tokens are typically small and positive, respectively. On the other hand, when the threshold is below optimal, the total number of tokens tends to be zero, and when the threshold is larger than optimal, the number of green tokens tends to be relatively large. These properties are used to adjust the threshold in an online manner when the offered load is unknown. In few words, the threshold is increased in the absence of tokens, and it is decreased if the number of green tokens is large enough.

In the general case, there are as many thresholds as the number of server pool classes, and the optimal threshold values depend intricately on the utility functions and the offered load. However, the ranking $\menor$ introduced in Section \ref{sub: structure of an optimal solution} makes it possible to express all the thresholds as a function of the coordinates $\sigma_n^*$. Hence, the optimal thresholds can still be found through a one-directional search, but now in the totally ordered space $\calI_+$. Moreover, since this space is countable, the search can be carried out by adjusting the integral parameter $\br_n$ until $(i_{\br_n}, j_{\br_n})$ reaches the optimal value $\sigma_n^*$.

The learning scheme of SLTA operates so that all the thresholds are at their optimal values if and only if $\br_n$ is at its optimal value, and when this happens, the number of green tokens and the total number of tokens are typically small and positive, respectively. When $\br_n$ is below optimal, all the thresholds are smaller than or equal to their optimal values, and at least one of the thresholds is strictly smaller than optimal; in this case, the total number of tokens tends to be zero. Similarly, when $\br_n$ is above optimal, all the thresholds are larger than or equal to their optimal values and at least one of the thresholds is above optimal; as a result, the number of green tokens tends to be relatively large. As in the homogeneous case, these observations are used to adjust $\br_n$ over time. Loosely speaking, in the absence of tokens, $\br_n$ is increased by one unit, which implies that one of the thresholds is increased by one unit, and when the number of green tokens is large enough, $\br_n$ is decreased by one unit, and thus one of the thresholds is decreased by one unit.

The dispatching rule and the online learning scheme of SLTA make distinctions between green tokens of class $i_{\br_n - 1}$ and green tokens of any other class. The rationale is that the marginal utility of server pools of class $i$ with $\ell_i(\br_n) - 1$ tasks is the lowest when $i = i_{\br_n - 1}$, thus it makes sense to give green tokens of this class the lowest priority for receiving new tasks. While this may slightly improve performance, it is not crucial. Nevertheless, the differential treatment of class $i_{\br_n - 1}$ simplifies the mathematical analysis of SLTA. In particular, the distinction made in the description of the learning scheme ensures that if the system had yellow tokens, then it will continue to have yellow tokens after $\br_n$ is decreased, which is used in the below stated Remark \ref{rem: dynamics of threshold}. In addition, the differential treatment by the dispatching rule simplifies the proof of Proposition \ref{prop: upper bound for br}.

\subsection{Stochastic models}
\label{sub: stochastic models}

If service times are exponentially distributed, then $\bq_n$ and $(\bq_n, \br_n)$ are continuous-time Markov chains when the load balancing policies are JLMU and SLTA, respectively. In either case, the process $\bs_n$ that describes the normalized total number of tasks is defined by \eqref{eq: definition of s_n}. Due to the infinite-server dynamics of the system, $n \bs_n$ has the law of an $M/M/\infty$ queue with arrival rate $n\lambda$ and service rate $\mu$.

Let $\ell_1$ be the space of absolutely summable sequences in $\R^\calI$, equipped with the norm
\begin{equation*}
\norm{x}_1 \defeq \sum_{(i, j) \in \calI} |x(i, j)| \quad \text{for all} \quad x \in \ell_1.
\end{equation*}
Throughout we assume that $\bs_n(0)$ is finite, so $\bq_n(0)$ takes values in $\ell_1$. As a result, if we let $F_n \defeq \set{k / n}{0 \leq k \leq n}$, then $\bq_n$ takes values in the set
\begin{equation*}
Q_n \defeq \set{q \in F_n^\calI \cap \ell_1}{q(i, j + 1) \leq q(i, j) \leq q(i, 0) = \alpha_n(i)\ \text{for all}\ (i, j) \in \calI}.
\end{equation*}
If the load balancing policy is JLMU, then the state space of $\bq_n$ is defined as the subset $S_n$ of $Q_n$ that is reachable from an empty occupancy state. If the load balancing policy is SLTA, then the state space of $(\bq_n, \br_n)$ is the subset $S_n$ of $Q_n \times \set{r \in \N}{r \geq 1}$ that is reachable from an empty occupancy state with $\br_n = 1$.

The notation used for the processes $\bs_n$ and $\bq_n$, for the state space $S_n$, and for some other objects that will be defined later, is exactly the same for JLMU and SLTA, but we always indicate which policy is being considered.

\subsection{Asymptotic optimality}
\label{sub: asymptotic optimality}

Throughout the rest of the paper, we assume that there exist constants $\alpha(i) \in (0, 1)$ and a random variable $q_0$ such that the following limits hold:
\begin{equation}
\label{ass: size of classes and initial condition}
\begin{split}
&\lim_{n \to \infty} \alpha_n(i) = \alpha(i) \quad \text{for all} \quad i \quad \text{and} \quad \lim_{n \to \infty} \norm{\bq_n(0) - q_0}_1 = 0.
\end{split}
\end{equation}
In analogy with \eqref{eq: definition of sigma_n^N} and \eqref{eq: definition of q_n^N}, we consider the unique $\sigma_* = \left(i_{r_*}, j_{r_*}\right) \in \calI_+$ such that 
\begin{equation}
\label{eq: definition of sigma*}
\sum_{(i, j) \mayor \sigma_*} \alpha(i) \leq \rho < \sum_{(i, j) \mayorigual \sigma_*} \alpha(i),
\end{equation}
and we define a occupancy state $q_*$ in terms of $\sigma_*$ by
\begin{equation}
\label{eq: definition of q*}
q_*(i, j) \defeq \begin{cases}
0												&	\text{if} \quad (i, j) \menor \sigma_*, \\
\alpha(i)										&	\text{if} \quad (i, j) \mayor \sigma_*, \\
\rho - \sum_{(i, j) \mayor \sigma_*} \alpha(i)	&	\text{if} \quad (i, j) = \sigma_*,
\end{cases}
\quad \text{for all} \quad (i, j) \in \calI_+.
\end{equation}

In Section \ref{sub: drift analysis} we establish that $\bq_n$ and $(\bq_n, \br_n)$ have a unique stationary distribution for all $n$, when the assignment policies are JLMU and SLTA, respectively. The following theorem is proved in Section \ref{sub: proof of the asymptotic optimality} and implies that both policies are asymptotically optimal if the marginal utilities are bounded and the service times are exponentially distributed; we also require that \eqref{ass: size of classes and initial condition} holds and we impose some mild technical assumptions, to be stated in Section \ref{subsub: technical assumptions}. Note that the condition on the marginal utilities always holds when the utility functions are non-decreasing, due to the concavity of these functions.

\begin{theorem}
	\label{the: asymptotic optimality}
	Suppose that service times are exponentially distributed. Then the following statements hold under \eqref{ass: size of classes and initial condition} and the assumptions of Section \ref{subsub: technical assumptions}.
	\begin{enumerate}
		\item[(a)] Suppose that JLMU is used and let $q_n$ have the stationary distribution of $\bq_n$. Then the random variables $q_n$ converge weakly in $\ell_1$ to $q_*$.
		
		\item[(b)] Suppose that SLTA is used and let $(q_n, r_n)$ have the stationary distribution of $(\bq_n, \br_n)$. The random variables $(q_n, r_n)$ converge weakly in $\ell_1 \times \N$ to $(q_*, r_*)$.
	\end{enumerate}
	Furthermore, if the load balancing policy is either JLMU or SLTA and the marginal utilities are bounded, then the random variables $u(q_n)$ are uniformly integrable and
	\begin{equation*}
	\lim_{n \to \infty} \bE\left[u(q_n)\right] = u(q_*) = \lim_{n \to \infty} u\left(q_n^*\right).
	\end{equation*}
\end{theorem}

The claims concerning the stationary overall utilities are proved using (a) and (b), as well as the fact that $u(q)$ is a bounded linear functional of $q \in \ell_1$. In order to establish (a) and (b), we first use drift analysis to prove that the random variables in (a) and (b) are tight in $\ell_1$ and  $\ell_1 \times \N$, respectively. Then (a) is established through an interchange of limits argument based on a fluid limit and a global asymptotic stability result for the fluid dynamics; these two results are stated in Theorems \ref{the: fluid limit of jlmu} and \ref{the: global asymptotic stability}, respectively. A different type of argument is used to prove (b). Namely, the fluid limit step is circumvented and Theorem \ref{the: fluid limit of threshold} serves as the counterpart of Theorems \ref{the: global asymptotic stability} and \ref{the: fluid limit of jlmu}, as illustrated in Figure \ref{fig: interchange of limits}.

\begin{figure}
	\centering
	\begin{subfigure}{0.49\columnwidth}
		\centering
		\begin{tikzpicture} [
		roundnode/.style = {circle, draw = white, minimum height = 20mm, minimum width = 20mm},
		node distance = 20mm
		]
		%% Nodes
		\node[roundnode]	(1)							{\scriptsize{$\bq_n(t)$}};
		\node[roundnode]	(2)		[right = of 1]		{\scriptsize{$q_n$}};
		\node[roundnode]	(3)		[below = of 1]		{\scriptsize{$\bq(t)$}};
		\node[roundnode]	(4)		[below = of 2]		{\scriptsize{$q_*$}};
		%% Lines
		\draw[->] (1.east) -- node[above, midway]{\scriptsize{$t \to \infty$}} (2.west);
		\draw[->] (1.south) -- node[midway, rotate = -90, anchor = base, yshift = -3mm]{\scriptsize{$n \to \infty$}} node[midway, rotate = -90, anchor = base, yshift = 2mm]{\scriptsize{Theorem \ref{the: fluid limit of jlmu}}} (3.north);
		\draw[->] (2.south) -- node[midway, rotate = -90, anchor = base, yshift = 2mm]{\scriptsize{$n \to \infty$}} (4.north);
		\draw[->] (3.east) --  node[below, midway]{\scriptsize{$t \to \infty$}} node[above, midway]{\scriptsize{Theorem \ref{the: global asymptotic stability}}} (4.west);
		\end{tikzpicture}
		\label{fig: interchange of limits for jlmu}
	\end{subfigure}
	\hfill
	\begin{subfigure}{0.49\columnwidth}
		\centering
		\begin{tikzpicture} [
		roundnode/.style = {circle, draw = white, minimum height = 20mm, minimum width = 20mm},
		node distance = 2cm
		]
		%% Nodes
		\node[roundnode]	(1)							{\scriptsize{$\left(\bq_n(t), \br_n(t)\right)$}};
		\node[roundnode]	(2)		[right = of 1]		{\scriptsize{$(q_n, r_n)$}};
		\node[roundnode]	(4)		[below = of 2]		{\scriptsize{$(q_*, r_*)$}};
		\node[rotate = -45] at (1, -3) {\scriptsize{Theorem \ref{the: fluid limit of threshold}}};
		%% Lines
		\draw[->] (1.east) -- node[above, midway]{\scriptsize{$t \to \infty$}} (2.west);
		\draw[->] (2.south) -- node[midway, rotate = -90, anchor = base, yshift = 2mm]{\scriptsize{$n \to \infty$}} (4.north);
		\draw[->, dashed] (1.south) .. controls +(down:20mm) and +(left:20mm) .. (4.west);
		\end{tikzpicture}
		\label{fig: interchange of limits for threshold}
	\end{subfigure}
	\caption{Schematic view of the proofs of (a) and (b) of Theorem \ref{the: asymptotic optimality}, on the left and right, respectively.}
	\label{fig: interchange of limits}
\end{figure}
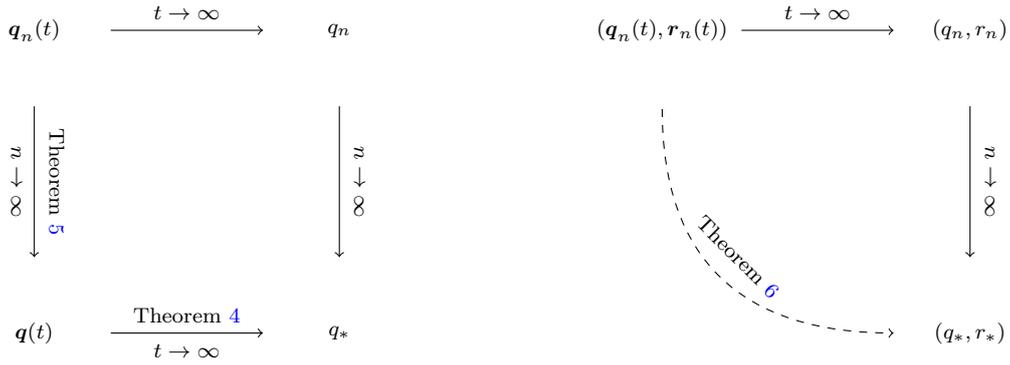

Deriving a fluid limit for a SLTA system would be inherently difficult due to the intricate interdependence between the dispatching rule and the learning scheme, and because the actions of the learning scheme are triggered by excursions of the occupancy process that have vanishing size. We deal with these challenges using a methodology of \cite{goldsztajn2021learning} to derive the fluid approximation of Theorem \ref{the: fluid limit of threshold}, which consists of asymptotic bounds, over arbitrarily long intervals of time, for the occupancy state and the thresholds. As noted earlier, this fluid approximation serves as a counterpart of both the fluid limit and the global asymptotic stability results for JLMU.

\subsubsection{Simulation experiments}
\label{subsub: simulation experiments}

The asymptotic optimality of JLMU and SLTA is illustrated by Table \ref{tab: empirical means}, which shows estimates of the mean stationary overall utility $\bE[u(q_n)]$ for simulation experiments with different values of $n$. All the estimates correspond to systems with two server pool classes of equal size and utility functions of the form $u_i(x) = x \log(r(i) / x)$ for $r = (20, 30)$. Also, two different values of $\rho$ are considered, so that the optimal task assignment defined in \eqref{eq: definition of q_n^N} takes two distinct forms. For $\rho = 9.75$, the optimal task assignment is such that all server pools of class $1$ have $8$ tasks, half of the server pools of class $2$ have $11$ tasks and the other half have $12$ tasks. For $\rho = 10$, the optimal task assignment is such that all server pools of class $1$ have $8$ tasks and all server pools of class $2$ have $12$ tasks.

\newcolumntype{C}[1]{>{\centering\arraybackslash}m{#1}}
\begin{table}
	\centering
	{\footnotesize
		\begin{tabular}{|C{6mm}|C{13mm}|C{24mm}|C{24mm}|C{13mm}|C{24mm}|C{24mm}|}
			\hline
			& \multicolumn{3}{c|}{$\rho = 9.75 \quad u^*(\rho) \cong 9.1731$} & \multicolumn{3}{c|}{$\rho = 10 \quad u^*(\rho) \cong 9.1629$} \\
			\hline
			$n$ & $u^*\left(\avg{\bs_n}\right)$ & $\avg{u(\bq_n)}$ JLMU & $\avg{u(\bq_n)}$ SLTA & $u^*\left(\avg{\bs_n}\right)$ & $\avg{u(\bq_n)}$ JLMU & $\avg{u(\bq_n)}$ SLTA \\
			\hline
			$50$ & $9.1752$ & $9.1652$ & $9.1649$ & $9.1594$ & $9.1433$ & $9.1439$ \\
			\hline
			$100$ & $9.1746$ & $9.1699$ & $9.1698$ & $9.1627$ & $9.1499$ & $9.1498$ \\
			\hline
			$150$ & $9.1738$ & $9.1706$ & $9.1705$ & $9.1633$ & $9.1534$ & $9.1534$ \\
			\hline
			$200$ & $9.1743$ & $9.1723$ & $9.1723$ & $9.1634$ & $9.1556$ & $9.1555$ \\
			\hline
			$250$ & $9.1735$ & $9.1720$ & $9.1720$ & $9.1639$ & $9.1572$ & $9.1572$ \\
			\hline
			$300$ & $9.1738$ & $9.1727$ & $9.1726$ & $9.1613$ & $9.1548$ & $9.1547$ \\
			\hline
			$350$ & $9.1742$ & $9.1734$ & $9.1734$ & $9.1619$ & $9.1563$ & $9.1562$ \\
			\hline
			$400$ & $9.1734$ & $9.1728$ & $9.1728$ & $9.1632$ & $9.1578$ & $9.1578$ \\
			\hline
			$450$ & $9.1737$ & $9.1732$ & $9.1731$ & $9.1630$ & $9.1576$ & $9.1576$ \\
			\hline
			$500$ & $9.1726$ & $9.1721$ & $9.1721$ & $9.1632$ & $9.1583$ & $9.1583$ \\
			\hline
		\end{tabular}
	}
	\caption{Results of simulation experiments with different values of $n$. The systems considered have two server pool classes of equal size and utility functions of the form $u_i(x) = x \log(r(i) / x)$ with $r = (20, 30)$. Service times are exponentially distributed with $\mu = 1$ and the time averages $\avg{\scdot}$ are computed over intervals of length $180$ with equilibrium initial conditions. The same sequences of inter-arrival and service times were used in the simulation experiments for JLMU and SLTA. Also, $\beta_n = 1 / n^{0.45}$ in the simulations for SLTA. The upper bound of Theorem \ref{the: upper bound for mean stationary overall utility} for a normalized offered load of $x$ is denoted by $u^*(x)$.}
	\label{tab: empirical means}
\end{table}

All the server pools of the same class have the same number of tasks when $\rho = 10$, and thus we say that the optimal task assignment $q_n^*$ is \emph{integral}; i.e., the optimal task assignment is integral if $q_n^*(\sigma_n^*) = 0$. In contrast, the optimal task assignment is \emph{fractional} when $\rho = 9.75$ since server pools of class $2$ may have either $11$ or $12$ tasks. Server pools of class $1$ behave similarly in the fractional and integral settings: almost all of the time all server pools of class $1$ have precisely $8$ tasks when $n$ is moderately large. However, the behavior of server pools of class $2$ depends on the setting. In the fractional case, server pools typically have $11$ or $12$ tasks and the fractions of server pools with $11$ and $12$ tasks oscillate around a half. In the integral case, server pools typically have $12$ tasks but a small number of server pools sometimes have $11$ or $13$ tasks instead. The aggregate utility of the system decreases by $\Delta(2, 11)$ whenever a server pool of class $2$ goes from $12$ to $11$ tasks, and increases by the same quantity when the server pool goes from $11$ to $12$ tasks. Therefore, the contributions to the average overall utility of the oscillations observed in the fractional case roughly balance each other. In the integral case, the aggregate utility increases by $\Delta(2, 12)$, instead of $\Delta(2, 11)$, if a server pool of class $2$ goes from $12$ to $13$ tasks. So in the integral case the contributions to the average overall utility of class $2$ server pools that drop to $11$ tasks or reach $13$ tasks are amplified by different marginal utilities. As a result, the mean stationary overall utility $\bE[u(q_n)]$ is closer to the upper bound $u(q_n^*)$ in the fractional case, as reflected by the estimates in Table \ref{tab: empirical means}.

Although there is a difference between the fractional and integral settings, in both settings the empirical mean of the overall utility $u(\bq_n)$ is extremely close to the upper bound $u(q_n^*)$ across all the values of $n$ listed in Table \ref{tab: empirical means}. Furthermore, the deviation of the empirical mean from the upper bound approaches zero as $n$ increases in both cases. We also observe that the empirical mean of $u(\bq_n)$ is almost the same for JLMU and SLTA in all the experiments, and particularly for the largest values of $n$.

A final remark on the simulation experiments is that the empirical mean of $u(\bq_n)$ is slightly larger than $u(q_n^*)$ in a few of the experiments within the fractional setting: both for JLMU and SLTA when $n = 350$ and just for JLMU when $n = 450$. It may be checked that the statement of Theorem \ref{the: upper bound for mean stationary overall utility} still holds if the stationary expectation sign is replaced by a time average and the upper bound is computed through \eqref{eq: definition of q_n^N} and \eqref{eq: definition of sigma_n^N} but with $\rho$ replaced by the time average of $\bs_n$; the proof does not change. The value of the upper bound when $\rho$ is replaced by the time average of $\bs_n$ is displayed in Table \ref{tab: empirical means}, and in all the experiments the empirical mean of $u(\bq_n)$ is indeed smaller than this empirical upper bound. Thus, the experiments where the empirical mean of $u(\bq_n)$ slightly exceeds the upper bound $u(q_n^*)$ are an indication of how close the performance of JLMU and SLTA is to optimal. 

\section{Approximation theorems}
\label{sec: approximation theorems}

In this section we assume exponential service times and we state several results used to prove Theorem \ref{the: asymptotic optimality}. In Section \ref{sub: fluid model of jlmu} we specify a fluid model of a JLMU system, based on differential equations, and we state some properties of this model. In Section \ref{sub: limit theorems} we state limit theorems that characterize the asymptotic transient behavior of JLMU and SLTA.

\subsection{Fluid model of JLMU}
\label{sub: fluid model of jlmu}

Consider a large-scale system where the load balancing policy is JLMU, and assume that $\alpha(i)$ is the fraction of server pools of class $i$. Then the occupancy state of the system remains within the set
\begin{equation*}
Q \defeq \set{q \in [0, 1]^\calI \cap \ell_1}{q(i, j + 1) \leq q(i, j) \leq q(i, 0) = \alpha(i)\ \text{for all}\ (i, j) \in \calI}.
\end{equation*}
The evolution of the occupancy state of this large-scale system can be modeled through the system of differential equations introduced in the following definition.

\begin{definition}
	\label{def: fluid trajectory}
	We say that $\map{\bq}{[0, \infty)}{Q}$ is a fluid trajectory if the coordinate functions $\bq(i, j)$ are absolutely continuous for all $(i, j) \in \calI$ and the following conditions hold almost everywhere with respect to the Lebesgue measure:
	\begin{subequations}
		\label{eq: fluid dynamics}
		\begin{align}
		&\dot{\bq}(i, j) = \Lambda(\bq, i, j) - \mu j \left[\bq(i, j) - \bq(i, j + 1)\right], \label{seq: fluid dynamics 1} \\
		&\Lambda(\bq, i, j) \geq 0 \quad \text{for all} \quad (i, j) \in \calI_+, \label{seq: fluid dynamics 2}
		\end{align}
	\end{subequations}
	where $\map{\Lambda}{Q \times \calI_+}{\R}$ is defined by
	\begin{equation*}
	\Lambda(q, i, j) \defeq \begin{cases}
	0																									&	\text{if} \quad (i, j) \menor \sigma(q), \\
	\mu j \left[\alpha(i) - q(i, j + 1)\right]															&	\text{if} \quad (i, j) \mayor \sigma(q), \\
	\displaystyle \lambda - \sum_{(k, l) \mayor \sigma(q)} \mu l \left[\alpha(k) - q(k, l + 1)\right]	&	\text{if} \quad (i, j) = \sigma(q).
	\end{cases}
	\end{equation*}
\end{definition}

In the latter definition, $\lambda$ is the arrival rate of tasks normalized by the number of server pools, $\mu$ is the service rate of tasks and $\bq(i, j)$ represents the fraction of server pools which are of class $i$ and have at least $j$ tasks. Thus, the system of differential equations \eqref{eq: fluid dynamics} has a simple interpretation. The right-most term of \eqref{seq: fluid dynamics 1} corresponds to the departure rate of tasks from server pools of class $i$ with exactly $j$ tasks and $\Lambda(\bq, i, j)$ represents the arrival rate of tasks to server pools which belong to class $i$ and have precisely $j - 1$ tasks. The definition of $\Lambda$ is motivated by the following remarks.
\begin{itemize}
	\item Server pools of class $i$ with exactly $j - 1$ tasks are not assigned additional tasks if $(i, j) \menor \sigma(\bq)$. Hence, we should have $\Lambda(\bq, i, j) = 0$ in this case.
	
	\item All server pools of class $i$ have at least $j$ tasks if $(i, j) \mayor \sigma(\bq)$. Therefore, $\Lambda(\bq, i, j)$ should be equal to the last term of \eqref{seq: fluid dynamics 1} in this case, since $\bq(i, j)$ is at its maximum value and thus its derivative should be zero.
	
	\item The total arrival rate of tasks normalized by the number of server pools is equal to $\lambda$ and this determines the value of $\Lambda(\bq, i, j)$ when $(i, j) = \sigma(\bq)$.
\end{itemize}

\subsubsection{Properties of fluid trajectories}
\label{subsub: properties of fluid trajectories - main results section}

The two results stated below are proved in Section \ref{sub: properties of fluid trajectories}. The first one is a uniqueness theorem for the solutions of \eqref{eq: fluid dynamics}. Existence is ensured by Theorem \ref{the: fluid limit of jlmu} of Section \ref{sub: limit theorems}.

\begin{theorem}
	\label{the: uniqueness of fluid trajectories}
	For each initial condition $q \in Q$ there exists at most a unique fluid trajectory $\bq$ with initial condition $\bq(0) = q$.
\end{theorem}

In order to prove this theorem, we first show that all fluid trajectories satisfy an infinite system of integral equations, stated using Skorokhod one-dimensional reflection mappings. The theorem is then proved using a Lipschitz property of these mappings and a uniqueness result for certain Kolmogorov backward equations.

Besides uniqueness of solutions of \eqref{eq: fluid dynamics}, we also establish that there exists a unique equilibrium point and that this equilibrium point is globally asymptotically stable, i.e., all fluid trajectories converge to the unique equilibrium over time.

\begin{theorem}
	\label{the: global asymptotic stability}
	Let $q_*$ be as in \eqref{eq: definition of q*}. Then $q_*$ is the unique equilibrium of \eqref{eq: fluid dynamics}. Furthermore, all fluid trajectories converge to $q_*$ in $\ell_1$ over time.
\end{theorem}

Recall that \eqref{eq: definition of q*} is the counterpart of \eqref{eq: definition of q_n^N}, which is used to formulate the upper bound for the mean stationary overall utility provided in Theorem \ref{the: upper bound for mean stationary overall utility}. It is not difficult to check that $u\left(q_n^*\right) \to u\left(q_*\right)$ as $n$ grows large, which hints at the asymptotic optimality of JLMU.

\subsection{Limit theorems}
\label{sub: limit theorems}

In Section \ref{sub: coupled construction of sample paths} we construct the processes defined in Section \ref{sub: stochastic models} on a common probability space $(\Omega, \calF, \prob)$ for all $n$, in such a way that the sample paths of the occupancy processes lie in the space $D_{\ell_1}[0, \infty)$ of c\`adl\`ag functions with values in $\ell_1$, which we endow with the topology of uniform convergence over compact sets. This construction is used to prove limit theorems that characterize the asymptotic transient behavior of JLMU and SLTA. Before stating these theorems, we introduce some mild technical assumptions.

\subsubsection{Technical assumptions}
\label{subsub: technical assumptions}

As indicated earlier, we assume that \eqref{ass: size of classes and initial condition} holds with $q_0$ a random variable that takes values in $Q$ and represents the limiting initial occupancy state. The initial number of tasks in the limit, normalized by the number of server pools, is defined as
\begin{equation}
\label{eq: initial normalized number of tasks}
s_0 \defeq \sum_{(i, j) \in \calI_+} q_0(i, j).
\end{equation}

If the load balancing policy is SLTA, then we assume that the first inequality in \eqref{eq: definition of sigma*} is strict and that there exists a constant $\gamma_0 \in (0, 1/ 2)$ such that
\begin{equation}
\label{ass: conditions on beta}
\lim_{n \to \infty} \beta_n = 0 \quad \text{and} \quad \liminf_{n \to \infty} n^{\gamma_0}\beta_n > 0.
\end{equation}
These assumptions are used to prove that the learning scheme reaches an equilibrium in all large enough systems with probability one. Finally, we adopt the following assumptions about the initial state of the system: there exists a random variable $R \geq 1$ such that
\begin{subequations}
	\label{ass: boundedness and goodness of r}
	\begin{align}
	&\br_n(0) \leq R, \\
	&\bq_n(0, i, j) < \alpha_n(i) \quad \text{for all} \quad (i, j) \menorigual \left(i_{\br_n(0)}, j_{\br_n(0)}\right), \label{ass: goodness of r}
	\end{align}
\end{subequations}
for all $n$ with probability one. We impose \eqref{ass: goodness of r} just to simplify the analysis; this property always holds after a certain time, which depends on the initial state of the system.

\begin{remark}
	\label{rem: dynamics of threshold}
	Property \eqref{ass: goodness of r} is preserved by arrivals and departures, thus it holds at all times provided that it holds at time zero. Furthermore, every new task is sent to a server pool with coordinates $(i, j) \mayorigual (i_{\br_n}, j_{\br_n})$ if the number of tokens is positive right before the arrival. Hence, \eqref{ass: goodness of r} implies that tasks are sent to server pools with coordinates $(i, j) \mayorigual (i_{\br_n}, j_{\br_n})$  at all times and for all $n$ with probability one.
\end{remark}

\subsubsection{Statements of the theorems}
\label{subsub: statements of the theorems}

First we state a fluid limit for JLMU, proven in Section \ref{sub: proof of the fluid limit}. In view of Theorem \ref{the: global asymptotic stability}, this fluid limit implies that, as $n$ grows large, the occupancy processes of JLMU approach functions which converge over time to the unique equilibrium of \eqref{eq: fluid dynamics}.

\begin{theorem}
	\label{the: fluid limit of jlmu}
	Suppose that the load balancing policy is JLMU. Then there exists a set of probability one $\Gamma$ with the following property. If $\omega \in \Gamma$, then $\bq_n(\omega)$ converges in $D_{\ell_1}[0, \infty)$ to the unique fluid trajectory with initial condition $q_0(\omega)$.
\end{theorem}

Since $q_0$ is arbitrary, the above theorem implies that solutions to \eqref{eq: fluid dynamics} exist for all initial conditions. Therefore, Theorems \ref{the: uniqueness of fluid trajectories} and \ref{the: fluid limit of jlmu} imply that for each initial condition $q \in Q$ there exists a unique fluid trajectory with initial condition $q$.

The proof of Theorem \ref{the: fluid limit of jlmu} uses a methodology of \cite{bramson1998state} to prove that, with probability one, every subsequence of $\set{\bq_n}{n \geq 1}$ has a further subsequence which converges uniformly over compact sets with respect to a metric for the product topology of $\R^\calI$. Then we show that this convergence in fact holds with respect to $\norm{\scdot}_1$ and that the limits of convergent subsequences are fluid trajectories, also with probability one.

The counterpart of Theorems \ref{the: global asymptotic stability} and \ref{the: fluid limit of jlmu} for SLTA is the following result. The proof is provided in Section \ref{sub: evolution of the learning scheme} and is based on a methodology developed in \cite{goldsztajn2021learning}.

\begin{theorem}
	\label{the: fluid limit of threshold}
	Suppose that the load balancing policy is SLTA and let $\sigma_*$ and $r_*$ be as in \eqref{eq: definition of sigma*}. There exist $\map{\tau_\eq}{[0, \infty)}{\R}$ and a set of probability one $\Gamma$ with the following property. If $\omega \in \Gamma$ and $T \geq \tau > \tau_\eq(s_0(\omega))$, then the next limits hold:
	\begin{subequations}
		\label{eq: limit of threshold}
		\begin{align}
		&\lim_{n \to \infty} \sup_{t \in [\tau, T]} |\br_n(\omega) - r_*| = 0, \label{seq: limit of threshold 1} \\
		&\lim_{n \to \infty} \sup_{t \in [\tau, T]} n^\gamma\left|\alpha(i) - \bq_n(\omega, t, i, j)\right| = 0 \quad \text{if} \quad (i, j) \mayor \sigma_* \quad \text{and} \quad \gamma \in [0, 1/2), \label{seq: limit of threshold 2} \\
		&\limsup_{n \to \infty} \sup_{t \in [\tau, T]} \sum_{(i, j) \menor \sigma^*} \bq_n(\omega, t, i, j)\e^{\mu(t - \tau)} \leq c(\omega, \tau), \label{seq: limit of threshold 3}
		\end{align} 
	\end{subequations}
	where $c(\omega, \tau)$ can be expressed in terms of $\rho$, $\tau$ and $s_0(\omega)$.
\end{theorem}

\section{Performance upper bound}
\label{sec: performance upper bound}

In this section we prove Theorem \ref{the: upper bound for mean stationary overall utility}, which provides an upper bound for the mean stationary overall utility of a system where the service time distribution has a finite mean but is otherwise general. Recall from Section \ref{sub: optimization problem} that the optimum of \eqref{pr: optimal task assignments} is an upper bound for the mean stationary overall utility. Therefore, we only need to prove that the task assignment $q_n^*$ defined in \eqref{eq: definition of q_n^N} is an optimizer of \eqref{pr: optimal task assignments}.

We say that a sequence $q \in \R^\calI$ is eventually zero if there exists $k > 0$ such that $q(i, j) = 0$ for all $i$ and $j > k$. The following lemma implies that $q_n^*$ is an optimizer of \eqref{pr: optimal task assignments} if we impose the additional constraint that the solution must be eventually zero.

\begin{lemma}
	\label{lem: optimal task assignment under eventually zero constraint}
	If $q$ satisfies the constraints of \eqref{pr: optimal task assignments} and is eventually zero, then $u(q) \leq u(q_n^*)$.
\end{lemma}

\begin{proof}
	Since $q$ is eventually zero, it is possible to write
	\begin{equation}
	\label{eq: overall utility of eventually zero sequence}
	\begin{split}
	u(q) &= \sum_{i = 1}^\infty \sum_{j = 0}^\infty u_i(j)\left[q(i, j) - q(i, j + 1)\right] \\
	&= \sum_{i = 1}^m u_i(0)q(i, 0) + \sum_{i = 1}^m \sum_{j = 1}^\infty \Delta(i, j - 1) q(i, j) \\
	&= \sum_{i = 1}^m u_i(0)\alpha_n(i) + \sum_{(i, j) \in \calI_+} \Delta(i, j - 1)q(i, j).
	\end{split}
	\end{equation}
	In the last expression, the terms of the summation are ordered with respect to $\menor$, and in particular in non-increasing order of the marginal utilities $\Delta(i, j - 1)$. The task assignment $q_n^*$ is obtained by choosing the coefficients $q(i, j)$ so that the first coefficients are maximal, while all the coefficients add up to $\rho$. Thus, $q_n^*$ maximizes the right-hand side of \eqref{eq: overall utility of eventually zero sequence}.
\end{proof}

We now provide a solution of \eqref{pr: optimal task assignments}, without imposing any additional constraints.

\begin{proposition}
	\label{prop: optimal task assignments}
	The task assignment $q_n^*$ defined in \eqref{eq: definition of q_n^N} is an optimizer of \eqref{pr: optimal task assignments}.
\end{proposition}

\begin{proof}
	By Lemma \ref{lem: optimal task assignment under eventually zero constraint}, it suffices to prove that $u(q) \leq u(q_n^*)$ for each $q$ that is not eventually zero and satisfies the constraints of \eqref{pr: optimal task assignments}. Next we fix one such sequence $q$ and we construct an eventually zero sequence $z$ such that $u(q) \leq u(z)$ and $z$ satisfies the constraints of \eqref{pr: optimal task assignments}. Then $u(q) \leq u(z) \leq u(q_n^*)$ by Lemma \ref{lem: optimal task assignment under eventually zero constraint}, as desired.
	
	Choose $k \in \N$ such that $\alpha_n(i) k > \rho$ for all $i$. For each $i$, we define $z(i, j)$ iteratively, by
	\begin{equation*}
	z(i, j) \defeq \begin{cases}
	\min\left\{q(i, j) + \sum_{l = k + 1}^\infty q(i, j) - \sum_{l = 0}^{j - 1} \left[z(i, l) - q(i, l)\right], \alpha_n(i)\right\} & \quad \text{if} \quad j \leq k, \\
	0 & \quad \text{if} \quad j > k.
	\end{cases}
	\end{equation*}
	Informally, each coefficient $q(i, j)$ can be regarded as a container with capacity $\alpha_n(i)$, as shown in Figure \ref{fig: construction of z}. For each $i$, the sequence $q(i, \scdot)$ is transformed into $z(i, \scdot)$ in two steps: first we remove all the mass from the coefficients $q(i, j)$ with $j > k$, and then we place this mass on the coefficients $q(i, j)$ with $j \leq k$. In the latter step, we start with the first coefficient, placing as much mass as possible without exceeding the capacity $\alpha_n(i)$. The remainder of mass is placed in the following coefficients in the same fashion, and in increasing order of $j$. Observe that all the mass will have been placed right after the coefficient $q(i, k)$ is done since $q$ satisfies the constraints of \eqref{pr: optimal task assignments} and $\alpha_n(i)k > \rho$. Furthermore, the following property holds:
	\begin{equation}
	\label{aux: relationship between q and z}
	\sum_{j = 0}^k \left[z(i, j) - q(i, j)\right] = \sum_{j = k + 1}^\infty q(i, j) \quad \text{for all} \quad i.
	\end{equation}
	
	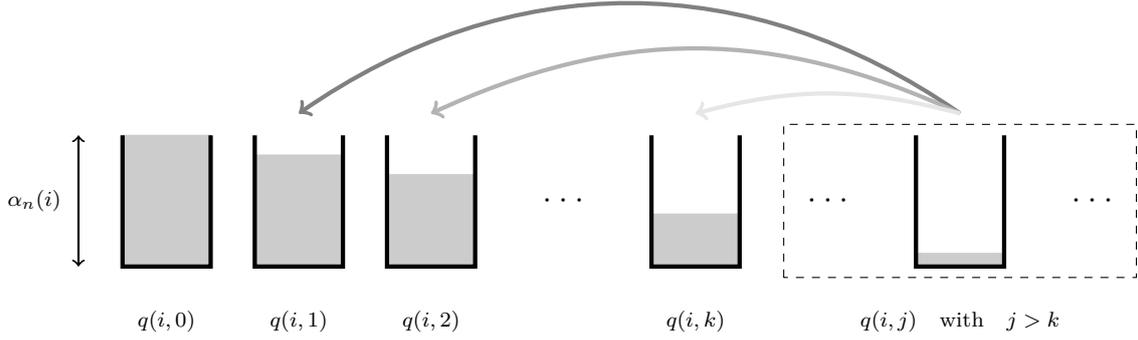
\begin{figure}
		\centering	
		\begin{tikzpicture}[x = 0.58cm, y = 0.58cm]
		
		%% Levels
		\foreach \x in {0, 3, 6, 12}
		\draw[color = black!20!white, fill = black!20!white] (\x, 0) rectangle (\x + 2, {3 * (20 - \x) / 20});
		
		\draw[color = black!20!white, fill = black!20!white] (18, 0) rectangle (20, 3 * 2 / 20);
		
		%% Containers
		\foreach \x in {0, 3, 6, 12}
		\draw[black, ultra thick] (\x, 3) -- (\x, 0) -- (\x + 2, 0) -- (\x + 2, 3);
		
		\node at (10, 1.5) {$\cdot\cdot\cdot$};
		\node at (16, 1.5) {$\cdot\cdot\cdot$};
		\draw[black, ultra thick] (18, 3) -- (18, 0) -- (20, 0) -- (20, 3);
		\node at (22, 1.5) {$\cdot\cdot\cdot$};
		\draw[black, dashed] (15, -0.25) rectangle (23, 3.25);
		
		%$ Labels
		\node at (1, - 1.25) {\scriptsize{$q(i, 0)$}};
		\node at (4, - 1.25) {\scriptsize{$q(i, 1)$}};
		\node at (7, - 1.25) {\scriptsize{$q(i, 2)$}};
		\node at (13, - 1.25) {\scriptsize{$q(i, k)$}};
		\node at (19, - 1.25) {\scriptsize{$q(i, j) \quad \text{with} \quad j > k$}};
		\node at (-2, 1.5) {\scriptsize{$\alpha_n(i)$}};
		
		%% Arrows
		\draw[black!50!white, ultra thick, ->] (19, 3.5) to[out = 145, in = 35] (4, 3.5);
		\draw[black!30!white, ultra thick, ->] (19, 3.5) to[out = 155, in = 25] (7, 3.5);
		\draw[black!10!white, ultra thick, ->] (19, 3.5) to[out = 165, in = 15] (13, 3.5);
		\draw[black, <->, thick] (-1, 0) -- (-1, 3);
		
		\end{tikzpicture}
		\caption{Schematic view of the construction of $z(i, \scdot)$ from $q(i, \scdot)$ for some fixed $i$.}
		\label{fig: construction of z}
	\end{figure}
	
	The overall utility of the task assignment $q$ satisfies:
	\begin{align*}
	u(q) &= \sum_{i = 1}^m \sum_{j = 0}^\infty u_i(j)\left[q(i, j) - q(i, j + 1)\right] \\
	&= \sum_{i = 1}^m\sum_{j = 0}^k u_i(j)\left[q(i, j) - q(i, j + 1)\right] + \sum_{i = 1}^m\sum_{j = k + 1}^\infty u_i(j)\left[q(i, j) - q(i, j + 1)\right] \\
	&\leq \sum_{i = 1}^m\sum_{j = 0}^k u_i(j)\left[q(i, j) - q(i, j + 1)\right] \\
	&+ \sum_{i = 1}^m\sum_{j = k + 1}^\infty \left[u_i(k) + (j - k)\Delta(i, k)\right]\left[q(i, j) - q(i, j + 1)\right] \\
	&= \sum_{i = 1}^m\sum_{j = 0}^k u_i(j)\left[q(i, j) - q(i, j + 1)\right] + \sum_{i = 1}^m u_i(k)q(i, k + 1) + \sum_{i = 1}^m \Delta(i, k) \sum_{j = k + 1}^\infty q(i, j) \\
	&= \sum_{i = 1}^m u_i(0)\alpha_n(i) + \sum_{i = 1}^m\sum_{j = 0}^k \Delta(i, j - 1)q(i, j) + \sum_{i = 1}^m \Delta(i, k) \sum_{j = k + 1}^\infty q(i, j) \\
	&= u(z) + \sum_{i = 1}^m\sum_{j = 0}^k \Delta(i, j - 1)\left[q(i, j) - z(i, j)\right] + \sum_{i = 1}^m \Delta(i, k) \sum_{j = k + 1}^\infty q(i, j).
	\end{align*}
	For the last step, recall that $z$ is eventually zero, so $u(z)$ can be computed as in \eqref{eq: overall utility of eventually zero sequence}. Note that $q(i, j) \leq z(i, j)$ and $\Delta(i, j - 1) \geq \Delta(i, k)$ if $j \leq k$. Therefore,
	\begin{align*}
	u(q) &\leq u(z) + \sum_{i = 1}^m \Delta(i, k) \sum_{j = 1}^k \left[q(i, j) - z(i, j)\right] + \sum_{i = 1}^m \Delta(i, k) \sum_{j = k + 1}^\infty q(i, j) = u(z) \leq u(q_n^*).
	\end{align*}
	The middle equality and last inequality follow from \eqref{aux: relationship between q and z} and Lemma \ref{lem: optimal task assignment under eventually zero constraint}, respectively.
\end{proof}

The proof of Theorem \ref{the: upper bound for mean stationary overall utility} follows easily from Proposition \ref{prop: optimal task assignments}.

\begin{proof}[Proof of Theorem \ref{the: upper bound for mean stationary overall utility}]
	As indicated at the end of Section \ref{sub: optimization problem}, the optimum of \eqref{pr: optimal task assignments} upper bounds the mean stationary overall utility $\bE\left[u(q_n)\right]$. Thus, it follows from Proposition \ref{prop: optimal task assignments} that $u\left(q_n^*\right)$ upper bounds the mean stationary overall utility.
\end{proof}

\section{Strong approximations}
\label{sec: strong approximations}

In this section we construct the processes defined in Section \ref{sub: stochastic models} on a common probability space for all $n$. In addition, we prove that $\set{\bq_n}{n \geq 1}$ is almost surely relatively compact in $D_{\ell_1}[0, \infty)$ both for JLMU and SLTA. The construction of the processes is carried out in Section \ref{sub: coupled construction of sample paths} and the relative compactness results are provided in Section \ref{sub: relative compactness results}.

\subsection{Coupled construction of sample paths}
\label{sub: coupled construction of sample paths}

Consider the following stochastic processes and random variables.
\begin{itemize}
	\item \textit{Driving Poisson processes:} a collection $\{\calN\} \cup \set{\calN_\nu}{\nu \in \calI_+}$ of independent Poisson processes with unit rate, for counting arrivals and departures. These processes are defined on a common probability space $(\Omega_D, \calF_D, \prob_D)$.
	
	\item \textit{Selection variables:} a family $\set{U_k}{k \geq 1}$ of independent random variables, uniformly distributed on $[0, 1)$ and defined on a common probability space $(\Omega_S, \calF_S, \prob_S)$.
	
	\item \textit{Initial conditions:} sequences $\set{\bq_n(0)}{n \geq 1}$, and also $\set{\br_n(0)}{n \geq 1}$ for SLTA, of random variables for describing the initial states of the systems, defined on a common probability space $(\Omega_I, \calF_I, \prob_I)$ and satisfying the assumptions of Section \ref{subsub: technical assumptions}.
\end{itemize}
Denote the completion of the product probability space of $(\Omega_D, \calF_D, \prob_D)$, $(\Omega_S, \calF_S, \prob_S)$ and $(\Omega_I, \calF_I, \prob_I)$ by $(\Omega, \calF, \prob)$. The processes introduced in Section \ref{sub: stochastic models} are constructed on the latter space as deterministic functions of the stochastic primitives.

\subsubsection{Construction for JLMU}
\label{subsub: construction for jlmu}

Let $\calN_n^\lambda(t) \defeq \calN(n \lambda t)$ for each $t \geq 0$ and each $n$. This quantity will be used to count the number of tasks arriving to the system with $n$ server pools during the interval $[0, t]$. Also, denote the jump times of $\calN_n^\lambda$ by $\set{\tau_{n, k}}{k \geq 1}$ and define $\tau_{n, 0} \defeq 0$. For each function $\map{\bq}{[0, \infty)}{Q_n}$ and each $n$, we define two counting processes, for arrivals and departures, denoted $\calA_n(\bq)$ and $\calD_n(\bq)$, respectively. The coordinates $(i, j) \in \calI$ of these processes are identically zero if $j = 0$, whereas the other coordinates are defined as follows:
\begin{subequations}
	\begin{align}
	&\calA_n(\bq, t, i, j) \defeq \frac{1}{n} \sum_{k = 1}^{\calN_n^\lambda(t)} \ind{(i, j) = \sigma\left(\bq\left(\tau_{n, k}^-\right)\right)}, \label{seq: arrival process for jlmu} \\
	&\calD_n(\bq, t, i, j) \defeq \frac{1}{n} \calN_{(i, j)}\left(n\int_0^t \mu j \left[\bq(s, i, j) - \bq(s, i, j + 1)\right]ds\right). \label{seq: departure process for jlmu}
	\end{align}
\end{subequations}

For each $n$, the functional equation
\begin{equation}
\label{eq: functional equation for jlmu}
\bq = \bq_n(0) + \calA_n(\bq) - \calD_n(\bq)
\end{equation}
has a unique solution with probability one. More precisely, there exists a set of probability one $\Gamma_0$ with the following property: for each $\omega \in \Gamma_0$ and each $n$, there exists a unique c\`adl\`ag function $\map{\bq_n(\omega)}{[0, \infty)}{Q_n}$ that solves \eqref{eq: functional equation for jlmu}. This solution can be constructed by forward induction on the jump times of the driving Poisson processes. The assumption $\bq_n(\omega, 0) \in Q_n$ implies that $\bq_n(\omega, 0)$ has finitely many non-zero coordinates and ensures that the constructed solution is defined on $[0, \infty)$ with probability one; i.e., the constructed solution does not explode in finite time.

The occupancy processes are defined by extending the above solutions to $\Omega$, setting $\bq_n(\omega, t) = 0$ for all $t \geq 0$ and all $\omega \notin \Gamma_0$. In addition, we let
\begin{equation*}
\calA_n \defeq \calA_n(\bq_n) \quad \text{and} \quad \calD_n \defeq \calD_n(\bq_n),
\end{equation*}
and we note that the sample paths of $\calA_n$, $\calD_n$ and $\bq_n$ lie in $D_{\ell_1}[0, \infty)$. The functional equation \eqref{eq: functional equation for jlmu} can now be rewritten as follows:
\begin{equation}
\label{eq: definition of q for jlmu}
\bq_n(\omega) = \bq_n(\omega, 0) + \calA_n(\omega) - \calD_n(\omega) \quad \text{for all} \quad \omega \in \Gamma_0.
\end{equation}

This construction endows the processes $\bq_n$ with the intended statistical behavior. The processes $\calA_n(i, j)$ count the arrivals to server pools of class $i$ with precisely $j - 1$ tasks and the processes $\calD_n(i, j)$ count the departures from server pools of class $i$ with exactly $j$ tasks. Indeed, $\calA_n(i, j)$ has a jump at the arrival epoch $\tau_{n, k}$ if and only if the incoming task should be assigned to a server pool of class $i$ with $j - 1$ tasks under the JLMU policy. In addition, the intensity of $\calD_n(i, j)$ equals the total number of tasks in server pools of class $i$ with exactly $j$ tasks times the rate at which tasks are executed, and this totals the departure rate from server pools of class $i$ with precisely $j$ tasks.

\subsubsection{Construction for SLTA}
\label{subsub: construction for threshold}

The processes $(\bq_n, \br_n)$ are constructed to a large extent as in Section \ref{subsub: construction for jlmu} when the load balancing policy is SLTA. The only differences arise in \eqref{seq: arrival process for jlmu} and \eqref{eq: functional equation for jlmu}. Specifically, \eqref{seq: arrival process for jlmu} has to be modified to capture the dispatching rule of SLTA and \eqref{eq: functional equation for jlmu} must be accompanied by another equation, for describing the evolution of $\br_n$.

The counterpart of \eqref{seq: arrival process for jlmu}, with an extra argument $\map{\br}{[0, \infty)}{\set{r \in \N}{r \geq 1}}$, is
\begin{equation*}
\calA_n(\bq, \br, t, i, j) \defeq \frac{1}{n} \sum_{k = 1}^{\calN_n^\lambda(t)} \eta_k\left(\bq\left(\tau_{n, k}^-\right), \br\left(\tau_{n, k}^-\right), i, j\right) \quad \text{for all} \quad (i, j) \in \calI_+.
\end{equation*}
The functions $\eta_k$ are defined in Appendix \ref{app: auxiliary results} using the selection variables $U_k$, so that they have the following property. If $(q, r)$ is the value of $(\bq_n, \br_n)$ when the $k^{\text{th}}$ task arrives, then SLTA sends this task to a server pool of class $i$ with precisely $j - 1$ tasks if and only if $\eta_k(q, r, i, j) = 1$. Moreover, $\eta_k(q, r, h, l) = 0$ for all $(h, l) \neq (i, j)$.

The analog of the functional equation \eqref{eq: functional equation for jlmu} is
\begin{subequations}
	\label{eq: functional equation for threshold}
	\begin{align}
	&\bq(t) = \bq_n(0) + \calA_n(\bq, \br, t) - \calD_n(\bq, t), \label{seq: functional equation for q for threshold} \\
	&\br(t) = \br_n(0) + \sum_{k = 1}^{\calN_n^\lambda(t)} \left[\indc_{I_{n, k}} - \indc_{D_{n, k}}\right], \label{seq: functional equation for r for threshold}
	\end{align}
\end{subequations}
where the sets $I_{n, k}$ and $D_{n, k}$ are defined formally in Appendix \ref{app: auxiliary results}. The former set indicates that $(\bq, \br)$ corresponds to a system with no green tokens and at most one yellow token right before the $k^{\text{th}}$ arrival. The latter set indicates that the number of green tokens is larger than or equal to $n \beta_n$ and that at least one of these tokens belongs to a server pool of class $i_{\br - 1}$ right before the $k^{\text{th}}$ arrival. Also, $\calD_n(\bq)$ is defined as in \eqref{seq: departure process for jlmu}.

As in Section \ref{subsub: construction for jlmu}, there exists a set of probability one $\Gamma_0$ with the next property. For each $\omega \in \Gamma_0$ and each $n$, there exists a unique  pair of c\`adl\`ag functions $\map{\bq_n(\omega)}{[0, \infty)}{Q_n}$ and $\map{\br_n(\omega)}{[0, \infty)}{\set{{r \in \N}}{r \geq 1}}$ that solve \eqref{eq: functional equation for threshold}; these functions can be constructed by forward induction on the jumps of the driving Poisson processes. The processes $(\bq_n, \br_n)$ are defined by extending the above solutions to $\Omega$, setting $\bq_n(\omega, t) = 0$ and $\br_n(\omega, t) = 0$ for all $t \geq 0$ and all $\omega \notin \Gamma_0$. In addition, we define
\begin{equation*}
\calA_n \defeq \calA_n(\bq_n, \br_n) \quad \text{and} \quad \calD_n \defeq \calD_n(\bq_n),
\end{equation*}
and we note that the sample paths of $\calA_n$, $\calD_n$ and $\bq_n$ lie in $D_{\ell_1}[0, \infty)$. The functional equation \eqref{eq: functional equation for threshold} can now be rewritten as follows:
\begin{subequations}
	\begin{align}
	&\bq_n(\omega, t) = \bq_n(\omega, 0) + \calA_n(\omega, t) - \calD_n(\omega, t), \label{seq: definition of q for threshold} \\
	&\br_n(\omega, t) = \br_n(\omega, 0) + \sum_{k = 1}^{\calN_n^\lambda(\omega, t)} \left[\indc_{I_{n, k}}(\omega) - \indc_{D_{n, k}}(\omega)\right], \label{eq: definition of r for threhsold}
	\end{align}
\end{subequations}
for all $\omega \in \Gamma_0$ and all $t \geq 0$.

\subsection{Relative compactness results}
\label{sub: relative compactness results}

Let $D_{\R^\calI}[0, \infty)$ denote the space of c\`adl\`ag functions on $[0, \infty)$ with values in $\R^\calI$. We endow the space $\R^\calI$ with the metric defined in Appendix \ref{app: relative compactness}, which is compatible with the product topology, and we equip $D_{\R^\calI}[0, \infty)$ with the topology of uniform convergence over compact sets. The following proposition is proved in Appendix \ref{app: relative compactness}.

\begin{proposition}
	\label{prop: relative compactness - product topology}
	Suppose that the load balancing policy is JLMU or SLTA. There exists a set of probability one $\Gamma_\infty$ where:
	\begin{subequations}
		\label{eq: properties within gamma infty}
		\begin{align}
		&\lim_{n \to \infty} \norm{\bq_n(0) - q_0}_1 = 0, \label{seq: convergence of initial conditions} \\
		&\lim_{n \to \infty} \sup_{t \in [0, T]} n^\gamma\left|\frac{1}{n}\calN_n^\lambda(t) - \lambda t\right| = 0, \label{seq: fslln for arrivals} \\
		&\lim_{n \to \infty} \sup_{t \in [0, \mu jT]} n^\gamma\left|\frac{1}{n}\calN_{(i, j)}(n t) - t\right| = 0 \quad \text{for all} \quad (i, j) \in \calI_+, \label{seq: fslln for departures}
		\end{align}
	\end{subequations}
	for all $T \geq 0$ and $\gamma \in [0, 1 / 2)$. Also, $\set{\calA_n(\omega)}{n \geq 1}$, $\set{\calD_n(\omega)}{n \geq 1}$ and $\set{\bq_n(\omega)}{n \geq 1}$ are relatively compact subsets of $D_{\R^\calI}[0, \infty)$ for all $\omega \in \Gamma_\infty$ and satisfy that the limit of every convergent subsequence is a function with locally Lipschitz coordinates. If the load balancing policy is SLTA, then there exists a random variable $R \geq 1$ such that, apart from the above properties, we also have on $\Gamma_\infty$ that:
	\begin{equation}
	\label{eq: additional property of threshold within gamma infty}
	\br_n(0) \leq R \quad \text{and} \quad \bq_n(0, i, j) < \alpha_n(i) \quad \text{for all} \quad (i, j) \menorigual \left(i_{\br_n(0)}, j_{\br_n(0)}\right) \quad \text{and} \quad n.
	\end{equation}
\end{proposition}

The product topology of $\R^\calI$ is coarser than the topology of $\ell_1$, thus convergence in $D_{\R^\calI}[0, \infty)$ does not imply convergence in $D_{\ell_1}[0, \infty)$. The following technical lemma is used to demonstrate that $\set{\calA_n}{n \geq 1}$, $\set{\calD_n}{n \geq 1}$ and $\set{\bq_n}{n \geq 1}$ are relatively compact in $D_{\ell_1}[0, \infty)$ with probability one; the proof is given in Appendix \ref{app: auxiliary results}.

\begin{lemma}
	\label{lem: set of nice sample paths}
	Suppose that the load balancing policy is JLMU or SLTA. There exists a set of probability one $\Gamma \subset \Gamma_\infty$ with the following property. For each $\omega \in \Gamma$ and $T \geq 0$, there exist $j_T(\omega)$ and $n_T(\omega)$ such that
	\begin{equation*}
	\calA_n(\omega, t, i, j) = 0 \quad \text{for all} \quad t \in [0, T], \quad 1 \leq i \leq m, \quad j > j_T(\omega) \quad \text{and} \quad n \geq n_T(\omega).
	\end{equation*}
	Also, if the load balancing policy is SLTA, then there exists $R_T(\omega)$ such that
	\begin{equation*}
	\br_n(\omega, t) \leq R_T(\omega) \quad \text{for all} \quad t \in [0, T] \quad \text{and} \quad n \geq n_T(\omega).
	\end{equation*}
\end{lemma}

\begin{proposition}
	\label{prop: relative compactness - l1}
	The sequences $\set{\calA_n(\omega)}{n \geq 1}$, $\set{\calD_n(\omega)}{n \geq 1}$ and $\set{\bq_n(\omega)}{n \geq 1}$ are relatively compact in $D_{\ell_1}[0, \infty)$ for all $\omega \in \Gamma$.
\end{proposition}

\begin{proof}
	We fix some $\omega \in \Gamma$ which we omit from the notation. For $\set{\calA_n}{n \geq 1}$, the claim is a straightforward consequence of Proposition \ref{prop: relative compactness - product topology} and Lemma~\ref{lem: set of nice sample paths}. Below we prove the claim for $\set{\bq_n}{n \geq 1}$. If the load balancing policy is JLMU, then \eqref{eq: definition of q for jlmu} and \eqref{seq: convergence of initial conditions} imply that the claim also holds for $\set{\calD_n}{n \geq 1}$. If the load balancing policy is SLTA, then we must invoke \eqref{seq: definition of q for threshold} instead of \eqref{eq: definition of q for jlmu}.
	
	Consider any increasing sequence of natural numbers. By Proposition \ref{prop: relative compactness - product topology}, there exists a subsequence $\calK$ such that $\set{\bq_k}{k \in \calK}$ converges in $D_{\R^\calI}[0, \infty)$ to a function $\bq \in D_{\R^\calI}[0, \infty)$ that satisfies $\bq(0) = q_0$ and has locally Lipschitz coordinates. Therefore, it suffices to prove that the latter limit in fact holds in $D_{\ell_1}[0, \infty)$. More specifically, we have to demonstrate that $\bq \in D_{\ell_1}[0, \infty)$ and that
	\begin{equation*}
	\lim_{k \to \infty} \sup_{t \in [0, T]} \norm{\bq_k(t) - \bq(t)}_1 = 0 \quad \text{for all} \quad T \geq 0.
	\end{equation*}
	
	For this purpose, fix arbitrary $T \geq 0$ and $\varepsilon > 0$. In addition, let $j_T$ and $n_T$ be as in the statement of Lemma \ref{lem: set of nice sample paths}, which implies that $\bq_k(i, j)$ and $\bq(i, j)$ are non-increasing on $[0, T]$ provided that $j > j_T$ and $k \geq n_T$. The coordinates of $\bq$ are continuous, thus we may conclude from the monotone convergence theorem that $\norm{\bq(s) - \bq(t)}_1 \to 0$ as $s \to t \in [0, T]$ monotonically, from above or below. Since $T$ is arbitrary, $\bq$ is continuous with respect to $\norm{\scdot}_1$ and, in particular, $\bq \in D_{\ell_1}[0, \infty)$.
	
	For all $t \in [0, T]$, $l \geq j_T$ and $k \geq n_T$, we have
	\begin{align*}
	\sum_{i = 1}^m \sum_{j > l} \left|\bq_k(t, i, j) - \bq(t, i, j)\right| &\leq \sum_{i = 1}^m \sum_{j > l} \bq_k(t, i, j) + \sum_{i = 1}^m \sum_{j > l} \bq(t, i, j) \\
	&\leq \sum_{i = 1}^m \sum_{j > l} \bq_k(0, i, j) + \sum_{i = 1}^m \sum_{j > l} q_0(i, j) \\
	&\leq \norm{\bq_k(0) - q_0}_1 + 2\sum_{i = 1}^m \sum_{j > l} q_0(i, j).
	\end{align*}
	Since $q_0 \in \ell_1$ and $\norm{\bq_k(0) - q_0}_1 \to 0$ with $k$, there exist $j_\varepsilon \geq j_T$ and $k_0 \geq n_T$ such that the following inequality holds for all $t \in [0, T]$ and $k \geq k_0$:
	\begin{equation*}
	\sum_{i = 1}^m \sum_{j > j_\varepsilon} \left|\bq_k(t, i, j) - \bq(t, i, j)\right| \leq \norm{\bq_k(0) - q_0}_1 + 2\sum_{i = 1}^m \sum_{j > j_\varepsilon} q_0(i, j) \leq \frac{\varepsilon}{2}.
	\end{equation*}
	
	Note that convergence in $D_{\R^\calI}[0, \infty)$ implies uniform convergence over compact sets of the coordinate functions. In particular, there exists $k_\varepsilon \geq k_0$ such that
	\begin{equation*}
	\sup_{t \in [0, T]} \sum_{i = 1}^m \sum_{j = 0}^{j_\varepsilon} \left|\bq_k(t, i, j) - \bq(t, i, j)\right| \leq \frac{\varepsilon}{2} \quad \text{for all} \quad k \geq k_\varepsilon. 
	\end{equation*}
	Therefore, we conclude that
	\begin{equation*}
	\sup_{t \in [0, T]} \norm{\bq_k(t) - \bq(t)}_1 \leq \varepsilon \quad \text{for all} \quad k \geq k_\varepsilon,
	\end{equation*}
	which completes the proof since $T$ and $\varepsilon$ are arbitrary.
\end{proof}

\section{Limiting behavior of JLMU}
\label{sec: limiting behavior of jlmu}

In this section we assume that JLMU is used and we prove Theorems \ref{the: uniqueness of fluid trajectories}, \ref{the: global asymptotic stability} and \ref{the: fluid limit of jlmu}. The first two theorems are proved in Section \ref{sub: properties of fluid trajectories}, which is devoted to the study of fluid trajectories. The proof of Theorem \ref{the: fluid limit of jlmu} is provided in Section \ref{sub: proof of the fluid limit}.

\subsection{Properties of fluid trajectories}
\label{sub: properties of fluid trajectories}

In order to prove the uniqueness of fluid trajectories, we show that every fluid trajectory satisfies a system of equations involving one-dimensional Skorokhod reflection mappings. Then we use a Lipschitz property of these mappings to prove that the system of equations cannot have multiple solutions for a given initial condition. The proof strategy is inspired by a fluid limit derived in \cite{bhamidi2022near} using Skorokhod reflection mappings; this fluid limit corresponds to a system of parallel single-server queues with a JSQ policy.

Consider the space $D[0, \infty)$ of all real c\`adl\`ag functions defined on $[0, \infty)$ and let
\begin{equation*}
\norm{\bx}_T \defeq \sup_{t \in [0, T]} |\bx(t)| \quad \text{for all} \quad \bx \in D[0, \infty) \quad \text{and} \quad T \geq 0.
\end{equation*}
The next lemma introduces the one-dimensional Skorokhod mappings with upper reflecting barrier; a proof is provided in Appendix \ref{app: auxiliary results}.

\begin{lemma}
	\label{lem: skorokhod mappings with upper reflecting barrier}
	Fix $\alpha \in \R$ and suppose that $\bx \in D[0, \infty)$ is such that $\bx(0) \leq \alpha$. Then there exist unique $\by, \bz \in D[0, \infty)$ such that the following statements hold.
	\begin{enumerate}
		\item[(a)] $\bz(t) = \bx(t) - \by(t) \leq \alpha$ for all $t \geq 0$.
		
		\item[(b)] $\by$ is non-decreasing, thus almost everywhere differentiable, and $\by(0) = 0$.
		
		\item[(c)] $\by$ is flat off $\set{t \geq 0}{\bz(t) = \alpha}$, i.e., $\dot{\by}(t) \ind{\bz(t) < \alpha} = 0$ almost everywhere.
	\end{enumerate}
	The map $(\Psi_\alpha, \Phi_\alpha)$ such that $\Psi_\alpha(\bx) = \by$ and $\Phi_\alpha(\bx) = \bz$ is called the one-dimensional Skorokhod mapping with upper reflecting barrier at $\alpha$ and satisfies
	\begin{equation}
	\label{eq: explicit expressions for reflection mappings}
	\Psi_\alpha(\bx)(t) = \sup_{s \in [0, t]} \left[\bx(s) - \alpha\right]^+ \quad \text{and} \quad \Phi_\alpha(\bx)(t) = \bx(t) - \Psi_\alpha(\bx)(t).
	\end{equation}
	In addition, if $\bx, \by \in D[0, \infty)$ are any two functions such that $\bx(0), \by(0) \leq \alpha$, then for each $T \geq 0$ we have the following Lipschitz properties:
	\begin{equation*}
	\norm{\Psi_\alpha(\bx) - \Psi_\alpha(\by)}_T \leq \norm{\bx - \by}_T \quad \text{and} \quad \norm{\Phi_\alpha(\bx) - \Phi_\alpha(\by)}_T \leq 2 \norm{\bx - \by}_T.
	\end{equation*} 
\end{lemma}

Consider c\`adl\`ag functions $\map{\bx}{[0, \infty)}{\R^\calI}$ such that $\bx(0, i, j) \leq \alpha(i)$ for all $(i, j) \in \calI_+$ and families of c\`adl\`ag functions $\map{\bv_k}{[0, \infty)}{\R}$ such that $\bv_k(0) = 0$ for all $k \in \N$. Define for each $k \geq 1$ a mapping $\Theta_k$ as follows:
\begin{equation*}
\Theta_k(\bx, \bv)(t) \defeq \bx(0, i_k, j_k) + \bv_{k - 1}(t) - \int_0^t \mu j_k \left[\bx(s, i_k, j_k) - \bx(s, i_k, j_k + 1)\right]ds;
\end{equation*}
here recall the enumeration of $\calI_+$ introduced in Section \ref{sub: threshold}. The next lemma establishes that $(\bq, \bw)$ satisfies the following set of conditions if $\bq$ is a fluid trajectory and $\bw$ is defined suitably in terms of $\bq$.
\begin{subequations}
	\label{eq: system of reflection equations}
	\begin{align}
	&\bv_k = \Psi_{\alpha(i_k)}\left[\Theta_k(\bx, \bv)\right] \quad \text{for all} \quad k \geq 1, \label{seq: system of reflection equations 1}\\
	&\bx(i_k, j_k) = \Phi_{\alpha(i_k)} \left[\Theta_k(\bx, \bv)\right] \quad \text{for all} \quad k \geq 1, \label{seq: system of reflection equations 2}\\
	&\bv_0(t) = \lambda t \quad \text{for all} \quad t \geq 0, \label{seq: system of reflection equations 3}\\
	&\bx(t, i, 0) = \alpha(i) \quad \text{for all} \quad t \geq 0 \quad \text{and} \quad i. \label{seq: system of reflection equations 4}
	\end{align}
\end{subequations}

\begin{lemma}
	\label{lem: system of reflection equations}
	Let $\bq$ be a fluid trajectory and define $\br$ such that
	\begin{equation*}
	\sigma\left(\bq(t)\right) = \left(i_{\br(t)}, j_{\br(t)}\right) \quad \text{for all} \quad t \geq 0.
	\end{equation*}
	Also, consider the absolutely continuous functions $\bw_k$ such that $\bw_k(0) = 0$ and
	\begin{equation*}
	\dot{\bw}_k(t) = \left[\lambda - \sum_{s = 1}^k \mu j_s \left[\alpha(i_s) - \bq(t, i_s, j_s + 1)\right]\right]\ind{k < \br(t)} \quad \text{for all} \quad k \geq 0.
	\end{equation*}
	Then $(\bq, \bw)$ satisfies \eqref{eq: system of reflection equations}.
\end{lemma}

\begin{proof}
	It is clear that $(\bq, \bw)$ satisfies \eqref{seq: system of reflection equations 3} and \eqref{seq: system of reflection equations 4}, so we only need to verify that \eqref{seq: system of reflection equations 1} and \eqref{seq: system of reflection equations 2} hold as well. By Lemma \ref{lem: skorokhod mappings with upper reflecting barrier}, it is enough to check the following properties.
	\begin{enumerate}
		\item[(a)] $\bq(t, i_k, j_k) = \Theta_k(\bq, \bw)(t) - \bw_k(t) \leq \alpha(i_k)$ for all $t \geq 0$ and $k \geq 1$.
		
		\item[(b)] $\bw_k(0) = 0$ and $\dot{\bw}_k(t) \geq 0$ almost everywhere for all $k \geq 1$.
		
		\item[(c)] $\dot{\bw}_k(t) \ind{\bq(t, i_k, j_k) < \alpha(i_k)} = 0$ almost everywhere for all $k \geq 1$.
	\end{enumerate}
	
	In order to establish (a), it suffices to show that
	\begin{equation}
	\label{eq: relation between derivatives of q and v}
	\dot{\bq}(i_k, j_k) = \dot{\Theta}_k(\bq, \bw) - \dot{\bw}_k = \dot{\bw}_{k - 1} - \dot{\bw}_k - \mu j_k \left[\bq(i_k, j_k) - \bq(i_k, j_k + 1)\right]
	\end{equation}
	holds almost everywhere. Indeed, (a) holds at $t = 0$ and $\bq(i_k, j_k) \leq \alpha(i_k)$ at all times. Since $\dot{\bw}_{k - 1} - \dot{\bw}_k = \Lambda(\bq, i_k, j_k)$ and $\bq$ is a fluid trajectory, we obtain \eqref{eq: relation between derivatives of q and v} from \eqref{seq: fluid dynamics 1}.
	
	Property (b) is a consequence of the definition of $\bw_k$ and \eqref{seq: fluid dynamics 2}, and (c) follows from the following observation. If $\bq(t, i_k, j_k) < \alpha(i_k)$, then $(i_k, j_k) \menorigual \sigma\left(\bq(t)\right)$ and thus $k \geq \br(t)$, which implies that $\dot{\bw}_k(t) = 0$.
\end{proof}

Next we prove Theorem \ref{the: uniqueness of fluid trajectories}. As noted earlier, the proof relies on the Lipschitz property of the Skorokhod reflection mapppings $\Psi_\alpha$ and $\Phi_\alpha$. In addition, a uniqueness of solutions result for certain Kolmogorov backward equations is used.

\begin{proof}[Proof of Theorem \ref{the: uniqueness of fluid trajectories}]
	Suppose that there exist two fluid trajectories $\bx$ and $\by$ such that $\bx(0) = \by(0) = q$. Define $\bv$ in terms of $\bx$ and $\bw$ in terms of $\by$ as in Lemma \ref{lem: system of reflection equations}. It follows from the same lemma that $(\bx, \bv)$ and $(\by, \bw)$ satisfy \eqref{eq: system of reflection equations}. Next we fix $T > 0$ and we prove that $\bx(t) = \by(t)$ and $\bv(t) = \bw(t)$ for all $t \in [0, T]$.
	
	As a first step, we demonstrate that there exists $M > 0$ such that
	\begin{equation}
	\label{eq: bv_k and bw_k are zero for large k}
	\bv_k(t) = \bw_k(t) = 0 \quad \text{for all} \quad t \in [0, T] \quad \text{and} \quad k \geq M.
	\end{equation}
	Since $\dot{\bv}_{k - 1} - \dot{\bv}_k = \Lambda(\bx, i_k, j_k) \geq 0$, we conclude that $\dot{\bv}_k \leq \dot{\bv}_{k - 1} \leq \dot{\bv}_0 = \lambda$ almost everywhere and for all $k \geq 1$. It follows from \eqref{seq: fluid dynamics 1}, or equivalently from \eqref{seq: system of reflection equations 1} and \eqref{seq: system of reflection equations 2}, that the following inequalities hold almost everywhere:
	\begin{equation}
	\label{eq: upper bound for derivative of bx}
	\dot{\bx}(i_k, j_k) \leq \dot{\bv}_{k - 1} - \dot{\bv}_k \leq \dot{\bv}_{k - 1} \leq \lambda \quad \text{for all} \quad k \geq 1.
	\end{equation}
	Define $\alpha_{\min} \defeq \min \set{\alpha(i)}{1 \leq i \leq m}$ and note that there exist $\varepsilon > 0$ and $k_0$ such that $q(i_k, j_k) \leq \varepsilon < \alpha_{\min}$ for all $k \geq k_0$ since $q \in Q \subset \ell_1$. Property (c) of Lemma \ref{lem: skorokhod mappings with upper reflecting barrier} implies that $\bv_k$ is zero until $\bx(i_k, j_k)$ reaches $\alpha(i_k) \geq \alpha_{\min}$ for the first time. Using this remark and \eqref{eq: upper bound for derivative of bx}, it is possible to prove by induction on $k \geq k_0$ that
	\begin{equation*}
	\bv_k(t) = 0 \quad \text{for all} \quad t \in \left[0, (k + 1 - k_0)\frac{\alpha_{\min} - \varepsilon}{\lambda}\right).
	\end{equation*}
	The same property holds if $(\bx, \bv)$ is replaced by $(\by, \bw)$, thus \eqref{eq: bv_k and bw_k are zero for large k} holds.
	
	Next we show that $\bx(t, i, j) = \by(t, i, j)$ for all $t \in [0, T]$ and $(i, j) \menor (i_M, j_M)$. For this purpose, fix an arbitrary $i$ and let $J_i \defeq \min \set{j \geq 1}{(i, j) \menor (i_M, j_M)}$. By \eqref{eq: bv_k and bw_k are zero for large k}, both $\bx(i)$ and $\by(i)$ satisfy the following initial value problem:
	\begin{equation}
	\label{eq: backward equations of a pure birth process}
	\dot{\bz}(j) = -\mu j \left[\bz(j) - \bz(j + 1)\right] \quad \text{and} \quad \bz(0, j) = q(i, j) \quad \text{for all} \quad j \geq J_i.
	\end{equation}
	The above system of differential equations are the backward Kolmogorov equations of the pure birth process with state space $E_i \defeq \set{j \in \N}{j \geq J_i}$ that has birth rate $\lambda_j \defeq \mu j$ at state $j$. This process is non-explosive since $\sum_{j \geq J_i} 1 / \lambda_j = \infty$, hence it follows from \cite{fontes1989note} that the initial value problem \eqref{eq: backward equations of a pure birth process} has a unique solution $\bz$ such that $\bz$ is bounded on $[0, t] \times E_i$ for all $t \geq 0$. Both $\bx(i)$ and $\by(i)$ satisfy the latter condition since fluid trajectories take values in $Q$, thus $\bx(t, i, j) = \by(t, i, j)$ for all $t \geq 0$ and $j \geq J_i$.
	
	We conclude by proving that $\bx(i_k, j_k) = \by(i_k, j_k)$ and $\bv_k = \bw_k$ along the interval $[0, T]$ for all $k \leq M$. Let $\boldf(i, j) \defeq \bx(i, j) - \by(i, j)$ and $\bg_k \defeq \bv_k - \bw_k$. The subsequent arguments are analogous to those in \cite[Section 4.1]{bhamidi2022near}.
	
	For all $t \in [0, T]$ and $k \leq M$, we have
	\begin{align*}
	&\norm{\boldf(i_k, j_k)}_t \leq 2 \norm{\bg_{k - 1}}_t + 2 \mu j_k \int_0^t \norm{\boldf(i_k, j_k)}_s ds + 2 \mu j_k \int_0^t \norm{\boldf(i_k, j_k + 1)}_s ds, \\
	&\norm{\bg_k}_t \leq \norm{\bg_{k - 1}}_t + \mu j_k \int_0^t \norm{\boldf(i_k, j_k)}_s ds + \mu j_k \int_0^t \norm{\boldf(i_k, j_k + 1)}_s ds;
	\end{align*}
	these inequalities follow from the Lipschitz properties of $\Psi_{\alpha(i_k)}$ and $\Phi_{\alpha(i_k)}$. Let us define $J \defeq \max\set{j_k}{k \leq M}$ and $\bh(t) \defeq \max\set{\norm{\boldf(i_k, j_k)}_t}{k \leq M}$, then
	\begin{equation*}
	\norm{\bg_k}_t \leq \norm{\bg_{k - 1}}_t + 2 \mu J \int_0^t \bh(s)ds \leq 2k \mu J \int_0^t \bh(s)ds
	\end{equation*}
	for all $t \in [0, T]$ and $k \leq M$. For the first inequality, note that $\boldf(i_k, j_k + 1)$ is identically zero along the interval $[0, T]$ if $(i_k, j_k + 1) = (i_l, j_l)$ for some $l > M$, and for the last inequality observe that $\bg_0$ is identically zero by \eqref{seq: system of reflection equations 3}. In addition, we have
	\begin{equation*}
	\norm{\boldf(i_k, j_k)}_t \leq 4k \mu J \int_0^t \bh(s)ds + 4 \mu J \int_0^t \bh(s)ds \leq 4(M + 1)\mu J \int_0^t \bh(s)ds
	\end{equation*}
	for all $t \in [0, T]$ and $k \leq M$. We conclude that
	\begin{equation*}
	\bh(t) \leq 4(M + 1)\mu J \int_0^t \bh(s)ds \quad \text{for all} \quad t \in [0, T].
	\end{equation*}
	Therefore, Gr\"onwall's inequality yields $\bh(t) = 0$ for all $t \in [0, T]$ and this in turn implies that $(\bx, \bv) = (\by, \bw)$ along the interval $[0, T]$.
\end{proof}

We conclude this section by establishing that \eqref{eq: fluid dynamics} has a unique equilibrium point and that all fluid trajectories converge to this equilibrium point over time.

\begin{proof}[Proof of Theorem \ref{the: global asymptotic stability}]
	First we verify that $q_*$ is an equilibrium of \eqref{eq: fluid dynamics}. To this end, note that $\sigma\left(q_*\right) = \sigma_*$ and that the right-hand side of \eqref{seq: fluid dynamics 1} equals zero if $(i, j) \neq \sigma_*$ and $\bq = q_*$. It only remains to be shown that this also holds for $(i_*, j_*) \defeq \sigma_*$.
	
	If $(i, j) = \sigma(q)$, then
	\begin{equation*}
		\Lambda(q, i, j) - \mu j \left[q(i, j) - q(i, j + 1)\right] = \lambda - \sum_{(r, s) \mayorigual (i, j)} \mu s \left[q(r , s) - q(r, s + 1)\right].
	\end{equation*}
	Define $J^i \defeq \max \set{j \geq 1}{(i, j) \mayorigual \sigma_*}$ for each $i$. If $(i, j) = (i_*, j_*)$ and $\bq$ is replaced by $q_*$, then the right-hand side of \eqref{seq: fluid dynamics 1} equals
	\begin{align*}
		\lambda - \sum_{(i, j) \mayorigual \sigma_*} \mu j \left[q_*(i, j) - q_*(i, j + 1)\right] &= \lambda - \sum_{i = 1}^m \sum_{j = 1}^{J^i} \mu j \left[q_*(i, j) - q_*(i, j + 1)\right] \\
		&= \lambda - \sum_{i \neq i_*} \mu J^i \alpha(i) - \mu(j_* - 1) \alpha(i_*) - \mu q_*(i_*, j_*) \\
		&= \sum_{(i, j) \mayor \sigma_*} \mu \alpha(i) - \sum_{i \neq i_*} \mu J^i \alpha(i) - \mu (j_* - 1) \alpha(i_*).
	\end{align*}
	The expression in the last line equals zero by definition of $J^i$ and therefore we conclude that $q_*$ is indeed an equilibrium point of \eqref{eq: fluid dynamics}.
	
	Next we prove that all fluid trajectories converge coordinatewise to $q_*$ over time; this implies, in particular, that $q_*$ is the unique equilibrium of \eqref{eq: fluid dynamics}. Afterwards we prove that all fluid trajectories in fact converge to $q_*$ in $\ell_1$.
	
	If $\bq$ is a fluid trajectory and $T \geq 0$, then there exists $j_T$ such that $\Lambda(\bq, i, j) = 0$ in $[0, T]$ for all $i$ and $j \geq j_T$. This can be established directly from \eqref{eq: fluid dynamics}, but also using Lemma \ref{lem: set of nice sample paths}, Theorem \ref{the: fluid limit of jlmu} and the uniqueness of solutions. For each $i$ and $k \geq j_T$,
	\begin{equation*}
		\sum_{j = j_T}^k \dot{\bq}(i, j) = - \sum_{j = j_T}^k \mu j \left[\bq(i, j) - \bq(i, j + 1)\right] = -\mu(j_T - 1)\bq(i, j_T) - \sum_{j = j_T}^k \bq(i, j) + \mu k\bq(i, k + 1). 
	\end{equation*}
	The last term vanishes as $k \to \infty$ because $\bq$ takes values in $\ell_1$. Also, $\bq(i, j)$ does not increase in $[0, T]$ for all $j \geq j_T$ since $\Lambda(\bq, i, j) = 0$ along $[0, T]$. Therefore, 
	\begin{align*}
		\left|\sum_{j = j_T}^k \dot{\bq}(t, i, j) - \sum_{j = j_T}^\infty \dot{\bq}(t, i, j)\right| &= \sum_{j = k + 1}^\infty \bq(t, i, j) + \mu k\bq(t, i, k + 1) \\
		&\leq \sum_{j = k + 1}^\infty \bq(0, i, j) + \mu k\bq(0, i, k + 1) \quad \text{for all} \quad t \in [0, T].
	\end{align*}
	The right-hand side vanishes as $k \to \infty$ because $\bq(0) \in \ell_1$, thus the left-hand side converges uniformly to zero over $[0, T]$ and, by \cite[Theorem 7.17]{rudin1976principles}, the derivative of
	\begin{equation*}
		\sum_{j = j_T}^\infty \bq(t, i, j) \quad \text{is} \quad \sum_{j = j_T}^\infty \dot{\bq}(t, i, j) \quad \text{for all} \quad i \quad \text{and} \quad t \in (0, T).
	\end{equation*}
	This allows for the interchanges of summation and differentation that appear below.
	
	Consider the function
	\begin{equation*}
		\boldf \defeq \sum_{(i, j) \in \calI_+} \bq(i, j).
	\end{equation*}
	It follows from \eqref{seq: fluid dynamics 1} that $\dot{\boldf} = \lambda - \mu\boldf$. Thus, $\boldf(t) = \rho + \left[\boldf(0) - \rho\right]\e^{-\mu t}$ for all $t \geq 0$. We conclude from this identity and \eqref{eq: definition of sigma*} that
	\begin{equation*}
		t_0 \defeq \min \set{t \geq 0}{\boldf(t) \leq \sum_{(i, j) \mayorigual \sigma_*} \alpha(i)}
	\end{equation*}
	exists and is finite. If $t > t_0$, then the inequality inside the minimum sign is strict, and this implies that $\sigma\left(\bq(t)\right) \mayorigual \sigma_*$ since $\bq(t, i, j) = \alpha(i)$ for all $(i, j) \mayor \sigma\left(\bq(t)\right)$.
	
	Consider the function
	\begin{equation*}
		\bg \defeq \sum_{(i, j) \menorigual \sigma_*} \bq(i, j).
	\end{equation*}
	As noted above, if $t > t_0$, then $\sigma\left(\bq(t)\right) \mayorigual \sigma_*$, and thus we have
	\begin{align*}
		\dot{\bg}(t) &= \left[\lambda - \sum_{(i, j) \mayor \sigma_*} \mu j \left[\bq(t, i, j) - \bq(t, i, j + 1)\right]\right] \ind{\sigma(\bq(t)) = \sigma_*} \\
		&- \sum_{(i, j) \menorigual \sigma_*} \mu j \left[\bq(t, i, j) - \bq(t, i, j + 1)\right],
	\end{align*} 
	since $\Lambda(\bq(t), i, j) = 0$ for all $(i, j) \menorigual \sigma_*$ except perhaps for $(i, j) = \sigma_*$. Hence,
	\begin{align*}
		\dot{\bg}(t) &= \left[\lambda - \sum_{(i, j) \in \calI_+} \mu j \left[\bq(t, i, j) - \bq(t, i, j + 1)\right]\right] \ind{\sigma(\bq(t)) = \sigma_*} \\
		&- \left[\sum_{(i, j) \menorigual \sigma_*} \mu j \left[\bq(t, i, j) - \bq(t, i, j + 1)\right]\right]\ind{\sigma(\bq(t)) \neq \sigma_*} \\
		&= \left[\lambda - \sum_{(i, j) \in \calI_+} \mu \bq(t, i, j)\right] \ind{\sigma(\bq(t)) = \sigma_*} \\
		&- \left[\sum_{i = 1}^m \sum_{j = J_i}^\infty \mu j \left[\bq(t, i, j) - \bq(t, i, j + 1)\right]\right]\ind{\sigma(\bq(t)) \neq \sigma_*},
	\end{align*}
	where $J_i \defeq \min \set{j \geq 1}{(i, j) \menorigual \sigma_*}$. Note that $\bq(t, i, j) = \alpha(i)$ if $(i, j) \mayor \sigma\left(\bq(t)\right)$. Thus,
	\begin{align*}
		\dot{\bg}(t) &= \left[\lambda - \mu \sum_{(i, j) \mayor \sigma_*} \alpha(i) - \mu\bg(t)\right]\ind{\sigma\left(\bq(t)\right) = \sigma_*} \\
		&- \left[\sum_{i = 1}^m \left(\mu(J_i - 1)\bq(t, i, J_i) + \sum_{j = J_i}^\infty \mu \bq(t, i, j)\right)\right]\ind{\sigma\left(\bq(t)\right) \neq \sigma_*} \\
		&= \left[\lambda - \mu \sum_{(i, j) \mayor \sigma_*} \alpha(i) - \mu\bg(t)\right]\ind{\sigma\left(\bq(t)\right) = \sigma_*} \\
		&- \mu\bg(t) \ind{\sigma\left(\bq(t)\right) \neq \sigma_*} - \mu\sum_{i = 1}^m (J_i - 1) \bq(t, i, J_i)\ind{\sigma\left(\bq(t)\right) \neq \sigma_*} \\
		&\leq \left[\lambda - \mu \sum_{(i, j) \mayor \sigma_*} \alpha(i)\right]^+ - \mu \bg(t).
	\end{align*}
	By definition of $\sigma_*$, we have $\theta \defeq \sum_{(i, j) \mayor \sigma_*} \alpha(i) \leq \rho$, and from the above bound for $\dot{\bg}$ we get
	\begin{equation*}
		\bg(t) \leq \rho - \theta + \left[\bg(t_0) - (\rho - \theta)\right]\e^{-\mu(t - t_0)} \quad \text{for all} \quad t \geq t_0.
	\end{equation*}
	We conclude that
	\begin{equation*}
		\liminf_{t \to \infty} \sum_{(i, j) \mayor \sigma_*} \bq(t, i, j) = \liminf_{t \to \infty} \left[\boldf(t) - \bg(t)\right] \geq \rho - (\rho - \theta) = \sum_{(i, j) \mayor \sigma_*} \alpha(i).
	\end{equation*}
	This proves that $\bq(t, i, j) \to \alpha(i) = q_*(i, j)$ over time for all $(i, j) \mayor \sigma_*$.
	
	Consider now the function
	\begin{equation*}
		\bh \defeq \sum_{(i, j) \menor \sigma_*} \bq(i, j).
	\end{equation*}
	Recall that $\sigma\left(\bq(t)\right) \mayorigual \sigma_*$ for all $t > t_0$. Therefore,
	\begin{equation*}
		\dot{\bh}(t) = - \sum_{(i, j) \menor \sigma_*} \mu j \left[\bq(t, i, j) - \bq(t, i , j + 1)\right] \leq - \mu \bh(t) \quad \text{for all} \quad t > t_0.
	\end{equation*}
	It follows that $\bh(t) \leq \bh(t_0) \e^{-\mu(t - t_0)}$ for all $t \geq t_0$ and thus $\bq(t, i, j) \to 0 = q_*(i, j)$ over time for all $(i, j) \menor \sigma_*$. Furthermore, we have
	\begin{equation*}
		\lim_{t \to \infty} \bq(t, \sigma_*) = \lim_{t \to \infty} \left[\boldf(t) - \bh(t) - \sum_{(i, j) \mayor \sigma_*} \bq(t, i, j)\right] = \rho - \sum_{(i, j) \mayor \sigma_*} \alpha(i) = q_*(\sigma_*),
	\end{equation*}
	and thus $\bq(t, i, j) \to q_*(i, j)$ for all $(i, j) \in \calI$.
	
	Finally, observe that
	\begin{equation*}
		\lim_{t \to \infty} \sum_{(i, j) \menor \sigma_*} \left|\bq(t, i, j) - q_*(i, j)\right| = \lim_{t \to \infty} \bh(t) = 0.
	\end{equation*}
	Consequently, $\bq(t) \to q_*$ over time not only coordinatewise but also in $\ell_1$.
\end{proof}

\subsection{Proof of the fluid limit}
\label{sub: proof of the fluid limit}

In order to prove Theorem \ref{the: fluid limit of jlmu}, it suffices to demonstrate, for each $\omega \in \Gamma$, that every subsequence of $\set{\bq_n(\omega)}{n \geq 1}$ has a further subsequence with a limit in $D_{\ell_1}[0, \infty)$, and that this limit is the unique fluid trajectory starting at $q_0(\omega)$. The first part is covered by Proposition \ref{prop: relative compactness - l1}, every subsequence of $\set{\bq_n(\omega)}{n \geq 1}$ has a further subsequence with a limit in $D_{\ell_1}[0, \infty)$. Next we characterize the limits of the convergent subsequences.

Let us fix an arbitrary $\omega \in \Gamma$, which we omit from the notation for brevity, and an increasing sequence $\calK \subset \N$ such that $\set{\calA_k}{k \in \calK}$, $\set{\calD_k}{k \in \calK}$ and $\set{q_k}{k \in \calK}$ converge in $D_{\ell_1}[0, \infty)$ to certain functions $\ba$, $\bd$ and $\bq$, respectively, which have locally Lipschitz coordinates by Proposition \ref{prop: relative compactness - product topology}. In order to characterize these three limits, it suffices to just characterize $\ba$ and $\bd$ because \eqref{eq: definition of q for jlmu} and \eqref{seq: convergence of initial conditions} imply that
\begin{equation}
\label{eq: fluid limit as arrivals minus departures}
\bq = q_0 + \ba - \bd.
\end{equation}
Since $\ba$ and $\bd$ have locally Lipschitz coordinates, there exists $\calR \subset (0, \infty)$ such that $\calR^c$ has zero Lebesgue measure and the derivatives of $\ba(i, j)$ and $\bd(i, j)$ exist for all $(i, j) \in \calI$ at all points in $\calR$. These derivatives are zero if $j = 0$ by the definitions of $\calA_k$ and $\calD_k$. The following lemma computes the derivatives for $(i, j) \in \calI_+$.

\begin{lemma}
	\label{lem: fluid derivatives}
	Fix an arbitrary $t_0 \in \calR$, we have
	\begin{equation}
	\label{eq: fluid departure rates}
	\dot{\bd}(t_0, i, j) = \mu j \left[\bq(t_0, i, j) - \bq(t_0, i, j + 1)\right] \quad \text{for all} \quad (i, j) \in \calI_+.
	\end{equation}
	Furthermore, $\bq(t_0)$ and the derivatives $\dot{\ba}(t_0, i, j)$ satisfy
	\begin{equation}
	\label{eq: fluid arrival rates}
	\Lambda(\bq(t_0), i, j) = \dot{\ba}(t_0, i, j) \geq 0 \quad \text{for all} \quad (i, j) \in \calI_+.
	\end{equation}
\end{lemma}

\begin{proof}
	The sequences $\set{\calD_k(i, j)}{k \in \calK}$ and $\set{\bq_k(i, j)}{k \in \calK}$ converge uniformly over compact sets to $\bd(i, j)$ and $\bq(i, j)$, respectively, for all $(i, j) \in \calI$. This remark, the definition of $\calD_k$ and \eqref{seq: fslln for departures} imply that
	\begin{equation*}
	\bd(t, i, j) = \int_0^t \mu j \left[\bq(s, i, j) - \bq(s, i, j + 1)\right]ds \quad \text{for all} \quad t \geq 0.
	\end{equation*}
	It is clear that this identity establishes \eqref{eq: fluid departure rates}.
	
	We now prove that
	\begin{equation}
	\label{eq: properties of fluid arrival rates}
	\sum_{(i, j) \in \calI_+} \dot{\ba}(t_0, i, j) = \lambda \quad \text{and} \quad \dot{\ba}(t_0, i, j) \geq 0 \quad \text{for all} \quad (i, j) \in \calI_+.
	\end{equation}
	The derivatives $\dot{\ba}(t_0, i, j)$ are non-negative since the processes $\calA_k(i, j)$ are non-decreasing, so we only need to show that the derivatives add up to $\lambda$. For this purpose, note that
	\begin{equation*}
	\sum_{(i, j) \in \calI_+} \calA_k(t, i, j) = \calN_k^\lambda(t) \quad \text{for all} \quad t \geq 0 \quad \text{and} \quad k \in \calK.
	\end{equation*}
	Fix $T > t_0$, and let $j_T$ and $n_T$ be as in Lemma \ref{lem: set of nice sample paths}. The left-hand side has at most $m j_T$ non-zero terms for all $k \geq n_T$ and $t \in [0, T]$. It follows from \eqref{seq: fslln for arrivals} that
	\begin{equation*}
	\sum_{(i, j) \in \calI_+} \ba(t, i, j) = \lambda t \quad \text{for all} \quad t \in [0, T].
	\end{equation*}
	This yields \eqref{eq: properties of fluid arrival rates} since the left-hand side has at most $mj_T$ non-zero terms.
	
	It follows from \eqref{eq: fluid limit as arrivals minus departures} that the derivative of $\bq(i, j)$ exists at $t_0$ and
	\begin{equation*}
	\dot{\bq}(t_0, i, j) = \dot{\ba}(t_0, i, j) - \dot{\bd}(t_0, i, j) \quad \text{for all} \quad (i, j) \in \calI.
	\end{equation*}
	Note that $\bq(i, j)$ is upper bounded by $\alpha(i)$, so $\bq(t_0, i, j) = \alpha(i)$ implies $\dot{\bq}(t_0, i, j) = 0$. Thus,
	\begin{equation}
	\label{eq: fluid and departure arrival rate are equal}
	\dot{\ba}(t_0, i, j) = \dot{\bd}(t_0, i, j) = \mu j [\alpha(i) - \bq(t_0, i, j + 1)] \quad \text{if} \quad \bq(t_0, i, j) = \alpha(i).
	\end{equation}
	In order to prove the equality in \eqref{eq: fluid arrival rates}, define $\sigma_0 = (i_0, j_0) \defeq \sigma\left(\bq(t_0)\right)$. If $(i, j) \mayor \sigma_0$, then $\bq(t_0, i, j) = \bq(t_0, i, j - 1)$ by \eqref{eq: definition of sigma}. Moreover, for a fixed $i$, the marginal utility $\Delta(i, j)$ does not increase with $j$ since $u_i$ is a concave function. Therefore, $(i, j) \mayor \sigma_0$ and $j > 1$ imply that $(i, j - 1) \mayor \sigma_0$. We conclude that
	\begin{equation*}
	\bq(t_0, i, j) = \bq(t_0, i, 0) = \alpha(i) \quad \text{for all} \quad (i, j) \mayor \sigma_0.
	\end{equation*}
	The last property and \eqref{eq: fluid and departure arrival rate are equal} imply that \eqref{eq: fluid arrival rates} holds if $(i, j) \mayor \sigma_0$.
	
	Note that $\bq(t_0, \sigma_0) < \alpha(i_0)$ by \eqref{eq: definition of sigma}. Since $\bq(\sigma_0)$ is continuous and $\bq_k(\sigma_0)$ converges uniformly over compact sets to $\bq(\sigma_0)$, there exist $\varepsilon > 0$ and $k_\varepsilon \in \calK$ such that
	\begin{equation*}
	\bq_k(t, \sigma_0) < \alpha_k(i_0) \quad \text{for all} \quad t \in (t_0 - \varepsilon, t_0 + \varepsilon) \quad \text{and} \quad k \geq k_\varepsilon.
	\end{equation*}
	It follows from \eqref{eq: definition of sigma} and the last statement that
	\begin{equation*}
	\sigma_0 \menorigual \sigma\left(\bq_k(t)\right) \quad \text{for all} \quad t \in (t_0 - \varepsilon, t_0 + \varepsilon) \quad \text{and} \quad k \geq k_\varepsilon.
	\end{equation*}
	Thus, $\calA_k(i, j)$ is constant over $(t_0 - \varepsilon, t_0 + \varepsilon)$ if $k \geq k_\varepsilon$ and $(i, j) \menor \sigma_0$. Indeed, server pools of class $i$ with exactly $j - 1$ tasks are not assigned incoming tasks in the system with $k$ server pools if $(i, j) \menor \sigma\left(\bq_k\right)$. This proves \eqref{eq: fluid arrival rates} for $(i, j) \menor \sigma_0$, and we conclude from \eqref{eq: properties of fluid arrival rates} that \eqref{eq: fluid arrival rates} must also hold in the case $(i, j) = \sigma_0$.
\end{proof}

Below we complete the proof of Theorem \ref{the: fluid limit of jlmu}.

\begin{proof}[Proof of Theorem \ref{the: fluid limit of jlmu}.]
	As above, we fix some $\omega \in \Gamma$ which we omit from the notation. Every subsequence of $\set{\bq_n}{n \geq 1}$ has a further subsequence which converges in $D_{\ell_1}[0, \infty)$ by Proposition \ref{prop: relative compactness - l1}. It follows from \eqref{seq: convergence of initial conditions} and Lemma \ref{lem: fluid derivatives} that the limit $\bq$ of this convergent subsequence is a fluid trajectory with $\bq(0) = q_0$, which determines $\bq$ by Theorem \ref{the: uniqueness of fluid trajectories}. 
\end{proof}

\section{Limiting behavior of SLTA}
\label{sec: limiting behavior of threshold}

In this section we assume that the load balancing policy is SLTA and we leverage a methodology developed in \cite{goldsztajn2021learning} to prove Theorem \ref{the: fluid limit of threshold}. The first steps of the proof are carried out in Section \ref{sub: asymptotic dynamical properties}, where we establish that certain dynamical properties of the system hold asymptotically with probability one. The proof is completed in Section \ref{sub: evolution of the learning scheme}, where we analyze the evolution of the learning scheme over time. While the arguments used here are more involved due to the heterogeneity of the system, most of the proofs are conceptually similar to those in \cite{goldsztajn2021learning}, and hence are deferred to Appendix \ref{app: limiting behavior of threshold}.

\subsection{Asymptotic dynamical properties}
\label{sub: asymptotic dynamical properties}

In this section we establish asymptotic dynamical properties pertaining to the total and tail mass processes, which are defined as
\begin{equation*}
\bs_n \defeq \sum_{(i, j) \in \calI_+} \bq_n(i, j) \quad \text{and} \quad \bv_n(r) \defeq \sum_{(i, j) \menorigual (i_r, j_r)} \bq_n(i, j) \quad \text{for all} \quad r \geq 1,
\end{equation*}
respectively. Recall that the total mass process was introduced in \eqref{eq: definition of s_n} and represents the total number of tasks in the system, normalized by the number of server pools. The following proposition is proved in Appendix \ref{app: limiting behavior of threshold}.

\begin{proposition}
	\label{prop: fluid limit of total number of tasks}
	For each $\omega \in \Gamma$, the sequence $\set{\bs_n(\omega)}{n \geq 1}$ converges uniformly over compact sets to the unique function $\bs(\omega)$ such that
	\begin{equation}
	\label{eq: fluid limit of total number of tasks}
	\dot{\bs}(\omega) = \lambda - \mu \bs(\omega) \quad \text{and} \quad \bs(\omega, 0) = s_0(\omega),
	\end{equation}
	where $s_0$ is as defined in \eqref{eq: initial normalized number of tasks}. Explicitly, $\bs(\omega, t) = \rho + \left[s_0(\omega) - \rho\right]\e^{-\mu t}$.  
\end{proposition}

While the above law of large numbers is known to hold weakly, it is not straightforward that it holds with probability one under the coupled construction of sample paths adopted in Section \ref{sub: coupled construction of sample paths}; this fact is established in Proposition \ref{prop: fluid limit of total number of tasks}.

The next result is also proved in Appendix \ref{app: limiting behavior of threshold}, it provides an asymptotic upper bound for certain tail mass processes, under specific conditions concerning $\bq_n$ and $\br_n$. The upper bound implies at least an exponentially fast decay over time.

\begin{proposition}
	\label{prop: asymptotic dynamical property of tail mass processes}
	Suppose that the next conditions hold for a given $\omega \in \Gamma$ and a given increasing sequence $\calK$ of natural numbers.
	\begin{enumerate}
		\item[(a)] The sequence $\set{\bq_k(\omega)}{k \in \calK}$ converges in $D_{\ell_1}[0, \infty)$ to some function $\bq$.
		
		\item[(b)] There exist $r > 1$ and $0 \leq t_0 < t_1$ such that
		\begin{equation*}
		\br_k(\omega, t) \leq r \quad \text{and} \quad \sum_{(i, j) \mayor (i_r, j_r)} \bq_k(\omega, t, i, j) < \sum_{(i, j) \mayor (i_r, j_r)} \alpha_k(i)
		\end{equation*}
		for all $t \in [t_0, t_1]$ and $k \in \calK$.
	\end{enumerate}
	Then $\bq(i, j)$ is differentiable on $(t_0, t_1)$ for all $(i, j) \menorigual (i_r, j_r)$ and satisfies
	\begin{equation*}
	\dot{\bq}(t, i, j) = -\mu j \left[\bq(t, i, j) - \bq(t, i, j + 1)\right] \quad \text{for all} \quad t \in (t_0, t_1).
	\end{equation*}
	Furthermore, the sequence of tail mass processes $\set{\bv_k(\omega, r)}{k \in \calK}$ converges uniformly over compact sets to a function $\bv(\omega, r)$ that satisfies
	\begin{equation*}
	\bv(\omega, t, r) < \bs(\omega, t_0)\e^{-\mu\left(t - t_0\right)} \quad \text{for all} \quad t \in [t_0, t_1].
	\end{equation*}
\end{proposition}

\subsection{Evolution of the learning scheme}
\label{sub: evolution of the learning scheme}

In this section we complete the proof of Theorem \ref{the: fluid limit of threshold}. In Section \ref{subsub: preliminary results} we establish that there exists a neighborhood of zero outside of which $\br_n$ is asymptotically upper bounded by $r_*$ with probability one. This property partially proves \eqref{seq: limit of threshold 1} and is used to obtain \eqref{seq: limit of threshold 3}. The proof of \eqref{seq: limit of threshold 1} is finished in Section \ref{subsub: proof of fluid limit of threshold}, where we also establish \eqref{seq: limit of threshold 2}.

\subsubsection{Preliminary results}
\label{subsub: preliminary results}

The following proposition states that $\br_n$ is asymptotically upper bounded by $r_*$ outside of a neighborhood of zero with probability one; the proof is deferred to Appendix \ref{app: limiting behavior of threshold}.
\begin{proposition}
	\label{prop: upper bound for br}
	There exists a function $\map{\tau_\bound}{[0, \infty)}{\R}$ with the following property. If $\omega \in \Gamma$ and  $T \geq \tau > \tau_\bound(s_0(\omega))$, then there exists $n_\bound^{\tau, T}(\omega)$ such that
	\begin{equation*}
	\br_n(\omega, t) \leq r_* \quad \text{for all} \quad t \in [\tau, T] \quad \text{and} \quad n \geq n_\bound^{\tau, T}(\omega).
	\end{equation*} 
\end{proposition}

The following corollary establishes \eqref{seq: limit of threshold 3}.

\begin{corollary}
	\label{cor: decay of the tail}
	For each $\omega \in \Gamma$ and $T \geq \tau > \tau_\bound(s_0(\omega))$, we have
	\begin{equation*}
	\limsup_{n \to \infty} \sup_{t \in [\tau, T]} \bv_n(\omega, t, r_* + 1)\e^{\mu (t - \tau)} \leq \bs(\omega, \tau).
	\end{equation*}
	In particular, \eqref{seq: limit of threshold 3} holds.
\end{corollary}

\begin{proof}
	We fix $\omega \in \Gamma$ and $T \geq \tau > \tau_\bound(s_0(\omega))$, and we omit $\omega$ from the notation for brevity. Suppose that the statement of the corollary does not hold, then there exist $\varepsilon > 0$ and an increasing sequence $\calK$ of natural numbers such that
	\begin{align*}
	\sup_{t \in [\tau, T]} \bv_k(t, r_* + 1)\e^{\mu (t - \tau)} > \bs(\tau) + \varepsilon \quad \text{for all} \quad k \in \calK.
	\end{align*}
	
	By Propositions \ref{prop: relative compactness - l1} and \ref{prop: upper bound for br}, we may assume that $\set{\bq_k}{k \in \calK}$ has a limit in $D_{\ell_1}[0, \infty)$ and that $\br_k(t) \leq r_*$ for all $t \in [\tau, T]$ and all $k \in \calK$. The latter property implies that
	\begin{equation*}
	\sum_{(i, j) \mayorigual (i_{r_*}, j_{r_*})} \bq_k(t, i, j) < \sum_{(i, j) \mayorigual (i_{r_*}, j_{r_*})} \alpha_k(i) \quad \text{for all} \quad t \in [\tau, T],
	\end{equation*}
	since otherwise the number of tokens would be zero, which cannot occur by Remark \ref{rem: dynamics of threshold}. Therefore, Proposition \ref{prop: asymptotic dynamical property of tail mass processes} holds with $r = r_* + 1$ along the interval $[\tau, T]$, and in particular, there exists a function $\bv(r_* + 1)$ such that
	\begin{align*}
	&\bv(t, r_* + 1) < \bs(\tau)\e^{-\mu (t - \tau)} \quad \text{for all} \quad t \in [\tau , T], \\
	&\lim_{k \to \infty} \sup_{t \in [\tau, T]} \left|\bv_k(t, r_* + 1) - \bv(t, r_* + 1)\right| = 0.
	\end{align*}
	This leads to a contradiction, so the statement of the corollary must hold.
\end{proof}

\subsubsection{Proof of Theorem \ref{the: fluid limit of threshold}}
\label{subsub: proof of fluid limit of threshold}

Below we complete the proof of Theorem \ref{the: fluid limit of threshold}. For this purpose, let
\begin{equation}
\label{eq: error process}
\begin{split}
\delta_n(t, r) &\defeq \frac{1}{n}\calN_n^\lambda(t) - \lambda t \\
&- \sum_{(i, j) \mayor (i_r, j_r)} \left[\calD_n(t, i, j) - \int_0^t \mu j \left[\bq_n(s, i, j) - \bq_n(s, i, j + 1)\right]ds\right] 
\end{split} 
\end{equation}
for all $t \geq 0$ and $r \geq 1$. It follows from \eqref{seq: fslln for arrivals} and \eqref{seq: fslln for departures} that
\begin{equation}
\label{eq: error process goes to zero}
\lim_{n \to \infty} \sup_{t \in [0, T]} n^\gamma |\delta_n(\omega, t, r)| = 0 \quad \text{for all} \quad \gamma \in [0, 1 / 2), \quad T \geq 0 \quad \text{and} \quad \omega \in \Gamma.
\end{equation}

The following two technical lemmas are proved in Appendix \ref{app: limiting behavior of threshold}.

\begin{lemma}
	\label{lem: lemma 1 for limit of threshold}
	Fix $\omega \in \Gamma$, $T \geq 0$ and $r > 1$. Suppose that there exist an increasing sequence $\calK$ of natural numbers and random times $0 \leq \tau_{k, 1} \leq \tau_{k, 2} \leq T$ such that
	\begin{equation*}
	\br_k(\omega, t) \leq r \quad \text{and} \quad \sum_{(i, j) \mayor (i_r, j_r)} \bq_k(\omega, t, i, j) < \sum_{(i, j) \mayor (i_r, j_r)} \alpha_k(i)
	\end{equation*}
	for all $t \in \left[\tau_{k, 1}(\omega), \tau_{k, 2}(\omega)\right)$ and $k \in \calK$. Then
	\begin{align*}
	&\sum_{(i, j) \mayor (i_r, j_r)} \bq_k(\omega, t, i, j) - \sum_{(i, j) \mayor (i_r, j_r)} \bq_k\left(\omega, \tau_{k, 1}(\omega), i, j\right) \geq \\
	&\left[t - \tau_{k, 1}(\omega)\right] \left[\lambda - \mu \sum_{(i, j) \mayor (i_r, j_r)} \alpha_k(i)\right] - 2\sup_{s \in [0, T]} \left|\delta_k(\omega, s, r)\right|
	\end{align*}
	for all $t \in \left[\tau_{k, 1}(\omega), \tau_{k, 2}(\omega)\right]$ and $k \in \calK$.
\end{lemma}

\begin{lemma}
	\label{lem: lemma 2 for the limit of threhsold}
	Fix $\omega \in \Gamma$, $T \geq 0$ and $1 \leq r \leq r_*$. Assume that there exist an increasing sequence $\calK$ of natural numbers and random times $0 \leq \zeta_{k, 1} \leq \zeta_{k, 2} \leq T$ such that
	\begin{equation*}
	\sum_{(i, j) \mayor (i_r, j_r)} \bq_k\left(\omega, \zeta_{k, 1}(\omega), i, j\right) = \sum_{(i, j) \mayor (i_r, j_r)} \alpha_k(i) \quad \text{and} \quad \br_k(\omega, t) \leq r
	\end{equation*}
	for all $t \in \left[\zeta_{k, 1}(\omega), \zeta_{k, 2}(\omega)\right]$ and $k \in \calK$. For each $\gamma \in [0, 1 / 2)$, we have
	\begin{equation*}
	\br_k(\omega, t) = r \quad \text{and} \quad \bq_k(\omega, t, i, j) \geq \alpha_k(i) - k^{-\gamma}
	\end{equation*}
	for all $(i, j) \mayor (i_r, j_r)$, $t \in \left[\zeta_{k, 1}(\omega), \zeta_{k, 2}(\omega)\right]$ and all large enough $k \in \calK$.
\end{lemma}

The above lemmas are used to complete the proof of Theorem \ref{the: fluid limit of threshold}.

\begin{proof}[Proof of Theorem \ref{the: fluid limit of threshold}.]
	We define $\tau_\eq$ as follows:
	\begin{equation*}
	\tau_\eq(s) \defeq \tau_\bound(s) + \frac{1}{\mu}\log \left(\frac{\rho}{\rho - \sum_{(i, j) \mayor \sigma_*} \alpha(i)}\right) \quad \text{for all} \quad s \geq 0.
	\end{equation*}
	Fix $\omega \in \Gamma$ and $T \geq \tau > \tau_\eq(s_0(\omega))$ as in the statement of the theorem; we omit $\omega$ from the notation for brevity. Given $0 < \varepsilon < \rho - \sum_{(i, j) \mayor \sigma_*} \alpha(i)$, we define
	\begin{equation*}
	\tau(\varepsilon) \defeq \min \set{t \geq 0}{\rho\left(1 - \e^{-\mu t}\right) - \varepsilon \geq \sum_{(i, j) \mayor \sigma_*} \alpha(i)} = \frac{1}{\mu}\log \left(\frac{\rho}{\rho - \sum_{(i, j) \mayor \sigma_*} \alpha(i) - \varepsilon}\right).
	\end{equation*}
	Fix $\tau_0 > \tau_\bound(s_0)$ and $\varepsilon$ such that $\tau = \tau_0 + \tau(\varepsilon) + \varepsilon$. This is possible since $\tau_\bound(s_0) + \tau(\varepsilon) + \varepsilon$ decreases to $\tau_\eq(s_0)$ as $\varepsilon \to 0$. In addition, consider the random times
	\begin{equation*}
	\xi_n \defeq \inf \set{t \geq \tau_0}{\sum_{(i, j) \mayor \sigma_*} \bq_n(t, i, j) = \sum_{(i, j) \mayor \sigma_*} \alpha_n(i)}.
	\end{equation*}
	The proofs of \eqref{seq: limit of threshold 1} and \eqref{seq: limit of threshold 2} will be completed if we demonstrate that $\xi_n \leq \tau$ for all large enough $n$. Indeed, if this is established, then \eqref{seq: limit of threshold 1} and \eqref{seq: limit of threshold 2} follow from Lemma \ref{lem: lemma 2 for the limit of threhsold} with $r \defeq r_*$, $\zeta_{n, 1} \defeq \xi_n$ and $\zeta_{n, 2} \defeq T$. The hypotheses of the lemma hold since $\br_n(t) \leq r_*$ for all $t \in [\tau_0, T]$ and all large enough $n$ by the choice of $\tau_0$ and Proposition \ref{prop: upper bound for br}.
	
	In order to prove that $\xi_n \leq \tau$ for all large enough $n$, we show that
	\begin{equation}
	\label{aux: limsup of xi}
	\limsup_{n \to \infty} \xi_n \leq \tau_0 + \tau(\varepsilon) < \tau_0 + \tau(\varepsilon) + \varepsilon = \tau.
	\end{equation}
	If $r_* = 1$, then $\xi_n = \tau_0$ for all $n$ and the above inequality holds, so suppose that $r_* > 1$.
	
	Assume that \eqref{aux: limsup of xi} does not hold. Then there exists an increasing sequence $\calK$ of natural numbers such that $\xi_k > \tau_0 + \tau(\varepsilon)$ for all $k \in \calK$. Moreover, by Propositions \ref{prop: relative compactness - l1} and \ref{prop: upper bound for br}, this sequence may be chosen so that the next two properties hold.
	\begin{enumerate}
		\item[(i)] $\set{\bq_k}{k \in \calK}$ converges in $D_{\ell_1}[0, \infty)$.
		
		\item[(ii)] $\br_k(t) \leq r_*$ for all $t \in [\tau_0, T]$ and $k \in \calK$.
	\end{enumerate}

	The definition of $\xi_k$ implies that
	\begin{equation*}
	\sum_{(i, j) \mayor \sigma_*} \bq_k(t, i, j) < \sum_{(i, j) \mayor \sigma_*} \alpha_k(i) \quad \text{for all} \quad t \in \left[\tau_0, \tau_0 + \tau(\varepsilon)\right] \subset \left[\tau_0, \xi_k\right) \quad \text{and} \quad k \in \calK.
	\end{equation*}
	The hypotheses of Proposition \ref{prop: asymptotic dynamical property of tail mass processes} hold with $r \defeq r_*$, $t_0 \defeq \tau_0$ and $t_1 \defeq \tau_0 + \tau(\varepsilon)$, by the above remark and properties (i) and (ii). Let $\bv(r_*)$ be the function defined in this proposition, as the uniform limit of the tail processes $\bv_k(r_*)$ over $[0, T]$. It follows from Propositions \ref{prop: fluid limit of total number of tasks} and \ref{prop: asymptotic dynamical property of tail mass processes} that
	\begin{equation*}
	\sup_{t \in [0, T]} \left|\bs_k(t) - \bv_k(t, r_*) - \left[\bs(t) - \bv(t, r_*)\right]\right| \leq \frac{\varepsilon}{2}
	\end{equation*}
	for all sufficiently large $k \in \calK$. For each of these $k$, we have
	\begin{align*}
	\sum_{(i, j) \mayor \sigma_*} \bq_k\left(\tau_0 + \tau(\varepsilon), i, j\right) &= \bs_k\left(\tau_0 + \tau(\varepsilon)\right) - \bv_k(\tau_0 + \tau(\varepsilon), r_*) \\
	&\geq \bs\left(\tau_0 + \tau(\varepsilon)\right) - \bv(\tau_0 + \tau(\varepsilon), r_*) - \frac{\varepsilon}{2} \\
	&> \rho + \left[\bs(\tau_0) - \rho\right]\e^{-\mu\tau(\varepsilon)} - \bs(\tau_0)\e^{-\mu\tau(\varepsilon)} - \frac{\varepsilon}{2} \\
	&= \rho \left(1 - \e^{-\mu\tau(\varepsilon)}\right) - \frac{\varepsilon}{2} = \sum_{(i, j) \mayor (i_r, j_r)} \alpha(i) + \frac{\varepsilon}{2}.
	\end{align*}
	The third inequality follows from Proposition \ref{prop: asymptotic dynamical property of tail mass processes}, and the last equality from the definition of $\tau(\varepsilon)$. It follows from \eqref{ass: size of classes and initial condition} that the right-hand side is strictly larger than $\sum_{(i, j) \mayor \sigma_*} \alpha_k(i)$ for all large enough $k \in \calK$, which is a contradiction.
	
	We conclude that \eqref{aux: limsup of xi} holds, which proves \eqref{seq: limit of threshold 1} and \eqref{seq: limit of threshold 2}. We had already proved \eqref{seq: limit of threshold 3} in Corollary \ref{cor: decay of the tail}, thus the proof of the theorem is complete.
\end{proof}

\section{Asymptotic optimality}
\label{sec: asymptotic optimality}

In this section we prove Theorem \ref{the: asymptotic optimality}. Specifically, in Section \ref{sub: drift analysis} we use drift analysis to demonstrate that the continuous-time Markov chains introduced in Section \ref{sub: stochastic models} are irreducible and positive-recurrent, and to derive upper bounds for certain expectations and tail probabilities. In Section \ref{sub: proof of the asymptotic optimality} we use these upper bounds to establish that the stationary distributions of the latter Markov chains are tight, and then we complete the proof of Theorem \ref{the: asymptotic optimality} using the results of Sections \ref{sec: limiting behavior of jlmu} and \ref{sec: limiting behavior of threshold}. Finally, in Section \ref{sub: suboptimality result} we demonstrate that JLMU is not optimal in general, although it is asymptotically optimal.

\subsection{Drift analysis}
\label{sub: drift analysis}

Denote the state space and the generator matrix of the continuous-time Markov chains defined in Section \ref{sub: stochastic models} by $S_n$ and $A_n$, respectively. We use exactly the same notation for JLMU and SLTA, but we always indicate which policy is being considered.  The drift of a function $\map{f}{S_n}{[0, \infty)}$ is the function $A_nf$ defined by
\begin{equation*}
A_nf(x) \defeq \sum_{y \in S_n} A_n(x, y) f(y) = \sum_{y \neq x} A_n(x, y) \left[f(y) - f(x)\right] > -\infty \quad \text{for all} \quad x \in S_n.
\end{equation*}
The proof of the following proposition uses a Foster-Lyapunov argument, which is based on the drift of certain suitably chosen functions.

\begin{proposition}
	\label{prop: irreducibility and positive recurrence}
	For each given $n$, the two continuous-time Markov chains introduced in Section \ref{sub: stochastic models} are irreducible and positive-recurrent. In particular, each of these Markov chains has a unique stationary distribution $\pi_n$.
\end{proposition}

\begin{proof}
	Suppose first that the load balancing policy is JLMU. Any occupancy state can reach the empty occupancy state after a finite number of consecutive departures. By the definition of $S_n$ provided in Section \ref{sub: stochastic models}, the latter remark implies that $\bq_n$ is irreducible. Moreover, $\bq_n$ is the empty occupancy state if and only if $\bs_n = 0$, which implies that the empty occupancy state is positive-recurrent, because the $M/M/\infty$ queue $\bs_n$ is irreducible and positive-recurrent. Thus, $\bq_n$ is positive-recurrent.
	
	Suppose now that the load balancing policy is SLTA. Any state $(q, r) \in S_n$ can reach the empty occupancy sate with $\br_n = r$ after a finite number of consecutive departures. Moreover, the latter state can reach the empty occupancy state with $\br_n = 1$ after a finite number of alternate arrivals and departures. We conclude from the definition of $S_n$ provided in Section \ref{sub: stochastic models} that $(\bq_n, \br_n)$ is irreducible.
	
	Next we use a Foster-Lyapunov argument to prove the positive recurrence. Consider the functions $\map{f, g}{S_n}{[0, \infty)}$ defined by
	\begin{equation}
	\label{eq: definition of f and g}
	f(q, r) \defeq \sum_{(i, j) \in \calI_+} q(i, j) \quad \text{and} \quad g(q, r) \defeq \frac{r}{n} \quad \text{for all} \quad (q, r) \in S_n.
	\end{equation}
	All server pools together form an infinite-server system, thus $A_nf(q, r) = \lambda - \mu f(q, r)$ for all $(q, r) \in S_n$. In addition, we have
	\begin{equation*}
	A_ng(q, r) = \lambda \left[\indc_I(q, r) - \indc_D(q, r)\right] \quad \text{for all} \quad (q, r) \in S_n.
	\end{equation*}
	Here $I$ corresponds to those states $(q, r)$ such that $\br_n$ increases if $(\bq_n, \br_n) = (q, r)$ and the next event is an arrival. Specifically,
	\begin{equation*}
	I \defeq \set{(q, r) \in S_n}{q(i, j) = \alpha_n(i)\ \text{for all}\ (i, j) \mayor (i_r, j_r),\ nq(i_r, j_r) = n\alpha_n(i_r) - 1}.
	\end{equation*}
	Also, $D$ corresponds to those states $(q, r)$ such that $\br_n$ decreases if $(\bq_n, \br_n) = (q, r)$ and the next event is an arrival. Specifically,
	\begin{equation*}
	D \defeq \set{(q, r) \in S_n}{r > 1,\ n - \sum_{i = 1}^m nq\left(i, \ell_i(r)\right) \geq n\beta_n,\ q(i_{r - 1}, j_{r - 1}) < \alpha_n(i_{r - 1})}.
	\end{equation*}
	
	Consider the function $h \defeq f + 2 g$ and let $F$ be the set of those $(q, r) \in S_n$ that satisfy the following two conditions.
	\begin{enumerate}
		\item[(i)] $f(q, r) \leq 4\rho$.
		
		\item[(ii)] $r = 1$ or $r > 1$ and $\left[\alpha_n(i_{r - 1}) - \beta_n\right] j_{r - 1} \leq 4\rho$.
	\end{enumerate}
	The first condition holds for finitely many $q \in Q_n$ and the second condition holds for finitely many $r \geq 1$, thus $F$ is finite. Next we establish that $A_nh \leq -\lambda + 4\lambda \indc_F$. Note that $(\bq_n, \br_n)$ is non-explosive since the infinite-server queue $\bs_n$ has this property. Therefore, it follows from \cite[Proposition 2.2.1]{hajek2006notes} that $(\bq_n, \br_n)$ is positive-recurrent.
	
	The latter inequality holds for all $(q, r) \in F$ since $A_n h \leq 3\lambda$. Hence, let us assume that $(q, r) \notin F$. Suppose that $(q, r) \notin F$ violates (i). Then
	\begin{equation*}
	A_nh(q, r) \leq 3\lambda - \mu f(q, r) < 3\lambda - 4\lambda = -\lambda.
	\end{equation*}
	Assume now that $(q, r) \notin F$ satisfies (i). Then $(q, r)$ satisfies (i) and violates~(ii), which implies that $r > 1$ and $\left[\alpha_n(i_{r - 1}) - \beta_n\right]j_{r - 1} > 4\rho \geq f(q, r)$. From this we conclude that $q(i_{r - 1}, j_{r - 1}) < \alpha_n(i_{r - 1}) - \beta_n$, since otherwise $f(q, r) \geq j_{r - 1}q(i_{r - 1}, j_{r - 1}) > 4\rho$. Thus,
	\begin{equation*}
	n - \sum_{i = 1}^m nq\left(i, \ell_i(r)\right) = \sum_{i = 1}^m n\left[\alpha_n(i) - q\left(i, \ell_i(r)\right)\right] \geq n\left[\alpha_n(i_{r - 1}) - q(i_{r - 1}, j_{r - 1})\right] > n\beta_n.
	\end{equation*}
	It follows that $(q, r) \in D$, thus $A_nh(q, r) = \lambda - \mu f(q, r) - 2\lambda \leq -\lambda$.
\end{proof}

Next we provide upper bounds for certain expectations and tail probabilities, which are used in the following section to demonstrate that the sequence of stationary distributions $\set{\pi_n}{n \geq 1}$ is tight, both for JLMU and SLTA. First we state a technical lemma; the proof follows from Fubini's theorem and is provided in Appendix \ref{app: auxiliary results}.

\begin{lemma}
	\label{lem: fubini}
	Let $x_n$ have the stationary distribution $\pi_n$, where $x_n = q_n$ if JLMU is used and $x_n = (q_n, r_n)$ if SLTA is used. If $\map{f}{S_n}{[0, \infty)}$ satisfies
	\begin{equation*}
	\bE\left[\sum_{y \in S_n} \left|A_n(x_n, y)\right|f(y)\right] < \infty, \quad \text{then} \quad \bE\left[A_nf(x_n)\right] = 0.
	\end{equation*}
\end{lemma}

Consider the quantities
\begin{equation}
\label{eq: definition of theta n k}
\theta_n^k \defeq \sum_{(i, j) \mayorigual (i_k, j_k)} \alpha_n(i) \quad \text{for all} \quad k \geq 1.
\end{equation}
The above lemma is used to prove the following two propositions.

\begin{proposition}
	\label{prop: expectation bound jlmu}
	Suppose that the load balancing policy is JLMU, fix $n$ and consider the functions $\map{f_k}{S_n}{[0, \infty)}$ defined by
	\begin{equation*}
	f_k(q) \defeq \sum_{(i, j) \menor (i_k, j_k)} q(i, j) \quad \text{for all} \quad q \in S_n \quad \text{and} \quad k \geq 1.
	\end{equation*}
	If $q_n$ has the stationary distribution $\pi_n$, then
	\begin{equation*}
	\bE\left[f_k(q_n)\right] \leq \rho \e^{-n\left(\theta_n^k - 2\rho\right)} \quad \text{for all} \quad k \geq 1.
	\end{equation*}
\end{proposition}

\begin{proof}
	Fix some $k \geq 1$ and define $J_i \defeq \min \set{j \geq 1}{(i, j) \menor (i_k, j_k)}$. The drift of $f_k$ with respect to $\bq_n$ satisfies
	\begin{equation}
	\label{aux: drift bound}
	\begin{split}
	A_nf_k(q) &= \lambda \ind{\sigma(q) \menor (i_k, j_k)} - \sum_{(i, j) \menor (i_k, j_k)} \mu j \left[q(i, j) - q(i, j + 1)\right] \\
	&= \lambda \ind{\sigma(q) \menor (i_k, j_k)} - \sum_{i = 1}^m \sum_{j = J_i}^\infty \mu j \left[q(i, j) - q(i, j + 1)\right] \\
	&= \lambda \ind{\sigma(q) \menor (i_k, j_k)} - \sum_{i = 1}^m \left[\mu \left(J_i - 1\right) q\left(i, J_i\right) + \mu \sum_{j = J_i}^\infty q(i, j)\right] \\
	&= \lambda \ind{\sigma(q) \menor (i_k, j_k)} - \mu f_k(q) - \mu \sum_{i = 1}^m \left(J_i - 1\right) q\left(i, J_i\right) \\
	&\leq \lambda \ind{f(q) \geq \theta_n^k} - \mu f_k(q) \quad \text{for all} \quad q \in S_n,
	\end{split}
	\end{equation}
	where $f(q)$ is as in \eqref{eq: definition of f and g}. For the last step, note that $(i, j) \mayor \sigma(q)$ implies that $q(i, j) = \alpha_n(i)$.
	
	Observe that $\left|A_n(x, x)\right| = n\lambda + n \mu f(x)$ because $n f(x)$ is the total number of tasks at the occupancy state $x$ and $n \lambda$ is the arrival rate of tasks. Hence,
	\begin{align*}
	\sum_{y \in S_n} \left|A_n(x, y)\right|f_k(y) &= \left|A_n(x, x)\right|f_k(x) + \sum_{y \neq x} A_n(x, y)f_k(y) \\
	&= 2\left|A_n(x, x)\right|f_k(x) + \sum_{y \in S_n} A_n(x, y) f_k(y) \\
	&= 2\left[n\lambda + n \mu f(x)\right]f_k(x) + A_nf_k(x) \leq 2 n \left[\lambda + \mu f(x)\right] f(x) + \lambda.
	\end{align*}
	The right-hand side has a finite mean with respect to $\pi_n$ since $nf(q_n)$ is the total number of tasks in the system in stationarity, which is Poisson distributed with mean $n\rho$. Thus, we conclude that $\bE\left[A_nf_k(q_n)\right] = 0$ by Lemma \ref{lem: fubini}.
	
	Taking expectations with respect to $\pi_n$ on both sides of \eqref{aux: drift bound}, and recalling that $n f(q_n)$ is Poisson distributed with mean $n\rho$, we obtain
	\begin{equation*}
	\bE\left[f_k\left(q_n\right)\right] \leq \rho \bE\left[\ind{f(q_n) \geq \theta_n^k}\right] = \rho \bP\left(nf(q_n) \geq n\theta_n^k\right) \leq \rho \e^{-n\left(\theta_n^k - 2\rho\right)},
	\end{equation*}
	where the last inequality follows from a Chernoff bound.
\end{proof}

\begin{proposition}
	\label{prop: tail probabilities threshold}
	Suppose that the load balancing policy is SLTA, fix $n$ and consider the functions $\map{f_k}{S_n}{[0, \infty)}$ defined by
	\begin{equation*}
	f_k(q, r) \defeq \sum_{(i, j) \menor (i_k, j_k)} q(i, j) \quad \text{for all} \quad (q, r) \in S_n \quad \text{and} \quad k \geq 1.
	\end{equation*}
	Let $(q_n, r_n)$ have the stationary distribution $\pi_n$. For each $k \geq 1$, we have
	\begin{subequations}
		\begin{align}
		&\bE\left[f_k(q_n, r_n)\right] \leq \rho \bP \left(r_n > k\right), \label{seq: expectation bound threshold} \\
		&\bP \left(r_n > k\right) \leq \e^{-n\left(\theta_n^k - 2\rho\right)} + \e^{-n\left[\theta_n^k - L(k) \beta_n - 2\rho\right]} + \e^{-n\left[j(k)\alpha_n^{\min} - 2\rho\right]}, \label{seq: tail probabilities bound threshold}
		\end{align}
	\end{subequations}
	where $\theta_n^k$ is defined as in \eqref{eq: definition of theta n k}, $\alpha_n^{\min} \defeq \min \set{\alpha_n(i)}{1 \leq i \leq m}$,
	\begin{align*}
	L(k) \defeq \max \set{\ell_i(k + 1)}{1 \leq i \leq m} \quad \text{and} \quad j(k) \defeq \min \set{j_r}{r \geq k}.
	\end{align*}
\end{proposition}

\begin{proof}
	Fix $k \geq 1$. As in the proof of Proposition \ref{prop: expectation bound jlmu}, we see that
	\begin{equation*}
	A_nf_k(q, r) \leq \lambda \ind{r > k} - \mu f_k(q, r) \quad \text{for all} \quad (q, r) \in S_n,
	\end{equation*}
	and that $f_k$ satisfies the hypothesis of Lemma \ref{lem: fubini}. Then we obtain \eqref{seq: expectation bound threshold} by taking the expectation with respect to $\pi_n$ on both sides of the latter inequality.
	
	Consider the sets $I$ and $D$ defined in the proof of Proposition \ref{prop: irreducibility and positive recurrence} and let
	\begin{align*}
	&I_k \defeq I \cap \set{(q, r) \in S_n}{r \geq k} \quad \text{and} \quad D_k \defeq D \cap \set{(q, r) \in S_n}{r > k}.
	\end{align*}
	The first step of the proof of \eqref{seq: tail probabilities bound threshold} is to establish that
	\begin{equation}
	\label{aux: probabilities of i k and d k}
	\bP\left((q_n, r_n) \in I_k\right) = \bP\left((q_n, r_n) \in D_k\right).
	\end{equation}
	
	Fix $l > k$ and consider the function $\map{g_k^l}{S_n}{[0, \infty)}$ defined by
	\begin{equation*}
	g_k^l(q, r) = \frac{1}{n}\left[(r - k)^+ \ind{r < l} + (l - k)\ind{r \geq l}\right] \quad \text{for all} \quad (q, r) \in S_n.
	\end{equation*}
	As in the proof of Proposition \ref{prop: irreducibility and positive recurrence}, we obtain
	\begin{equation}
	\label{aux: drift of g k l}
	A_ng_k^l(q, r) = \lambda \left[\indc_{I_k^l}(q, r) - \indc_{D_k^l}(q, r)\right],
	\end{equation}
	where the sets $I_k^l$ and $D_k^l$ are defined by
	\begin{align*}
	&I_k^l \defeq I \cap \set{(q, r) \in S_n}{k \leq r < l} \quad \text{and} \quad D_k^l \defeq D \cap \set{(q, r) \in S_n}{k < r \leq l}.
	\end{align*}
	
	Define $f(q, r)$ as in \eqref{eq: definition of f and g} and note that $\left|A_n\left((x, r), (x, r)\right)\right| = n\lambda + n\mu f(x, r)$. Thus,
	\begin{align*}
	\sum_{(y, s) \in S_n} \left|A_n\left((x, r), (y, s)\right)\right| g_k^l(y, s) &= 2\left|A_n\left((x, r), (y, s)\right)\right| g_k^l(x, r) + A_ng_k^l(x, r) \\
	&\leq 2n\left[\lambda + \mu f(x, r)\right] (l - k) + \lambda.
	\end{align*}
	The right-hand side has a finite mean with respect to $\pi_n$ since $nf(q_n, r_n)$ is the total number of tasks in stationarity, which is Poisson distributed with mean $n\rho$. Therefore, it follows from Lemma \ref{lem: fubini} and \eqref{aux: drift of g k l} that $\bP\left((q_n, r_n) \in I_k^l\right) = \bP\left((q_n, r_n) \in D_k^l\right)$. The sets $I_k^l$ and $D_k^l$ increase to $I_k$ and $D_k$, respectively, as $l \to \infty$. This implies \eqref{aux: probabilities of i k and d k} since
	\begin{align*}
	\bP\left((q_n, r_n) \in I_k\right) &= \lim_{l \to \infty} \bP\left((q_n, r_n) \in I_k^l\right) = \lim_{l \to \infty} \bP\left((q_n, r_n) \in D_k^l\right) = \bP\left((q_n, r_n) \in D_k\right).
	\end{align*}
	
	Now we may write
	\begin{equation}
	\label{aux: bound for probability of rn > k}
	\begin{split}
	\bP\left(r_n > k\right) &= \bP\left(r_n > k, (q_n, r_n) \in D\right) + \bP\left(r_n > k, (q_n, r_n) \notin D\right) \\
	&= \bP\left((q_n, r_n) \in D_k\right) + \bP\left(r_n > k, (q_n, r_n) \notin D\right) \\
	&= \bP\left((q_n, r_n) \in I_k\right) + \bP\left(r_n > k, (q_n, r_n) \notin D\right).
	\end{split}
	\end{equation}
	Using the definition of $I_k$, we can bound the first term on the last line by
	\begin{equation}
	\label{aux: bound for probability of rn > k part 1}
	\bP\left(f(q_n, r_n) \geq \sum_{(i, j) \mayor (i_{r_n}, j_{r_n})} \alpha_n(i), r_n > k\right) \leq \bP\left(f(q_n, r_n) \geq \theta_n^k\right).
	\end{equation}
	The second term on the last line of \eqref{aux: bound for probability of rn > k} can be bounded by
	\begin{equation}
	\label{aux: bound for probability of rn > k part 2}
	\begin{split}
	&\bP\left(\sum_{i = 1}^m q_n\left(i, \ell_i(r_n)\right) > 1 - \beta_n, r_n > k\right) + \bP\left(q_n(i_{r_n - 1}, j_{r_n - 1}) = \alpha_n(i_{r_n - 1}), r_n > k\right) \leq \\
	&\bP\left(f(q_n, r_n) \geq \theta_n^k - L(k) \beta_n\right) + \bP\left(f(q_n, r_n) \geq j(k) \alpha_n^{\min}\right).
	\end{split}
	\end{equation}
	For the last inequality, observe that the condition inside the first probability sign of the left-hand side of \eqref{aux: bound for probability of rn > k part 2} implies that
	\begin{equation*}
	\sum_{i = 1}^m q_n\left(i, \ell_i(k + 1)\right) \geq \sum_{i = 1}^m q_n\left(i, \ell_i(r_n)\right) > 1 - \beta_n,
	\end{equation*}
	and this in turn implies that
	\begin{equation*}
	f(q_n, r_n) \geq \sum_{(i, j) \mayorigual (i_k, j_k)} \alpha_n(i) - L(k) \beta_n = \theta_n^k - L(k) \beta_n,
	\end{equation*}
	since $q_n(i, j)$ is non-increasing in $j$ for all $i$. In addition, the condition inside the second probability sign of the left-hand side of \eqref{aux: bound for probability of rn > k part 2} implies that
	\begin{equation*}
	f(q_n, r_n) \geq j_{r_n - 1}q_n(i_{r_n - 1}, j_{r_n -1}) = j_{r_n - 1}\alpha_n(i_{r_n - 1}) \geq j(k) \alpha_n^{\min}.
	\end{equation*}
	We obtain \eqref{seq: tail probabilities bound threshold} from \eqref{aux: bound for probability of rn > k part 1} and \eqref{aux: bound for probability of rn > k part 2}, recalling that $nf(q_n, r_n)$ is Poisson distributed with mean $n \rho$ and applying Chernoff bounds.
\end{proof}

\subsection{Proof of the asymptotic optimality}
\label{sub: proof of the asymptotic optimality}

In this section we prove Theorem \ref{the: asymptotic optimality}. As a first step, we establish that the sequence of stationary distributions $\set{\pi_n}{n \geq 1}$ is tight both for JLMU and SLTA.

\begin{proposition}
	\label{prop: tightness of stationary distributions}
	If the load balancing policy is JLMU, then $\set{\pi_n}{n \geq 1}$ is tight in $\ell_1$. If the load balancing policy is SLTA, then $\set{\pi_n}{n \geq 1}$ is tight in $\ell_1 \times \N$.
\end{proposition}

\begin{proof}
	Suppose first that the load balancing policy is JLMU and let $q_n$ have the stationary distribution $\pi_n$ for each $n$. The sequence $\set{q_n}{n \geq 1}$ is tight with respect to the product topology since the random variables $q_n$ take values in $[0, 1]^\calI$, which is compact with respect to the product topology. Therefore, as in \cite[Lemma 2]{mukherjee2018universality}, the tightness in $\ell_1$ of $\set{q_n}{n \geq 1}$ will follow if we establish that
	\begin{equation}
	\label{aux: tightness criterion}
	\lim_{k \to \infty} \limsup_{n \to \infty} \bP\left(\sum_{(i, j) \menor (i_k, j_k)} q_n(i, j) > \varepsilon\right) = 0 \quad \text{for all} \quad \varepsilon > 0.
	\end{equation}
	
	By Proposition \ref{prop: expectation bound jlmu} and Markov's inequality, we have
	\begin{equation*}
	\bP\left(\sum_{(i, j) \menor (i_k, j_k)} q_n(i, j) > \varepsilon\right) \leq \frac{\rho}{\varepsilon}\e^{-n\left(\theta_n^k - 2\rho\right)} \quad \text{for all} \quad k, n \geq 1 \quad \text{and} \quad \varepsilon > 0.
	\end{equation*}
	For all sufficiently large $k$, the exponent on the right-hand side converges to minus infinity as $n$ grows large and $k$ is held fixed. Thus, $\set{q_n}{n \geq 1}$ is tight in $\ell_1$.
	
	Suppose now that the load balancing policy is SLTA and let $(q_n, r_n)$ have the stationary distribution $\pi_n$ for each $n$. In order to prove that $\set{(q_n, r_n)}{n \geq 1}$ is tight in $\ell_1 \times \N$, it suffices to show that $\set{q_n}{n \geq 1}$ and $\set{r_n}{n \geq 1}$ are tight in $\ell_1$ and $\N$, respectively. Indeed, if the latter properties hold, then for each $\varepsilon > 0$ there exist compact sets $K_q \subset \ell_1$ and $K_r \subset \N$ such that
	\begin{equation*}
	\max \left\{\bP\left(q_n \notin K_q\right), \bP\left(r_n \notin K_r\right)\right\} \leq \frac{\varepsilon}{2} \quad \text{for all} \quad n.
	\end{equation*}
	Therefore, the compact set $K_q \times K_r \subset \ell_1 \times \N$ satisfies
	\begin{equation*}
	\bP\left((q_n, r_n) \notin K_q \times K_r\right) \leq \bP\left(q_n \notin K_q\right) + \bP\left(r_n \notin K_r\right) \leq \varepsilon \quad \text{for all} \quad n.
	\end{equation*} 
	
	By Proposition \ref{prop: tail probabilities threshold} and Markov's inequality, we have
	\begin{align*}
	&\bP\left(\sum_{(i, j) \menor (i_k, j_k)} q_n(i, j) > \varepsilon\right) \leq \frac{\rho}{\varepsilon} \bP\left(r_n > k\right) \quad \text{for all} \quad k, n \geq 1 \quad \text{and} \quad \varepsilon > 0, \\
	&\bP\left(r_n > k\right) \leq \e^{-n\left(\theta_n^k - 2\rho\right)} + \e^{-n\left[\theta_n^k - L(k)\beta_n - 2\rho\right]} + \e^{-n\left[j(k)\alpha_n^{\min} - 2\rho\right]} \quad \text{for all} \quad k, n \geq 1.
	\end{align*}
	For all large enough $k$, the right-hand side of the second inequality is summable over $n$ and in particular vanishes with $n$. This implies that the sequences $\set{q_n}{n \geq 1}$ and $\set{r_n}{n \geq 1}$ are tight in $\ell_1$ and $\N$, respectively.
\end{proof}

We also need the following technical lemma.

\begin{lemma}
	\label{lem: overall utility}
	Fix $q \in \ell_1$ and suppose that the marginal utilities are bounded. Then
	\begin{equation*}
	u(q) = \sum_{i = 1}^m u_i(0)q(i, 0) + \sum_{(i, j) \in \calI_+} \Delta(i, j - 1)q(i, j).
	\end{equation*}
\end{lemma}

\begin{proof}
	Note that
	\begin{align*}
	u(q) &= \sum_{i = 1}^m \sum_{j = 0}^\infty u_i(j)\left[q(i, j) - q(i, j + 1)\right] \\
	&= \lim_{k \to \infty} \sum_{i = 1}^m \sum_{j = 0}^k u_i(j)\left[q(i, j) - q(i, j + 1)\right] \\
	&= \lim_{k \to \infty} \sum_{i = 1}^m \left[u_i(0)q(i, 0) + \sum_{j = 1}^k \Delta(i, j - 1)q(i, j) - u_i(k)q(i, k + 1)\right], \\
	&= \sum_{i = 1}^m u_i(0)q(i, 0) + \sum_{(i, j) \in \calI_+} \Delta(i, j - 1)q(i, j) - \lim_{k \to \infty} \sum_{i = 1}^m u_i(k)q(i, k + 1).
	\end{align*}
	The second term in the last line is absolutely convergent since the marginal utilities are bounded and $q \in \ell_1$. Moreover, there exists $M \geq 0$ such that
	\begin{align*}
	\lim_{k \to \infty} \left|u_i(k)\right| q(i, k + 1) &\leq \lim_{k \to \infty} \left|u_i(0) + \sum_{j = 0}^{k - 1} \Delta(i, j)\right| q(i, k + 1) \\
	&\leq \lim_{k \to \infty} \left(\left|u_i(0)\right| + Mk\right) q(i, k + 1) \quad \text{for all} \quad i,
	\end{align*}
	and the latter limit is zero for all $i$ since $q \in \ell_1$.
\end{proof}

Now we are ready to prove Theorem \ref{the: asymptotic optimality}.

\begin{proof}[Proof of Theorem \ref{the: asymptotic optimality}.]
	Suppose first that the load balancing policy is JLMU. It follows from Prokhorov's theorem and Proposition \ref{prop: tightness of stationary distributions} that the stationary distributions $\set{\pi_n}{n \geq 1}$ are relatively compact in $\ell_1$, thus every subsequence has a further subsequence which converges in distribution. To establish (a), it suffices to prove the following statement: if $\calK$ is an increasing sequence of natural numbers such that $\set{\pi_k}{k \in \calK}$ converges weakly to $\pi$, then $\pi$ is the Dirac measure concentrated at $q_*$. Similarly, the sequence $\set{\pi_n}{n \geq 1}$ is relatively compact in $\ell_1 \times \N$ if the load balancing policy is SLTA, and to prove (b) it suffices to establish the following statement: if $\calK$ is an increasing sequence of natural number such that $\set{\pi_k}{k \in \calK}$ converges weakly to $\pi$, then $\pi$ is the Dirac measure at $(q_*, r_*)$. Below we prove (a) and (b) in parallel, proceeding as indicated above.
	
	Fix an arbitrary increasing sequence of natural numbers $\calK$ such that $\set{\pi_k}{k \in \calK}$ converges weakly to a certain probability measure $\pi$. The following constructions use Skorokhod's representation theorem. If the load balancing policy is JLMU, then there exist random variables $q_k$ and $q$, distributed as $\pi_k$ and $\pi$, respectively, that are defined on a common probability space $(\Omega_I, \calF_I, \prob_I)$ and satisfy:
	\begin{equation*}
		\lim_{k \to \infty} \norm{q_k(\omega) - q(\omega)}_1 = 0 \quad \text{for all} \quad \omega \in \Omega_I.
	\end{equation*}
	If the load balancing policy is SLTA, then there exist random variables $(q_k, r_k)$ and $(q, r)$, with distributions given by the probability measures $\pi_k$ and $\pi$, respectively, that are defined on some common probability space $(\Omega_I, \calF_I, \prob_I)$ and satisfy:
	\begin{equation*}
		\lim_{k \to \infty} \norm{q_k(\omega) - q(\omega)}_1 = 0\ \text{and}\ \lim_{k \to \infty} r_k(\omega) = r(\omega) \quad \text{for all} \quad \omega \in \Omega_I.
	\end{equation*}
	The second limit implies that there exists a random variable $R$ such that $r_k(\omega) \leq R(\omega)$ for all $k \in \calK$ and $\omega \in \Omega_I$. Moreover, $q_k(i, j) < \alpha_k(i)$ for all $(i, j) \menorigual (i_{r_k}, j_{r_k})$ on $\Omega_I$ by the definition of the state space $S_k$ for SLTA. Indeed, recall from Remark \ref{rem: dynamics of threshold} that the latter property is preserved by arrivals and departures and observe that it holds for the empty occupancy state with $\br_k = 1$.
	
	If the load balancing policy is JLMU, then we may construct occupancy proceses $\bq_k$ on a common probability space $(\Omega, \calF, \prob)$ as in Section \ref{subsub: construction for jlmu}, such that $\bq_k(0) = q_k$ and \eqref{ass: size of classes and initial condition} holds with $q_0 = q$. If the load balancing policy is SLTA, then we may construct processes $(\bq_k, \br_k)$ on a common probability space $(\Omega, \calF, \prob)$ as in Section \ref{subsub: construction for threshold}, such that $\left(\bq_k(0), \br_k(0)\right) = (q_k, r_k)$ and the assumptions of Section \ref{subsub: technical assumptions} hold with $q_0 = q$. For both load balancing policies, we may assume by Proposition \ref{prop: relative compactness - l1} that $\bq_k$ converges in $D_{\ell_1}[0, \infty)$ to a process $\bq$ with probability one, this may require to replace $\calK$ by a subsequence.
	
	If the load balancing policy is JLMU, then Theorem \ref{the: fluid limit of jlmu} implies that $\bq(\omega)$ is the unique fluid trajectory such that $\bq(\omega, 0) = q(\omega)$ for each $\omega \in \Omega$. Moreover, by Theorem \ref{the: global asymptotic stability},
	\begin{equation}
		\label{aux: limit over time jlmu}
		\lim_{t \to \infty} \norm{\bq(\omega, t) - q_*}_1 = 0 \quad \text{for all} \quad \omega \in \Omega.
	\end{equation}
	
	If the load balancing policy is SLTA, then \eqref{seq: limit of threshold 2} and \eqref{seq: limit of threshold 3} hold on a set of probability one by Theorem \ref{the: fluid limit of threshold}. Fix any $\omega \in \Omega$ such that $\bq_k(\omega)$ converges to $\bq(\omega)$ in $D_{\ell_1}[0, \infty)$ and  such that \eqref{seq: limit of threshold 2} and \eqref{seq: limit of threshold 3} hold. Also, choose $\tau > \tau_{\eq}(s_0(\omega))$. It follows from \eqref{seq: limit of threshold 2} that
		\begin{align*}
			|\alpha(i) - \bq(\omega, T, i, j)| &\leq \sup_{t \in [\tau, T]} |\alpha(i) - \bq(\omega, t, i, j)| \\
			&\leq \lim_{k \to \infty} \sup_{t \in [\tau, T]} |\alpha(i) - \bq_k(\omega, t, i, j)| + \lim_{k \to \infty} \sup_{t \in [\tau, T]} \norm{\bq_k(\omega, t) - \bq(\omega, t)}_1 = 0.
		\end{align*}
		for all $(i, j) \mayor \sigma_*$ and $T \geq \tau$. Thus, $\bq(\omega, t, i, j) = q_*(i, j)$ for all $(i, j) \mayor \sigma_*$ and $t \geq \tau$. Similarly, we conclude from \eqref{seq: limit of threshold 3} that
		\begin{align*}
			\sum_{(i, j) \menor \sigma_*} \bq(\omega, T, i, j)\e^{\mu(T - \tau)} &\leq \sup_{t \in [\tau, T]} \sum_{(i, j) \menor \sigma_*} \bq(\omega, t, i, j)\e^{\mu(t - \tau)} \\
			&\leq \limsup_{k \to \infty} \sup_{t \in [\tau, T]} \sum_{(i, j) \menor \sigma_*} \bq_k(\omega, t, i, j)\e^{\mu(t - \tau)} \\
			&+ \lim_{k \to \infty} \sup_{t \in [\tau, T]} \norm{\bq_k(\omega, t) - \bq(\omega, t)}_1\e^{\mu(t - \tau)} \leq c(\omega, \tau)
		\end{align*}
		for all $T \geq \tau$; for the last inequality, note that in the last line $\e^{\mu(t - \tau)} \leq \e^{\mu(T - \tau)}$ for all $t \in [\tau, T]$, so the last limit equals zero. Since $T$ can be arbitrarily large, we have
		\begin{equation*}
			\lim_{t \to \infty} \sum_{(i, j) \menor \sigma_*} \bq(\omega, t, i, j) \leq \lim_{t \to \infty} c(\omega, \tau)\e^{-\mu(t - \tau)} = 0,
		\end{equation*}
		and in particular $\bq(\omega, t, i, j) \to q_*(i, j)$ as $t \to \infty$ for all $(i, j) \menor \sigma_*$. This also holds for $(i, j) = \sigma_*$ by Proposition \ref{prop: fluid limit of total number of tasks}, so we conclude that
	\begin{equation}
		\label{aux: limit over time threshold}
		\prob\left(\lim_{t \to \infty} \norm{\bq(t) - q_*}_1 = 0\right) = 1.
	\end{equation}
	
	It follows from the stationarity of $\pi_k$ that $\bq_k(t)$ is distributed as $q_k$ for all $t \geq 0$ in the case of JLMU, and that $\left(\bq_k(t), \br_k(t)\right)$ has the same distribution as $(q_k, r_k)$ for all $t \geq 0$ if the load balancing policy is SLTA. Furthermore, recall that in either case we have
	\begin{equation*}
		\prob\left(\lim_{k \to \infty} \norm{\bq_k(t) - \bq(t)}_1 = 0\right) = 1 \quad \text{for all} \quad t \geq 0 \quad \text{and} \quad q_k \Rightarrow q \quad \text{as} \quad k \to \infty,
	\end{equation*}
	which implies that $\bq(t)$ is distributed as $q$ for all $t \geq 0$. Moreover, $\bq(t)$ converges weakly to $q_*$ as $t \to \infty$ by \eqref{aux: limit over time jlmu} and \eqref{aux: limit over time threshold}. Therefore, $q$ corresponds to the Dirac probability measure concentrated at $q_*$ both for JLMU and SLTA.
	
	This completes the proof of (a). To finish the proof of (b), observe that $s_0 = \rho$ with probability one since $q_0 = q$ is equal to $q_*$ with probability one. It follows from \eqref{seq: limit of threshold 1} that
	\begin{equation*}
		\prob\left(\lim_{k \to \infty} \br_k(t) = r_*\right) = 1 \quad \text{for all} \quad t > \tau_\eq(\rho).
	\end{equation*}
	Hence, $\br_k(t)$ converges weakly to $r_*$ for all $t > \tau_\eq(\rho)$. Since $\br_k(t)$ has the same distribution as $r_k$ for all $t \geq 0$, we conclude that $r_k$ converges weakly to $r_*$. Note that $q_*$ and $r_*$ are deterministic, thus $q_k$ and $r_k$ converge to $q_*$ and $r_*$, respectively, in probability, which implies that $(q_k, r_k)$ converges to $(q_*, r_*)$ in probability. This completes the proof of (b).
	
	Next we prove the statements about the stationary overall utilities, and here exactly the same arguments apply both for JLMU and SLTA. Suppose that $q_n$ has the stationary distribution. By (a) and (b), $q_n$ converges in probability to $q_*$ in $\ell_1$, and by Lemma \ref{lem: overall utility},
	\begin{equation*}
		u(x) = \sum_{i = 1}^m u_i(0)x(i, 0) + \sum_{(i, j) \in \calI_+} \Delta(i, j - 1)x(i, j) \quad \text{for all} \quad x \in \ell_1.
	\end{equation*}
	Since the marginal utilities are bounded, there exists a constant $a$ such that
	\begin{equation*}
		|u(x) - u(y)| \leq a \norm{x - y}_1 \quad \text{for all} \quad x, y \in \ell_1.
	\end{equation*}
	This implies that $u(q_n)$ converges in probability to $u(q_*)$.
	
	Finally, observe that $|u(q_n)| \leq b + c s_n$ with
	\begin{equation*}
		b \defeq \sum_{i = 1}^m |u_i(0)|, \quad c \defeq \max_{(i, j) \in \calI} |\Delta(i, j)| \quad \text{and} \quad s_n \defeq \sum_{(i, j) \in \calI_+} q_n(i, j).
	\end{equation*}
	Since $n s_n$ is Poisson distributed with mean $n\rho$, we have
	\begin{align*}
		\bE\left[u(q_n)^2\right] &\leq b^2 + 2bc\bE\left[s_n\right] + c^2\bE\left[s_n^2\right] \\
		&\leq b^2 + 2bc\rho + c^2\rho^2 + \frac{c^2 \rho}{n} \leq b^2 + 2bc\rho + c^2\rho^2 + c^2 \rho. 
	\end{align*}
	The last expression does not depend on $n$, therefore we conclude that $\set{u(q_n)}{n \geq 1}$ is uniformly integrable. As a result, we have
	\begin{equation*}
		\lim_{n \to \infty} \bE\left[u(q_n)\right] = u(q_*).
	\end{equation*}
	This completes the proof.
\end{proof}

\subsection{Suboptimality result}
\label{sub: suboptimality result}

In this section we prove that JLMU is generally not optimal in the pre-limit, although it becomes optimal as the number of server pools grows large. For this purpose, we construct an example in which JLMU is strictly outperformed by another policy.

Specifically, consider a system with two heterogeneous server pools and assume that the associated utility functions are of the following form:
\begin{equation*}
u_1(x) \defeq a \varepsilon x \quad \text{and} \quad u_2(y) \defeq a y \ind{y < 1} + a \ind{y \geq 1} \quad \text{for all} \quad x, y \geq 0,
\end{equation*}
where $a > 0$ and $\varepsilon \in (0, 1)$ are constants. Note that
\begin{equation*}
\Delta(1, 0) < \Delta(2, 0) \quad \text{but} \quad \Delta(1, j) > \Delta(2, j) \quad \text{for all} \quad j \geq 1.
\end{equation*}
As a result, JLMU sends tasks to server pool two if and only if this server pool is empty, thus the number of tasks in server pool two is zero or one in stationarity. This assignment rule guarantees the largest increase in the aggregate utility of the system at each arrival epoch. However, this increase can be very small when tasks are sent to server pool one, whereas the decrease in the aggregate utility can be comparatively large when a departure leaves server pool two with zero tasks. Particularly, this is the case when $\varepsilon$ is small. 

Suppose that tasks arrive as a Poisson process of intensity $\lambda$ with exponential service times of mean $1 / \mu$. In addition, let $\bX$ and $\bY$ denote the number of tasks in server pools one and two, respectively. Next we provide an upper bound for the mean stationary overall utility of JLMU when the utility functions are as defined above.

As already noted, a new task is sent to server pool two if and only if $\bY = 0$. Therefore, the following statements hold.
\begin{itemize}
	\item $\bY(t) \in \{0, 1\}$ for all sufficiently large $t$.
	
	\item Conditional on $\bY(0) \in \{0, 1\}$, the process $\bY$ alone is a birth-death process with state space $\{0, 1\}$, birth rate $\lambda$ and death rate $\mu$. 
\end{itemize}
By Proposition \ref{prop: irreducibility and positive recurrence}, the process $(\bX, \bY)$ has a unique stationary distribution $(X, Y)$, and by the above observations, we have
\begin{equation*}
\bP(Y = 0) = \frac{\mu}{\lambda + \mu} \quad \text{and} \quad \bP(Y = 1) = \frac{\lambda}{\lambda + \mu}.
\end{equation*}
If we let $U \defeq u_1(X) + u_2(Y)$ denote the aggregate utility in stationarity, then 
\begin{equation*}
\bE\left[U\right] = a \varepsilon \bE\left[X\right] + a \bP(Y = 1) \leq a \varepsilon \bE\left[X + Y\right] + a \bP(Y = 1) = a \left[\varepsilon \rho + \frac{\rho}{\rho + 1}\right].
\end{equation*}

Consider now the policy that sends all tasks to server pool two. While JLMU dispatches the new tasks in a greedy fashion, this other policy is conservative, since it tries to avoid drops of $u_2(\bY)$ from $a$ to zero, by keeping a positive number of tasks in server pool two.

The mean stationary aggregate utility can be computed explicitly for the policy that we described above since $\bY$ is now an $M/M/\infty$ queue and $\bX(t) = 0$ for all sufficiently large $t$. Let $(X, Y)$ be the stationary distribution of $(\bX, \bY)$ and let $V \defeq u_1(X) + u_2(Y)$ denote the aggregate utility of the system in stationarity, we have
\begin{equation*}
\bE\left[V\right] = a \bP(Y > 0) = a\left(1 - \e^{-\rho}\right).
\end{equation*}

It is not difficult to verify that $\varepsilon$ can be chosen so that
\begin{equation*}
\varepsilon \rho + \frac{\rho}{\rho + 1} < 1 - \e^{-\rho}
\end{equation*}
for all $\rho$ in some open interval contained in $(0, \infty)$. If $\varepsilon$ and $\rho$ are chosen so that the latter inequality holds, then $\bE\left[U\right] < \bE\left[V\right]$ and thus JLMU is strictly suboptimal.

In general, the right balance between greedy and conservative actions is difficult to determine and depends intricately on the set of utility functions. However, the benefits of conservative actions, that prevent the number of tasks in server pools of specific classes from dropping below certain occupancy levels, diminish as the scale of the system grows. Indeed, the average fraction of these server pools that have less tasks than the optimal quantity decreases as the number of server pools grows, even if the assignment policy is purely greedy as JLMU; essentially, this is a consequence of the increase in the number of server pools per class and the decrease in the coefficient of variation of the total number of tasks in the system. Therefore, the associated loss in mean stationary overall utility asymptotically vanishes as the scale of the system grows.

\begin{appendices}
	
	\section{Auxiliary results}
	\label{app: auxiliary results}
	
	\begin{proof}[Construction of Section \ref{subsub: construction for threshold}.]
		First we define the functions $\eta_k$. For this purpose, suppose that $q \in Q_n$ and $r \geq 1$ are the current values of $\bq_n$ and $\br_n$, respectively. Then
		\begin{equation*}
		\calG(q, r) \defeq \begin{cases}
		\set{(i, j) \in \calI_+}{(i, j) \mayor (i_r, j_r)} & \text{if} \quad r > 1, \\
		\emptyset & \text{if} \quad r = 1,
		\end{cases} 
		\end{equation*}
		is the set of the coordinates that could belong to a server pool with a green token; i.e., a server pool of class $i$ with exactly $j - 1$ tasks has a green token if and only if $(i, j) \in \calG(q, r)$. We partition the latter set of coordinates as follows:
		\begin{align*}
		&\calG_1(q, r) \defeq \calG(q, r) \cap \set{(i, j) \in \calI_+}{i \neq i_{r - 1}}, \\
		&\calG_2(q, r) \defeq \calG(q, r) \cap \set{(i, j) \in \calI_+}{i = i_{r - 1}}.
		\end{align*}
		
		Let $\delta_q(i, j) \defeq q(i, j - 1) - q(i, j)$ denote the fraction of server pools which are of class $i$ and have exactly $j - 1$ tasks when the occupancy state is $q$, and define
		\begin{equation*}
		I(q, i, j) \defeq \left[1 - \sum_{(k, l) \menorigual (i, j)} \delta_q(k, l), 1 - \sum_{(k, l) \menor (i, j)} \delta_q(k, l)\right) \quad \text{for all} \quad (i, j) \in \calI_+.
		\end{equation*}
		These intervals form a partition of $[0, 1)$ such that the length of $I(q, i, j)$ is the fraction of server pools which are of class $i$ and have exactly $j - 1$ tasks. The length of $I(q, i, j)$ is equal to the probability of picking a server pool with coordinates $(i, j)$ uniformly at random.
		
		Note that $G_1(q, r) \defeq \sum_{(k, l) \in \calG_1(q, r)} \delta_q(k, l)$ is the fraction of server pools of class $i \neq i_{r - 1}$ with a green token. If the latter quantity is positive, then we let
		\begin{equation*}
		G_1(q, r, i, j) \defeq \left[\frac{1 - \sum_{(k, l) \menorigual (i, j), (k, l) \in \calG_1(q, r)} \delta_q\left(k, l\right)}{G_1(q, r)}, \frac{1 - \sum_{(k, l) \menor (i, j), (k, l) \in \calG_1(q, r)} \delta_q\left(k, l\right)}{G_1(q, r)}\right).
		\end{equation*}
		These intervals yield another partition of $[0, 1)$, where the length of $G_1(q, r, i, j)$ is the fraction of server pools which are of class $i$ and have exactly $j - 1$ tasks, but only among those server pools which are of class $i \neq i_{r - 1}$ and have a green token. If $G_1(q, r) = 0$, then we define $G_1(q, r, i, j) \defeq \emptyset$ for all $(i, j) \in \calI_+$. Sets $G_2(q, r, i, j)$ are defined similarly.
		
		The functions $\eta_k$ are defined as follows:
		\begin{equation*}
		\eta_k(q, r, i, j) \defeq \begin{cases}
		\ind{U_k \in G_1(q, r, i, j)} & \text{if} \quad (i, j) \in \calG_1(q, r), \\
		\ind{G_1(q, r) = 0, U_k \in G_2(q, r, i, j)} & \text{if} \quad (i, j) \in \calG_2(q, r), \\
		\ind{G_1(q, r) = G_2(q, r) = 0, I(q, i_r, j_r) \neq \emptyset} & \text{if} \quad (i, j) = (i_r, j_r), \\
		\ind{G_1(q, r) = G_2(q, r) = 0, I(q, i_r, j_r) = \emptyset, U_k \in I(q, i, j)} & \text{otherwise}.
		\end{cases}
		\end{equation*}
		Observe that $\eta_k(q, r, i, j) \in \{0, 1\}$ for all $k$ and $(i, j) \in \calI_+$, and that for a fixed $k$ there exists a unique $(i, j) \in \calI_+$ such that $\eta_k(q, r, i, j) = 1$. Also, if $(q, r)$ is the value of $(\bq_n, \br_n)$ when the $k^{\text{th}}$ task arrives to the system, then $\eta_k(q, r, i, j) = 1$ if and only if this task has to be sent to a server pool of class $i$ with exactly $j - 1$ tasks.
		
		It only remains to define the sets $I_{n, k}$ and $D_{n, k}$ that appear in \eqref{seq: functional equation for r for threshold}. The former is the set of those $\omega \in \Omega$ that satisfy:
		\begin{align*}
		&\bq\left(\tau_{n, k}^-(\omega), i, j\right) = \alpha_n(i) \quad \text{for all} \quad (i, j) \mayor \left(i_{\br\left(\tau_{n, k}^-(\omega)\right)}, j_{\br\left(\tau_{n, k}^-(\omega)\right)}\right), \\
		&n\bq\left(\tau_{n, k}^-(\omega), i_{\br\left(\tau_{n, k}^-(\omega)\right)}, j_{\br\left(\tau_{n, k}^-(\omega)\right)}\right) \geq n\alpha_n\left(i_{\br\left(\tau_{n, k}^-(\omega)\right)}\right) - 1.
		\end{align*}
		If the state of the system is $(\bq, \br)$, then the first condition means that there are no green tokens right before the $k^{\text{th}}$ task arrives, and the second condition means that the number of yellow tokens is at most one. Finally, $D_{n, k}$ is the set of those $\omega \in \Omega$ such that
		\begin{align*}
		&\br\left(\tau_{n, k}^-(\omega)\right) > 1, \\
		&n - \sum_{i = 1}^m n\bq\left(\tau_{n, k}^-(\omega), i, \ell_i\left(\br\left(\tau_{n, k}^-(\omega)\right)\right)\right) \geq n\beta_n, \\
		&\bq\left(\tau_{n, k}^-(\omega), i_{\br\left(\tau_{n, k}^-(\omega)\right) - 1}, j_{\br\left(\tau_{n, k}^-(\omega)\right) - 1}\right) < \alpha_n \left(i_{\br\left(\tau_{n, k}^-(\omega)\right) - 1}\right).
		\end{align*}
		The second condition means that the number of green tokens is larger than or equal to $n \beta_n$ right before the $k^{\text{th}}$ task arrives to the system, and the last condition means that at least one of these tokens is of class $i_{\br - 1}$.
	\end{proof}
	
	\begin{proof}[Proof of Lemma \ref{lem: set of nice sample paths}.]
		Fix $T \geq 0$ and let $M_T \defeq \lambda T(\e - 1) + 1$. The set
		\begin{equation*}
		E_T \defeq \set{\omega \in \Omega}{\calN_n^\lambda(\omega, T) \leq n M_T\ \text{for all large enough}\ n}.
		\end{equation*}
		has probability one as a consequence of the Borel-Cantelli lemma and the Chernoff bound
		\begin{equation*}
		\prob\left(\calN_n^\lambda(T) > nM_T\right) \leq \frac{\e^{\lambda T(\e - 1)n}}{\e^{M_Tn}} = \e^{-n}.
		\end{equation*}
		As a result, the following set also has probability one:
		\begin{equation*}
		\Gamma \defeq \bigcap_{T \in \N} \left(E_T \cap \Gamma_\infty\right).
		\end{equation*}
		
		Note that for each $\omega \in \Gamma$ and $T \geq 0$ there exists $n_T^1(\omega)$ such that $\calN_n^\lambda(\omega, T) \leq n M_T$ for all $n \geq n_T^1(\omega)$. Fix $T \geq 0$, $\omega \in \Gamma$ and $0 < \varepsilon < \min \set{\alpha(i)}{1 \leq i \leq m}$. Next we establish that there exist $j_T(\omega)$ and $n_T(\omega)$ such that
		\begin{equation}
		\label{aux: technical lemma for relative compactness in l1}
		\bq_n\left(\omega, t, i, j_T(\omega)\right) \leq \varepsilon < \alpha_n(i) \quad \text{for all} \quad t \in [0, T], \quad 1 \leq i \leq m \quad \text{and} \quad n \geq n_T(\omega).
		\end{equation}
		If the load balancing policy is JLMU, then \eqref{eq: definition of sigma} and the above statement imply that $\sigma_j\left(\bq_n(\omega, t)\right) \leq j_T(\omega)$ for all $t \in [0, T]$ and $n \geq n_T(\omega)$, which in turn implies that the claim of the lemma holds. Suppose instead that SLTA is used and let
		\begin{equation*}
		r_T(\omega) \defeq \min \set{k \geq 1}{(i_k, j_k) \menorigual \left(i, j_T(\omega)\right)\ \text{for all}\ i}.
		\end{equation*}
		It follows from \eqref{eq: additional property of threshold within gamma infty} that $\br_n(\omega, t) \leq \max\{r_T(\omega), R(\omega)\} \eqdef R_T(\omega)$ for all $t \in [0, T]$ and $n \geq n_T(\omega)$. Furthermore, the latter property and Remark \ref{rem: dynamics of threshold} imply that $\calA_n(\omega, t, i, j) = 0$  for all $t \in [0, T]$, $n \geq n_T(\omega)$ and $(i, j) \menor (i_{R_T(\omega) + 1}, j_{R_T(\omega) + 1})$. Therefore, the claim of the lemma also holds for SLTA.
		
		The following arguments apply both for JLMU and SLTA. Below we omit $\omega$ from the notation for brevity. Since $q_0 \in \ell_1$, there exists $j_0$ such that $q_0(i, j_0) \leq \varepsilon / 4$ for all $i$, and by \eqref{ass: size of classes and initial condition} and \eqref{seq: convergence of initial conditions}, there exists $n_T \geq n_T^1$ such that
		\begin{equation*}
		\bq_n(0, i, j_0) \leq \frac{\varepsilon}{2} < \varepsilon < \alpha_n(i) \quad \text{for all} \quad 1 \leq i \leq m \quad \text{and} \quad n \geq n_T.
		\end{equation*}
		Define $j_T \defeq j_0 + \ceil{2 M_T / \varepsilon} - 1$. Also, fix $t \in [0, T]$, $1 \leq i \leq m$ and $n \geq n_T$. We have
		\begin{align*}
		\sum_{j = j_0}^{j_T} n\left[\bq_n(t, i, j) - \frac{\varepsilon}{2}\right] \leq \sum_{j = j_0}^{j_T} n\left[\bq_n(t, i, j) - \bq_n(0, i, j)\right] \leq \sum_{j = j_0}^{j_T} n\calA_n(t, i, j) \leq nM_T.
		\end{align*}
		For the first inequality, observe that $\bq_n(0, i, j) \leq \bq_n(0, i, j_0) \leq \varepsilon / 2$ for all $j \geq j_0$, and for the last inequality, note that $\calN_n^\lambda(t) \leq \calN_n^\lambda(T) \leq nM_T$ because $n \geq n_T^1$. Since $\bq_n(t, i, j)$ is non-increasing in $j$, we conclude that
		\begin{equation*}
		\bq_n(t, i, j_T) \leq \frac{1}{j_T - j_0 + 1}\sum_{j = j_0}^{j_T} \bq_n(t, i, j) \leq \frac{M_T + \frac{\varepsilon}{2}\left(j_T - j_0 + 1\right)}{j_T - j_0 + 1} < \varepsilon < \alpha_n(i).
		\end{equation*}
		This completes the proof.
	\end{proof}
	
	\begin{proof}[Proof of Lemma \ref{lem: skorokhod mappings with upper reflecting barrier}]
		Consider a function $\bx \in D[0, \infty)$ such that $\bx(0) \geq 0$.  Then there exist unique $\by, \bz \in D[0, \infty)$ such that the following statements hold.
		\begin{enumerate}
			\item[(a')] $\bz(t) = \bx(t) + \by(t) \geq 0$ for all $t \geq 0$.
			
			\item[(b')] $\by$ is non-decreasing and $\by(0) = 0$.
			
			\item[(c')] $\by$ is flat off $\set{t \geq 0}{\bz(t) = 0}$, i.e., $\dot{\by}(t)\ind{\bz(t) > 0} = 0$ almost everywhere.
		\end{enumerate}
		The map $(\Psi, \Phi)$ such that $\Psi(\bx) = \by$ and $\Phi(\bx) = \bz$ is called the one-dimensional Skorokhod mapping with lower reflecting barrier at zero and satisfies
		\begin{equation*}
		\Psi(\bx)(t) = \sup_{s \in [0, t]} \left[-\bx(s)\right]^+ \quad \text{and} \quad \Phi(\bx)(t) = \bx(t) + \Psi(\bx)(t).
		\end{equation*}
		Also, if $\bx, \by \in D[0, \infty)$ satisfy $\bx(0), \by(0) \geq 0$, then for each $T \geq 0$, we have
		\begin{equation*}
		\norm{\Psi(\bx) - \Psi(\by)}_T \leq \norm{\bx - \by}_T \quad \text{and} \quad \norm{\Phi(\bx) - \Phi(\by)}_T \leq 2 \norm{\bx - \by}_T.
		\end{equation*}
		The latter definition and properties of the Skorokhod mapping with lower reflecting barrier at zero can be found in \cite[Theorem 6.1]{chen2013fundamentals} and \cite[Lemma 6.14]{karatzas2014brownian}.
		
		Suppose $\bx$ is as in the statement of the lemma. Define $\by$ and $\bz$ as in \eqref{eq: explicit expressions for reflection mappings}; i.e.,
		\begin{equation*}
		\by(t) \defeq \sup_{s \in [0, T]} \left[\bx(s) - \alpha\right]^+ = \Psi(\alpha - \bx)(t) \quad \text{and} \quad \bz(t) \defeq \bx(t) - \by(t) = \alpha - \Phi(\alpha - \bx)(t).
		\end{equation*}
		Properties (b') and (c') imply (b) and (c). Also, (a) holds since $\bz = \alpha - \Phi(\alpha - \bx) \leq \alpha$ by (a') and
		$\bz = \bx - \by$ by definition.
		
		Conversely, suppose that $\by$ and $\bz$ satisfy (a)-(c). Then $\alpha - \bz = \alpha - \bx + \by \geq 0$ and $\by$ satisfies (b') and (c') with $\bz$ replaced by $\alpha - \bz$, hence we conclude that $\by = \Psi(\alpha - \bx)$ and $\alpha - \bz = \Phi(\alpha - \bx)$. Therefore, $\by$ and $\bz$ are as in \eqref{eq: explicit expressions for reflection mappings}.
		
		The Lipschitz properties of $\Psi_\alpha$ and $\Phi_\alpha$ follow from the Lipschitz properties of $\Psi$ and $\Phi$ and the identities $\Psi_\alpha(\bx) = \Psi(\alpha - \bx)$ and $\Phi_\alpha(\bx) = \alpha - \Phi(\alpha - \bx)$.
	\end{proof}
	
	\begin{proof}[Proof of Lemma \ref{lem: fubini}]
		By assumption and Tonelli's theorem,
		\begin{align*}
		\sum_{x, y \in S_n} \pi_n(x) \left|A_n(x, y)\right| f(y) &= \sum_{x \in S_n} \pi_n(x) \sum_{y \in S_n} \left|A_n(x, y)\right| f(y) \\
		&= \bE\left[\sum_{y \in S_n} \left|A_n(x_n, y)\right|f(y)\right] < \infty. 
		\end{align*}
		Therefore, we may use Fubini's theorem, which yields
		\begin{align*}
		0 = \sum_{y \in S_n} f(y) \sum_{x \in S_n} \pi_n(x) A_n(x, y) = \sum_{x \in S_n} \pi_n(x) \sum_{y \in S_n} A_n(x, y) f(y) = \bE\left[A_nf(x_n)\right].
		\end{align*}
		The first equality follows from the stationarity of $\pi_n$.
	\end{proof}
	
	\section{Relative compactness}
	\label{app: relative compactness}
	
	Consider the metric
	\begin{equation*}
	d(x, y) \defeq \sum_{(i, j) \in \calI} \frac{\min \{|x(i, j) - y(i, j)|, 1\}}{2^j} \quad \text{for all} \quad x, y \in \R^\calI,
	\end{equation*}
	which induces the product topology on $\R^\calI$. For each $T \geq 0$, let $D_{\R^\calI}[0, T]$ denote the space of all c\`adl\`ag functions on $[0, T]$ with values in $\R^\calI$ and equip this space with the uniform metric, defined by
	\begin{equation*}
	\varrho_T(\bx, \by) \defeq \sup_{t \in [0, T]} d\left(\bx(t), \by(t)\right) \quad \text{for all} \quad \bx, \by \in D_{\R^\calI}[0, T].
	\end{equation*}
	The topology of uniform convergence over compact sets on $D_{\R^\calI}[0, \infty)$ is compatible with the metric defined by
	\begin{equation*}
	\varrho_\infty(\bx, \by) \defeq \sum_{k = 0}^\infty \frac{\min\{\varrho_k(\bx, \by), 1\}}{2^k} \quad \text{for all} \quad \bx, \by \in D_{\R^\calI}[0, \infty).
	\end{equation*}
	Since $D_{\R^\calI}[0, \infty)$ is metrizable, a subset of $D_{\R^\calI}[0, \infty)$ is relatively compact if and only if every sequence within the subset has a convergent subsequence. The same holds for the metric spaces $D_{\R^\calI}[0, T]$ defined above.
	
	The following arguments apply both for JLMU and SLTA. Fix some arbitrary time horizon $T \geq 0$. In order to prove Proposition \ref{prop: relative compactness - product topology}, we first demonstrate that the sequences $\set{\calA_n}{n \geq 1}$, $\set{\calD_n}{n \geq 1}$ and $\set{\bq_n}{n \geq 1}$ are relatively compact in the space $D_{\R^\calI}[0, T]$ with probability one, and that every convergent subsequence has a componentwise Lipschitz limit. Here we are actually referring to the restrictions to the interval $[0, T]$ of the processes $\calA_n$, $\calD_n$ and $\bq_n$, but the symbols $|_{[0, T]}$ have been omitted in order to maintain a lighter notation; this is also done in the sequel when there is no risk of confusion. The latter properties are established using the methodological framework \cite{bramson1998state}. The following result defines a set of probability one where the relative compactness holds.
	
	\begin{lemma}
		\label{lem: definition of Gamma_T}
		There exists a set of probability one $\Gamma_T \subset \Gamma_0$ where:
		\begin{subequations}
			\begin{align}
			&\lim_{n \to \infty} \norm{\bq_n(0) - q_0}_1 = 0, \label{seq: convergence of initial conditions - [0, T]} \\
			&\lim_{n \to \infty} \sup_{t \in [0, T]} n^\gamma\left|\frac{1}{n} \calN_n^\lambda(t) - \lambda t\right| = 0, \label{seq: fslln for arrivals - [0, T]} \\
			&\lim_{n \to \infty} \sup_{t \in [0, \mu j T]} n^\gamma\left|\frac{1}{n} \calN_{(i, j)}(nt) - t\right| = 0 \quad \text{for all} \quad (i, j) \in \calI_+, \label{seq: fslln for departures - [0, T]}
			\end{align}
		\end{subequations}
		for all $\gamma \in [0, 1 / 2)$. Also, if the load balancing policy is SLTA, then apart from the latter properties, there exists a random variable $R \geq 1$ such that
		\begin{equation*}
		\br_n(0) \leq R \quad \text{and} \quad \bq_n(0, i, j) < \alpha_n(i) \quad \text{for all} \quad (i, j) \menorigual \left(i_{\br_n(0)}, j_{\br_n(0)}\right) \quad \text{and} \quad n.
		\end{equation*}
	\end{lemma}
	
	\begin{proof}
		This result is a straightforward consequence of \eqref{ass: size of classes and initial condition}, \eqref{ass: boundedness and goodness of r} and the refined strong law of large numbers for the Poisson process proven in \cite[Lemma 4]{goldsztajn2021learning}.
	\end{proof}
	
	We fix some $\omega \in \Gamma_T$, which we omit from the notation for brevity. Next we prove that the sequences $\set{\calA_n}{n \geq 1}$ and $\set{\calD_n}{n \geq 1}$ are relatively compact subsets of $D_{\R^\calI}[0, T]$ and that the limit of every convergent subsequence is a function with Lipschitz coordinates. If JLMU is used, it then follows from \eqref{eq: definition of q for jlmu} and \eqref{seq: convergence of initial conditions - [0, T]} that the sequence $\set{\bq_n}{n \geq 1}$ has the same properties. If SLTA is used, then we need to invoke \eqref{seq: definition of q for threshold} instead of \eqref{eq: definition of q for jlmu}. 
	
	The following characterization of relative compactness with respect to $\varrho_T$ will be useful. Let $D[0, T]$ denote the space of real c\`adl\`ag functions on $[0, T]$, with the uniform norm:
	\begin{equation*}
	\norm{\bx}_T = \sup_{t \in [0, T]} |\bx(t)| \quad \text{for all} \quad \bx \in D[0, T].
	\end{equation*}
	Observe that a sequence of functions $\bx_n \in D_{\R^\calI}[0, T]$ converges to $\bx \in D_{\R^\calI}[0, T]$ with respect to $\varrho_T$ if and only if $\bx_n(i, j)$ converges to $\bx(i, j)$ with respect to the uniform norm for all $(i, j) \in \calI$. Moreover, we have the following proposition.
	
	\begin{proposition}
		\label{prop: relative compactness from coordinates}
		The sequence $\set{\bx_n}{n \geq 1}$ is relatively compact in $D_{\R^\calI}[0, T]$ if and only if $\set{\bx_n(i, j)}{n \geq 1}$ is a relatively compact subset of $D[0, T]$ for all $(i, j) \in \calI$.
	\end{proposition}
	
	\begin{proof}
		We only need to prove the converse, so assume that $\set{\bx_n(i, j)}{n \geq 1}$ is relatively compact for all $(i, j) \in \calI$. Given an increasing sequence $\calK \subset \N$, we need to establish that there exists a subsequence of $\set{\bx_k}{k \in \calK}$ that converges with respect to $\varrho_T$. For this purpose, we may construct a family of sequences $\set{\calK_j}{j \geq 0}$ with the following properties.
		\begin{enumerate}
			\item[(a)] $\calK_{j + 1} \subset \calK_j \subset \calK$ for all $j \geq 0$.
			
			\item[(b)] $\set{\bx_k(i, j)}{k \in \calK_j}$ has a limit $\bx(i, j) \in D[0, T]$ for each $(i, j) \in \calI$.
		\end{enumerate}
		
		Define $\set{k_l}{l \geq 1} \subset \calK$ such that $k_l$ is the $l^{\text{th}}$ element of $\calK_l$. Then
		\begin{equation*}
		\lim_{l \to \infty} \norm{\bx_{k_l}(i, j) - \bx(i, j)}_T = 0 \quad \text{for all} \quad (i, j) \in \calI.
		\end{equation*} 
		Let $\bx \in D_{\R^\calI}[0, T]$ be the function with coordinates the limiting functions $\bx(i, j)$ introduced in (b). Then $\bx_{k_l}$ converges to $\bx$ with respect to $\varrho_T$.
	\end{proof}
	
	As a result, it suffices to establish that $\set{\calA_n(i, j)}{n \geq 1}$ and $\set{\calD_n(i, j)}{n \geq 1}$ are relatively compact in $D[0, T]$ for all $(i, j) \in \calI$. Consider the sets
	\begin{equation*}
	L_M \defeq \set{\bx \in D[0, T]}{\bx(0) = 0\ \text{and}\ |\bx(t) - \bx(s)| \leq M|t - s|\ \text{for all}\ s, t \in [0, T]}
	\end{equation*}
	which are compact for each $M > 0$ by the Arzel\'a-Ascoli theorem. For each $(i, j) \in \calI$, we prove that there exists $M_j$ such that $\calA_n(i, j)$ and $\calD_n(i, j)$ approach $L_{M_j}$ as $n$ grows to large. Then we use the compactness of $L_{M_j}$ to show that $\set{\calA_n(i, j)}{n \geq 1}$ and $\set{\calD_n(i, j)}{n \geq 1}$ are relatively compact subsets of $D[0, T]$. To this end, we introduce the spaces
	\begin{equation*}
	L_M^\varepsilon \defeq \set{\bx \in D[0, T]}{\bx(0) = 0\ \text{and}\ |\bx(t) - \bx(s)| \leq M|t - s| + \varepsilon\ \text{for all}\ s, t \in [0, T]}.
	\end{equation*}
	
	\begin{lemma}
		\label{lem: lemma from bramson}
		If $\bx \in L_M^\varepsilon$, then there exists $\by \in L_M$ such that $\norm{\bx - \by}_T \leq 4\varepsilon$.
	\end{lemma}
	
	The above lemma is a restatement of \cite[Lemma 4.2]{bramson1998state}, and together with the next lemma, it implies that for each $(i, j) \in \calI$ there exists a constant $M_j$ such that $\calA_n(i, j)$ and $\calD_n(i, j)$ approach $L_{M_j}$ as $n$ grows large.
	
	\begin{lemma}
		\label{lem: approximate lipschitz property}
		For each $(i, j) \in \calI$ there exist $M_j > 0$ and $\set{\varepsilon_n(i, j) > 0}{n \geq 1}$ such that $\calA_n(i, j)$ and $\calD_n(i, j)$ lie in $L_{M_j}^{\varepsilon_n(i, j)}$ for all $n$ and $\varepsilon_n(i, j) \to 0$ as $n$ grows large.
	\end{lemma}
	
	\begin{proof}
		If $j = 0$, then $\calA_n(i, j)$ and $\calD_n(i, j)$ are identically zero for all $i$ and $n$, so the claim holds trivially. Therefore, let us fix $(i, j) \in \calI_+$, in this case we have
		\begin{equation*}
		\left|\calA_n(t, i, j) - \calA_n(s, i, j)\right| \leq \frac{1}{n} \left|\calN_n^\lambda(t) - \calN_n^\lambda(s)\right| \leq \lambda\left|t - s\right| + 2\sup_{r \in [0, T]} \left|\frac{1}{n}\calN_n^\lambda(r) - \lambda r\right|
		\end{equation*}
		for all $s, t \in [0, T]$. It follows from Lemma \ref{lem: definition of Gamma_T} that there exists a vanishing sequence of positive real numbers $\set{\delta_n^1 > 0}{n \geq 1}$ such that
		\begin{align*}
		\left|\calA_n(t, i, j) - \calA_n(s, i, j)\right| \leq \lambda |t - s| + \delta_n^1 \quad \text{for all} \quad s, t \in [0, T].
		\end{align*}
		
		Consider now the non-decreasing function defined as
		\begin{align*}
		\boldf_n(t) \defeq \int_0^t \mu j\left[\bq_n(s, i, j) - \bq_n(s, i, j + 1)\right]ds \quad \text{for all} \quad t \in [0, T].
		\end{align*}
		Note that $\boldf_n(0) = 0$ and $|\boldf_n(t) - \boldf_n(s)| \leq \mu j|t - s|$ for all $s, t \in [0, T]$, so in particular, $\boldf_n(T) \leq \mu j T$. For all $s, t \in [0, T]$, we have
		\begin{align*}
		\left|\calD_n(t, i, j) - \calD_n(s, i, j)\right| &= \frac{1}{n} \left|\calN_{(i, j)}\left(n \boldf_n(t)\right) - \calN_{(i, j)}\left(n \boldf_n(s)\right)\right| \\
		&\leq |\boldf_n(t) - \boldf_n(s)| + 2\sup_{r \in [0, T]} \left|\frac{1}{n}\calN_{(i, j)}\left(n \boldf_n(r)\right) - \boldf_n(r)\right| \\
		&\leq \mu j|t - s| + 2\sup_{r \in [0, \mu j T]} \left|\frac{1}{n}\calN_{(i, j)}(n r) - r\right|.
		\end{align*}
		By \eqref{seq: fslln for departures - [0, T]}, there exists a vanishing sequence $\set{\delta_{n}^2(i, j) > 0}{n \geq 1}$ such that
		\begin{align*}
		\left|\calD_n(t, i, j) - \calD_n(s, i, j)\right| \leq \mu j|t - s| + \delta_{n}^2(i, j) \quad \text{for all} \quad s, t \in [0, T].
		\end{align*}
		The result follows setting $M_j \defeq \max\{\lambda, \mu j\}$ and $\varepsilon_n(i, j) \defeq \max\{\delta_n^1, \delta_{n}^2(i, j)\}$.
	\end{proof}
	
	We now demonstrate that $\set{\calA_n}{n \geq1}$, $\set{\calD_n}{n \geq 1}$ and $\set{\bq_n}{n \geq 1}$ are relatively compact subsets of $D_{\R^\calI}[0, T]$ with probability one, and that the limit of every convergent subsequence is componentwise Lipschitz.
	
	\begin{lemma}
		\label{lem: relative compactness of restricted sample paths}
		The sequences $\set{\calA_n(\omega)}{n \geq 1}$, $\set{\calD_n(\omega)}{n \geq 1}$ and $\set{\bq_n(\omega)}{n \geq 1}$ are all relatively compact in $D_{\R^\calI}[0, T]$ for all $\omega \in \Gamma_T$. Also, they have the property that the limit of every convergent subsequence is a function with Lipschitz coordinates. 
	\end{lemma}
	
	\begin{proof}
		We fix an arbitrary $\omega \in \Gamma_T$ and we omit it from the notation for brevity. As noted above, it suffices to prove that $\set{\calA_n(i, j)}{n \geq 1}$ and $\set{\calD_n(i, j)}{n \geq 1}$ are relatively compact subsets of $D[0, T]$ for all $(i, j) \in \calI$. We fix some $(i, j) \in \calI$ and demonstrate that $\set{\calA_n(i, j)}{n \geq 1}$ is relatively compact in $D[0, T]$, and that the limit of every convergent subsequence is Lipschitz. The same arguments can be used if $\calA_n$ is replaced by $\calD_n$.
		
		Let $M_j$ and $\set{\varepsilon_n(i, j)}{n \geq 1}$ be as in the statement of Lemma \ref{lem: approximate lipschitz property}. It follows from Lemma \ref{lem: lemma from bramson} that for each $n$ there exists $\bx_n(i, j) \in L_{M_j}$ such that
		\begin{equation*}
		\norm{\calA_n(i, j) - \bx_n(i, j)}_T \leq 4\varepsilon_n(i, j).
		\end{equation*}
		Recall that $L_{M_j}$ is compact, thus every increasing sequence of natural numbers has a subsequence $\calK$ such that $\set{\bx_k(i, j)}{k \in \calK}$ converges to a function $\bx \in L_{M_j}$. Furthermore,
		\begin{equation*}
		\limsup_{k \to \infty} \norm{\calA_k(i, j) - \bx(i, j)}_T \leq \lim_{k \to \infty} 4\varepsilon_k(i, j) + \lim_{k \to \infty} \norm{\bx_k(i, j) - \bx(i, j)}_T = 0,
		\end{equation*}
		where the limits are taken along $\calK$. This shows that every subsequence of $\set{\calA_n(i, j)}{n \geq 1}$ has a further subsequence which converges to a Lipschitz function.
	\end{proof}
	
	We are now ready to prove Proposition \ref{prop: relative compactness - product topology}
	
	\begin{proof}[Proof of Proposition \ref{prop: relative compactness - product topology}]
		Recall that $\Gamma_T$ has probability one for all $T \geq 0$. Hence,
		\begin{equation*}
		\Gamma_\infty \defeq \bigcap_{T \in \N} \Gamma_T
		\end{equation*}
		has probability one as well. Also, \eqref{eq: properties within gamma infty} and \eqref{eq: additional property of threshold within gamma infty} hold within $\Gamma_\infty$ by Lemma \ref{lem: definition of Gamma_T}.
		
		Next we fix an arbitrary $\omega \in \Gamma_\infty$, which we omit from the notation for brevity, and we prove that $\set{\bq_n}{n \geq 1}$ is a relatively compact subset of $D_{\R^\calI}[0, \infty)$ such that the limit of every convergent subsequence has locally Lipschitz coordinate functions. Exactly the same arguments can be used if $\bq_n$ is replaced by $\calA_n$ or $\calD_n$.
		
		Fix a sequence $\calK \subset \N$. We must show that $\set{\bq_k}{k \in \calK}$ has a subsequence which converges uniformly over compact sets to a function with locally Lipschitz coordinates. To this end, we may construct a family of sequences $\set{\calK_T}{T \in \N}$ such that:
		\begin{enumerate}
			\item[(a)] $\calK_{T + 1} \subset \calK_T \subset \calK$ for all $T \in \N$;
			
			\item [(b)] for each $T \in \N$, there exists $\bq_T \in D_{\R^\calI}[0, T]$ such that the coordinates of $\bq_T$ are Lipschitz and $\varrho_T\left(\bq_k|_{[0, T]}, \bq_T\right) \to 0$ as $k \to \infty$ with $k \in \calK_T$.
		\end{enumerate}
		
		Let $k_l$ denote the $l^{\text{th}}$ element of $\calK_l$. It follows from (a) and (b) that
		\begin{equation}
		\label{eq: convergence over compact sets}
		\lim_{l \to \infty} \varrho_T\left(\bq_{k_l}|_{[0, T]}, \bq_T\right) = 0 \quad \text{for all} \quad T \in \N.
		\end{equation}
		Note that $\bq_{k_l}(t) \to \bq_T(t)$ with $l$ for all $t \leq T$, so $\bq_S(t) = \bq_T(t)$ for all $t \leq S, T$. Thus, we may define a function $\bq \in D_{\R^\calI}[0, \infty)$ such that $\bq(t) = \bq_T(t)$ for all $t \leq T$. Moreover, (b) implies that $\bq$ has locally Lipschitz coordinates and \eqref{eq: convergence over compact sets} means that $\{\bq_{k_l} : l \geq 1\}$ converges uniformly over compact sets to $\bq$.
	\end{proof}
	
	\section{Limiting behavior of SLTA}
	\label{app: limiting behavior of threshold}
	
	\begin{proof}[Proof of Proposition \ref{prop: fluid limit of total number of tasks}.]
		We fix an arbitrary $\omega \in \Gamma$, which is omitted from the notation for brevity. It follows from \eqref{seq: definition of q for threshold} that
		\begin{align*}
		\bs_n &= \bs_n(0) + \sum_{(i, j) \in \calI_+} \left[\calA_n(i, j) - \calD_n(i, j)\right] = \bs_n(0) + \frac{1}{n}\calN_n^\lambda(t) - \sum_{(i, j) \in \calI_+} \calD_n(i, j) \quad \text{for all} \quad n.
		\end{align*}
		
		By Proposition \ref{prop: relative compactness - l1}, every increasing sequence of natural numbers has a subsequence $\calK$ such that the sequences $\set{\calD_k}{k \in \calK}$ and $\set{\bq_k}{k \in \calK}$ converge in $D_{\ell_1}[0, \infty)$ to certain functions $\bd$ and $\bq$, respectively. It follows from \eqref{seq: fslln for departures} that
		\begin{equation*}
		\bd(t, i, j) = \int_0^t \mu j \left[\bq(s, i, j) - \bq(s, i, j + 1)\right]ds \quad \text{for all} \quad t \geq 0 \quad \text{and} \quad (i, j) \in \calI_+.
		\end{equation*}
		Moreover, $\set{\bs_k}{k \in \calK}$ converges uniformly over compact sets to
		\begin{equation*}
		\tilde{\bs} \defeq \sum_{(i, j) \in \calI_+} \bq(i, j).
		\end{equation*}
		
		It follows from \eqref{seq: convergence of initial conditions} and \eqref{seq: fslln for arrivals} that
		\begin{equation*}
		\lim_{k \to \infty} \bs_k(0) = s_0 \quad \text{and} \quad \lim_{k \to \infty} \sup_{t \in [0, T]} \left|\frac{1}{k}\calN_k^\lambda(t) - \lambda t\right| = 0 \quad \text{for all} \quad T \geq 0.
		\end{equation*}
		Furthermore, the fact that $\set{\calD_k}{k \in \calK}$ converges to $\bd$ in $D_{\ell_1}[0, \infty)$ implies that
		\begin{align*}
		&\lim_{k \to \infty} \sup_{t \in [0, T]} \left|\sum_{(i, j) \in \calI_+} \calD_k(t, i, j) - \int_0^t \mu \tilde{\bs}(s)ds\right| = \\
		&\lim_{k \to \infty} \sup_{t \in [0, T]} \left|\sum_{i = 1}^m \sum_{j = 1}^\infty \left[\calD_k(t, i, j) - \int_0^t \mu j\left[\bq(s, i, j) - \bq(s, i, j + 1)\right]ds\right]\right| = \\
		&\lim_{k \to \infty} \sup_{t \in [0, T]} \left|\sum_{(i, j) \in \calI_+} \left[\calD_k(t, i, j) - \bd(t, i, j)\right]\right| = 0 \quad \text{for all} \quad T \geq 0.
		\end{align*}
		
		We conclude that $\tilde{\bs}$ satisfies
		\begin{equation*}
		\tilde{\bs}(t) = s_0 + \lambda t - \int_0^t \mu \tilde{\bs}(s)ds \quad \text{for all} \quad t \geq 0.
		\end{equation*}
		Equivalently, $\tilde{\bs}$ solves \eqref{eq: fluid limit of total number of tasks}, and therefore $\tilde{\bs} = \bs$. This demonstrates that every subsequence of $\set{\bs_n}{n \geq 1}$ has a further subsequence that converges to $\bs$ uniformly over compact sets. The claim of the proposition follows from this fact.
	\end{proof}
	
	\begin{proof}[Proof of Proposition \ref{prop: asymptotic dynamical property of tail mass processes}.]
		Let us omit $\omega$ from the notation for brevity. Recall that tasks are always assigned to server pools with coordinates $(i, j) \mayorigual (i_{\br_n}, j_{\br_n})$, as explained in Remark~\ref{rem: dynamics of threshold}. By (b), $(i_{\br_k}, j_{\br_k}) \mayorigual (i_r, j_r)$ along the interval $[t_0, t_1]$ for each $k \in \calK$. Moreover, (b) implies that the system has green tokens when the equality holds. Therefore, tasks are only assigned to server pools with coordinates $(i, j) \mayor (i_r, j_r)$ along the interval $[t_0, t_1]$ in all the systems comprised in $\calK$. Thus, $\calA_k(t, i, j) = \calA_k(t_0, i, j)$ for all $t \in [t_0, t_1]$ and $(i, j) \menorigual (i_r, j_r)$. We conclude from \eqref{seq: definition of q for threshold} that
		\begin{equation*}
		\bq_k(t, i, j) = \bq_k(t_0, i, j) - \left[\calD_k(t, i , j) - \calD_k(t_0, i, j)\right]
		\end{equation*}
		for all $t \in [t_0, t_1]$, $(i, j) \menorigual (i_r, j_r)$ and $k \in \calK$. As in the proof of Proposition \ref{prop: fluid limit of total number of tasks}, we obtain
		\begin{equation*}
		\bq(t, i, j) = \bq(t_0, i, j) - \int_{t_0}^t \mu j \left[\bq(s, i, j) - \bq(s, i, j + 1)\right]ds
		\end{equation*}
		for all $t \in [t_0, t_1]$, $(i, j) \menorigual (i_r, j_r)$ and $k \in \calK$. It follows that $\bq(i, j)$ is differentiable for all $(i, j) \menorigual (i_r, j_r)$ because $\bq$ has continuous coordinates by Proposition \ref{prop: relative compactness - product topology}. Furthermore, the system of differential equations in the statement of the proposition holds.
		
		Note that $\set{\bv_k(r)}{k \geq 1}$ converges uniformly over compact sets to
		\begin{equation*}
		\bv(r) \defeq \sum_{(i, j) \menorigual (i_r, j_r)} \bq(i, j)
		\end{equation*}
		by (a) and the definition of $\bv_k$. Let $J_i \defeq \min \set{j \geq 1}{(i, j) \menorigual (i_r, j_r)}$ and note that
		\begin{align*}
		\dot{\bv}(t, r) &= \sum_{i = 1}^m \sum_{j = J_i}^\infty -\mu j \left[\bq(t, i, j) - \bq(t, i, j + 1)\right] \\
		&= \sum_{i = 1}^m \left[-\mu (J_i - 1) \bq(t, i, J_i) - \mu \sum_{j = J_i}^\infty \bq(t, i, j)\right] = -\mu \sum_{i = 1}^m (J_i - 1)\bq(t, i, J_i) - \mu \bv(t, r). 
		\end{align*}
		Since the first term on the right-hand side is non-positive, we have
		\begin{equation*}
		\bv(t, r) \leq \bv(t_0, r) \e^{-\mu(t - t_0)} < \bs(t_0)\e^{-\mu(t - t_0)} \quad \text{for all} \quad t \in [t_0, t_1].
		\end{equation*}
		This completes the proof.
	\end{proof}

	\begin{proof}[Proof of Proposition \ref{prop: upper bound for br}.]
		Define
		\begin{equation*}
		\alpha_{\min} \defeq \min_{1 \leq i \leq m} \alpha(i), \quad J \defeq \floor{\frac{\rho + 1}{\alpha_{\min}}} \quad \text{and} \quad r_\bound \defeq \max \set{r \geq 2}{j_{r - 1} \leq J}. 
		\end{equation*}
		The function $\tau_\bound$ mentioned in the statement of the proposition is defined by
		\begin{equation*}
		\tau_\bound(s) = \begin{cases}
		\frac{1}{\mu}\left[\log\left(\frac{s - \rho}{\sum_{(i, j) \mayorigual \sigma_*} - \rho}\right)\right]^+ + \frac{J - j_{r_*}}{\mu}\left(r_\bound - r_*\right)& \text{if} \quad s > \rho, \\
		\frac{J - j_{r_*}}{\mu}\left(r_\bound - r_*\right) & \text{if} \quad s \leq \rho.
		\end{cases}
		\end{equation*}
		
		For brevity, we omit $\omega$ from the notation, and we let
		\begin{equation*}
		\Sigma(r) \defeq \sum_{(i, j) \mayor (i_r, j_r)} \alpha(i) \quad \text{and} \quad \Sigma_n(r) \defeq \sum_{(i, j) \mayor (i_r, j_r)} \alpha_n(i) \quad \text{for all} \quad n, r \geq 1.
		\end{equation*}
		It follows from \eqref{eq: definition of sigma*} that we can write $\Sigma\left(r_* + 1\right) = \rho + \delta + \varepsilon$ with $\delta, \varepsilon > 0$. We regard $\varepsilon$ as a function of $\delta$, determined by the latter identity, and for each
		\begin{equation*}
		0 < \delta < \min \left\{\Sigma(r_* + 1) - \rho, \alpha_{\min}\right\} \quad \text{and} \quad \theta \in (0, 1)
		\end{equation*}
		we define $\tau(\delta, \theta) \defeq \min \set{t \geq 0}{\left(s_0 - \rho\right)\e^{-\mu t} \leq \theta \varepsilon}$. Explicitly,
		\begin{equation*}
		\tau(\delta, \theta) = \begin{cases}
		\frac{1}{\mu}\left[\log\left(\frac{s_0 - \rho}{\theta \left[\Sigma\left(r_* + 1\right) - \rho - \delta\right]}\right)\right]^+ & \text{if} \quad s_0 > \rho, \\
		0 & \text{if} \quad s_0 \leq \rho.
		\end{cases}
		\end{equation*}
		Thus, the function $\bs$ defined by \eqref{eq: fluid limit of total number of tasks} satisfies
		\begin{equation}
		\label{aux: bound for bs}
		\bs(t) \leq \rho + \theta \varepsilon \quad \text{for all} \quad t \geq \tau(\delta, \theta).
		\end{equation}
		
		Note that $\tau_\bound(s_0)$ is the infimum over $\delta$ and $\theta$ of
		\begin{equation*}
		\tau(\delta, \theta) + \delta + \frac{J - j_{r_*}}{\mu \theta} \left(r_\bound - r_*\right).
		\end{equation*}
		We fix $\delta$ and $\theta$ such that the above expression is smaller than $\tau$ and we choose $n_0$ such that the following properties hold for all $n \geq n_0$.
		\begin{enumerate}	
			\item[(i)] $\br_n(t) \leq R_T$ for all $t \in [0, T]$.
			
			\item[(ii)] $\left|\bs_n(t) - \bs(t)\right| \leq (1 - \theta) \varepsilon$ for all $t \in [0, T]$.
			
			\item[(iii)] $1 / n \leq \beta_n \leq \min \set{\alpha_n(i)}{1 \leq i \leq m}$ and $\Sigma_n(r_* + 1) - L(r_* + 1) \beta_n \geq \rho + \varepsilon$, where we used the notation $L(r) \defeq \max \set{\ell_i(r)}{1 \leq i \leq m}$.
		\end{enumerate}
		Properties (i) and (ii) are possible by Lemma \ref{lem: set of nice sample paths} and Proposition \ref{prop: fluid limit of total number of tasks}, respectively, and (iii) can be enforced by assumptions \eqref{ass: size of classes and initial condition} and \eqref{ass: conditions on beta}.
		
		Fix $n \geq n_0$ and note that for each $r > r_*$ and $t \in [\tau(\delta, \theta), T]$ we have 
		\begin{equation*}
		\Sigma_n(r) - L(r)\beta_n \geq \Sigma_n(r_* + 1) - L(r_* + 1)\beta_n \geq \rho + \varepsilon \geq \bs_n(t).
		\end{equation*}
		The first inequality follows from the fact that $\Sigma_n(\scdot) - L(\scdot)\beta_n$ is non-decreasing provided that $\beta_n \leq \min \set{\alpha_n(i)}{1 \leq i \leq m}$, and the last inequality follows from (ii) and \eqref{aux: bound for bs}. We conclude from the above inequalities that
		\begin{equation}
		\label{aux: proof of bound for br}
		\sum_{i = 1}^m \bq_n\left(t, i, \ell_i(r)\right) \leq 1 - \beta_n \leq 1 - \frac{1}{n}\quad \text{for all} \quad t \in \left[\tau(\delta, \theta), T\right],
		\end{equation}
		since otherwise we would have $\bs_n(t) > \Sigma_n(r) - L(r)\beta_n$. It follows from \eqref{aux: proof of bound for br} with $r = r_* + 1$ that the total number of tokens is at least one if $t \in [\tau(\delta, \theta), T]$ and $\br_n(t) \geq r_*$, so we conclude that $\br_n$ does not increase beyond $r_*$ along the interval $\left[\tau(\delta, \theta), T\right]$. Hence, it only remains to prove that the amount of time after $\tau(\delta, \theta)$ until $\br_n \leq r_*$ is upper bounded by
		\begin{equation*}
		\delta + \frac{J - j_{r_*}}{\mu \theta} \left(r_\bound - r_*\right).
		\end{equation*}
		
		It follows from \eqref{aux: proof of bound for br} that the number of green tokens is
		\begin{equation*}
		n - \sum_{i = 1}^m n\bq_n(t, i, \ell_i\left(\br_n(t)\right)) \geq n\beta_n \quad \text{if} \quad \br_n(t) > r_*, \quad t \in \left[\tau(\delta, \theta), T\right] \quad \text{and} \quad n \geq n_0.
		\end{equation*}
		Therefore, $\br_n$ decreases with the next arrival if the latter conditions hold and
		\begin{equation*}
		\bq_n(t, i_{\br_n(t) - 1}, j_{\br_n(t) - 1}) < \alpha_n(i_{\br_n(t) - 1}).
		\end{equation*}
		
		Choose $n_\bound^{\tau, T} \geq n_0$ such that the following conditions hold.
		\begin{enumerate}
			\item[(iv)] $\min \set{\alpha_n(i)}{1 \leq i \leq m} > \alpha(\delta) \defeq \alpha_{\min} - \delta$.
			
			\item[(v)] Let $d \defeq \delta \left[2\left(r_\bound - r_*\right) + 1\right]^{-1}$. The following two properties hold on every subinterval of $[0, T]$ of length at least $d$. First, $\calN_n^\lambda(\scdot)$ has at least $R_T - r_\bound$ jumps. Second, for all $1 \leq i \leq m$ and $1 \leq j \leq J$, the process $\calN_{(i, j)}(n \scdot)$ has at least $n \theta d$ jumps followed by at least one jump of $\calN_n^\lambda(\scdot)$.
			
			\item[(vi)] $\mu n \left[\alpha_n(i) - \alpha(\delta)\right] \theta  d \geq 1$ for all $1 \leq i \leq m$.
		\end{enumerate}
		Properties (iv) and (v) can be enforced by \eqref{ass: size of classes and initial condition} and Proposition \ref{prop: relative compactness - product topology}, respectively.
		
		We may assume without loss of generality that
		\begin{equation*}
		J \geq \floor{\frac{\rho + \varepsilon}{\alpha(\delta)}},
		\end{equation*}
		choosing a smaller $\delta$ at the start of the proof if needed. By the above inequality,
		\begin{equation*}
		\bq_n(t, i, j) \leq \frac{\bs_n(t)}{j} \leq \frac{\rho + \varepsilon}{J + 1} < \alpha(\delta) \leq \alpha_n(i) \quad \text{if} \quad j > J, \quad t \in [\tau(\delta, \theta), T] \quad \text{and} \quad n \geq n_\bound^{\tau, T}.
		\end{equation*}
		Therefore, if $n \geq n_\bound^{\tau, T}$, $t \in [\tau(\delta, \theta), T]$ and $\br_n(t) > r_\bound$, then $\br_n$ decreases with each arrival until it reaches $r_\bound$. We conclude from (i) and (v) that
		\begin{equation*}
		\br_n(t) \leq r_\bound \quad \text{for all} \quad t \in \left[\tau(\delta, \theta) + d, T\right] \quad \text{and} \quad n \geq n_\bound^{\tau, T}.
		\end{equation*}
		
		Suppose now that $\br_n(t) = r \in (r_*, r_\bound]$ for some $t \in [\tau(\delta, \theta) + d, T]$ and $n \geq n_\bound^{\tau, T}$. Then $\br_n$ decreases as soon as a server pool of class $i_{r - 1}$ has a green token and a task arrives. Moreover, since the number of green tokens remains positive, no tasks are sent to server pools of class $i_{r - 1}$ before $\br_n$ decreases. At least $n\left[\alpha_n(i) - \alpha(\delta)\right]$ server pools of class $i_{r - 1}$ have strictly less than $J + 1$ tasks, thus it follows from (v) that the time until $\bq_n(i_{r - 1}, j_{r - 1})$ drops below $\alpha_n(i)$ is at most
		\begin{equation*}
		\ceil{\frac{n(J - j_{r - 1})\left[\alpha_n(i) - \alpha(\delta)\right] + 1}{\mu n \left[\alpha_n(i) - \alpha(\delta)\right]j_{r - 1}\theta  d}}d \leq \frac{J - j_{r - 1}}{\mu \theta} + \frac{1}{\mu n \left[\alpha_n(i) - \alpha(\delta)\right]\theta} + d \leq \frac{J - j_{r_*}}{\mu \theta} + 2d.
		\end{equation*}
		The numerator inside the ceiling function is the maximum number of tasks that need to leave the $n \left[\alpha_n(i) - \alpha(\delta)\right]$ server pools of class $i_{r - 1}$ with strictly fewer than $J + 1$ tasks to ensure that one of these server pools has strictly less than $j_{r - 1}$ tasks, and the quantity $n \left[\alpha_n(i) - \alpha(\delta)\right]j_{r - 1}$ in the denominator is a lower bound for the cumulative number of tasks in these server pools while all the server pools have at least $j_{r - 1}$ tasks. Also, note that the first inequality uses $j_{r - 1} \geq 1$ and the last inequality uses (vi). Since $\br_n$ decreases at most $r_\bound - r_*$ times before it reaches $r_*$, we have
		\begin{equation*}
		\br_n(t) \leq r_* \quad \text{for all} \quad t \in \left[\tau(\delta, \theta) + \delta + \frac{J - j_{r_*}}{\mu \theta}\left(r_\bound - r_*\right), T\right] \quad \text{and} \quad n \geq n_\bound^{\tau, T}.
		\end{equation*}
		This completes the proof.
	\end{proof}
	
	\begin{proof}[Proof of Lemma \ref{lem: lemma 1 for limit of threshold}.]
		We omit $\omega$ from the notation. It follows from \eqref{seq: definition of q for threshold} that
		\begin{align*}
		\sum_{(i, j) \mayor (i_r, j_r)} \bq_k(t, i, j) - \sum_{(i, j) \mayor (i_r, j_r)} \bq_k\left(\tau_{k, 1}, i, j\right) &= \sum_{(i, j) \mayor (i_r, j_r)} \left[\calA_k(t, i, j) - \calA_k\left(\tau_{k, 1}, i, j\right)\right] \\
		&- \sum_{(i, j) \mayor (i_r, j_r)} \left[\calD_k(t, i, j) - \calD_k\left(\tau_{k, 1}, i, j\right)\right].
		\end{align*}
		Note that all tasks are assigned to server pools with coordinates $(i, j) \mayor (i_r, j_r)$ during the interval $(\tau_{k, 1}, t]$, so the first term on the right-hand side can be expressed as follows:
		\begin{equation*}
		\sum_{(i, j) \mayor (i_r, j_r)} \left[\calA_k(t, i, j) - \calA_k\left(\tau_{k, 1}, i, j\right)\right] = \frac{1}{k}\left[\calN_k^\lambda(t) - \calN_k^\lambda\left(\tau_{k, 1}\right)\right].
		\end{equation*}
		
		Using the process $\delta_k(r)$ defined in \eqref{eq: error process}, we may write
		\begin{align*}
		\sum_{(i, j) \mayor (i_r, j_r)} \bq_k(t, i, j) &- \sum_{(i, j) \mayor (i_r, j_r)} \bq_k\left(\tau_{k, 1}, i, j\right) \\
		&= \left(t - \tau_{k, 1}\right) \lambda - \sum_{(i, j) \mayor (i_r, j_r)} \int_{\tau_{k, 1}}^t \mu j \left[\bq_k(s, i, j) - \bq_k(s, i, j + 1)\right]ds \\
		&+ \delta_k(t, r) - \delta_k(\tau_{k, 1}, r) \\
		&\geq \left(t - \tau_{k, 1}\right) \left[\lambda - \mu \sum_{(i, j) \mayor (i_r, j_r)} \alpha_k(i)\right] - 2\sup_{s \in [0, T]} \left|\delta_k(s, r)\right|.
		\end{align*}
		For the last inequality, observe that
		\begin{align*}
		\sum_{(i, j) \mayor (i_r, j_r)} j\left[\bq_k(i, j) - \bq_k(i, j + 1)\right] &= \sum_{i = 1}^m \sum_{j = 1}^{\ell_i(r)} j\left[\bq_k(i, j) - \bq_k(i, j + 1)\right] \\
		&= \sum_{i = 1}^m \left[\sum_{j = 1}^{\ell_i(r)} \bq_k(i, j) - \ell_i(r)\bq_k(i, \ell_i(r) + 1)\right] \\
		&\leq \sum_{(i, j) \mayor (i_r, j_r)} \alpha_k(i) \quad \text{for all} \quad k \in \calK.
		\end{align*}
		This completes the proof.
	\end{proof}
	
	\begin{proof}[Proof of Lemma \ref{lem: lemma 2 for the limit of threhsold}.]
		We omit $\omega$ from the notation and we assume that $r > 1$, otherwise the statement holds trivially. Fix an arbitrary $\gamma \in [0, 1/2)$ and let
		\begin{align*}
		&\xi_{k, 2} \defeq \inf \set{t \geq \zeta_{k, 1}}{\sum_{(i, j) \mayor (i_r, j_r)} \bq_k(t, i, j) \leq \sum_{(i, j) \mayor (i_r, j_r)} \alpha_k(i) - \min\left\{\beta_k, k^{-\gamma}\right\}}, \\
		&\xi_{k, 1} \defeq \sup \set{t \leq \xi_{k, 2}}{\sum_{(i, j) \mayor (i_r, j_r)} \bq_k(t, i, j) = \sum_{(i, j) \mayor (i_r, j_r)} \alpha_k(i)}.
		\end{align*}
		By assumption, $\br_k \leq r$ along the interval $\left[\zeta_{k, 1}, \zeta_{k, 2}\right]$ and $\br_k(\zeta_{k, 1}) = r$, since otherwise the number of tokens at time $\zeta_{k, 1}$ would be zero, and this cannot happen by Remark \ref{rem: dynamics of threshold}. Moreover, the number of green tokens is upper bounded by
		\begin{equation*}
		\sum_{(i, j) \mayor (i_r, j_r)} k\alpha_k(i) - \sum_{(i, j) \mayor (i_r, j_r)} k\bq_k(i, j) < k\beta_k
		\end{equation*}
		along the interval $\left[\zeta_{k, 1}, \xi_{k, 2}\right)$, which implies that $\br_k$ cannot decrease along $\left[\zeta_{k, 1}, \xi_{k, 2}\right]$. We conclude that $\br_k(t) = r$ for all $t \in \left[\zeta_{k, 1}, \zeta_{k, 2} \wedge \xi_{k, 2}\right]$ and $k \in \calK$. Moreover,
		\begin{equation*}
		\max \set{\alpha_k(i) - \bq_k(i, j)}{(i, j) \mayor (i_r, j_r)} \leq \sum_{(i, j) \mayor (i_r, j_r)} \alpha_k(i) - \sum_{(i, j) \mayor (i_r, j_r)} \bq_k(i, j) \leq k^{-\gamma} 
		\end{equation*}
		along $\left[\zeta_{k, 1}, \xi_{k, 2}\right]$. Thus, it suffices to prove that $\xi_{k, 2} \geq \zeta_{k, 2}$ for all large enough $k \in \calK$.
		
		In order to prove that $\xi_{k, 2} \geq \zeta_{k, 2}$, we show that
		\begin{equation}
		\label{aux: statement to prove lemma}
		\begin{split}
		&\sum_{(i, j) \mayor (i_r, j_r)} \alpha_k(i) - \sum_{(i, j) \mayor (i_r, j_r)} \bq_k(t, i, j) =\\
		&\sum_{(i, j) \mayor (i_r, j_r)} \bq_k\left(\xi_{k, 1}^-, i, j\right) - \sum_{(i, j) \mayor (i_r, j_r)} \bq_k(t, i, j) < \min\left\{\beta_k, k^{-\gamma}\right\}
		\end{split}
		\end{equation}
		for all $t \in \left[\xi_{k, 1}, \zeta_{k, 2} \wedge \xi_{k, 2}\right]$ and all large enough $k \in \calK$. Note that, by definition of $\xi_{k, 2}$, the latter statement implies that $\xi_{k, 2} > \zeta_{k, 2}$.
		
		Let $\tau_{k, 1} \defeq \zeta_{k, 2} \wedge \xi_{k, 1}$ and $\tau_{k, 2} \defeq \zeta_{k, 2} \wedge \xi_{k, 2}$. In addition, let $M_r$ be the cardinality of the finite set $\set{(i, j) \in \calI_+}{(i, j) \mayor (i_r, j_r)}$. For all $t \in \left[\tau_{k, 1}, \tau_{k, 2}\right]$, we have
		\begin{align*}
		\sum_{(i, j) \mayor (i_r, j_r)} \bq_k(t, i, j) &- \sum_{(i, j) \mayor (i_r, j_r)} \bq_k\left(\tau_{k, 1}^-, i, j\right) \\
		&\geq \sum_{(i, j) \mayor (i_r, j_r)} \bq_k(t, i, j) - \sum_{(i, j) \mayor (i_r, j_r)} \bq_k\left(\tau_{k, 1}, i, j\right) - \frac{M_r}{k} \\
		&\geq \left(t - \tau_{k, 1}\right) \left[\lambda - \mu \sum_{(i, j) \mayor (i_r, j_r)} \alpha_k(i)\right] - 2\sup_{s \in [0, T]} \left|\delta_k(s, r)\right| - \frac{M_r}{k} \\
		&\geq - 2\sup_{s \in [0, T]} \left|\delta_k(s, r)\right| - \frac{M_r}{k}.
		\end{align*}
		For the first inequality, note that the processes $\bq_k(i, j)$ have jumps of size $1 / k$. The second inequality follows from Lemma \ref{lem: lemma 1 for limit of threshold}, since
		\begin{equation*}
		\br_k(t) = r \quad \text{and} \quad \sum_{(i, j) \mayor (i_r, j_r)} \bq_k(t, i, j) < \sum_{(i, j) \mayor (i_r, j_r)} \alpha_k(i)
		\end{equation*}
		for all $t \in \left[\tau_{k, 1}, \tau_{k, 2}\right)$ and $k \in \calK$. The third inequality follows from \eqref{eq: definition of sigma*} and $r \leq r_*$.
		
		We conclude from \eqref{ass: conditions on beta} and \eqref{eq: error process goes to zero} that
		\begin{equation*}
		\lim_{k \to \infty} \frac{1}{\beta_k}\left[2\sup_{s \in [0, T]} |\delta_k(s, r)| + \frac{M_r}{k}\right] = \lim_{k \to \infty} \frac{1}{k^{\gamma_0}\beta_k}\left[2\sup_{s \in [0, T]} k^{\gamma_0}|\delta_k(s, r)| + \frac{M_r}{k^{1 - {\gamma_0}}}\right] = 0.
		\end{equation*}
		If $\beta_k$ is replaced by $k^{-\gamma}$ on the left-hand side, then the limit is also equal to zero by~\eqref{eq: error process goes to zero}. Consequently, for all sufficiently large $k \in \calK$, we have
		\begin{align*}
		\sum_{(i, j) \mayor (i_r, j_r)} \bq_k(t, i, j) - \sum_{(i, j) \mayor (i_r, j_r)} \bq_k\left(\tau_{k, 1}^-, i, j\right) &\geq - 2\sup_{s \in [0, T]} \left|\delta_k(s, r)\right| - \frac{M_r}{k} \\
		&> - \min\left\{\beta_k, k^{-\gamma}\right\}
		\end{align*}
		for all $t \in \left[\tau_{k, 1}, \tau_{k, 2}\right]$. This implies that \eqref{aux: statement to prove lemma} holds for all sufficiently large $k \in \calK$; observe that \eqref{aux: statement to prove lemma} holds trivially when $\xi_{k, 1} > \zeta_{k, 2}$.
	\end{proof}
	
\end{appendices}
	
%% References

\newcommand{\noop}[1]{}

\bibliographystyle{IEEEtranS}
\bibliography{bibliography}
	
\end{document}